\documentclass{amsart}

\usepackage{color}
\usepackage{amsmath} 
\usepackage{amssymb}
\usepackage{mathrsfs}

\newtheorem{theorem}{Theorem}[section] 
\newtheorem{claim}[theorem]{Claim}

\newtheorem{conclusion}[theorem]{Conclusion}

\theoremstyle{definition}
\newtheorem{definition}[theorem]{Definition}
\newtheorem{dc}[theorem]{Definition/Claim}
\newtheorem{example}[theorem]{Example}

\newtheorem{observation}[theorem]{Observation}

\newtheorem{discussion}[theorem]{Discussion}
\newtheorem{convention}[theorem]{Convention}

\theoremstyle{remark}
\newtheorem{question}[theorem]{Question}
\newtheorem{remark}[theorem]{Remark}
\newtheorem{notation}[theorem]{Notation}

\newcommand{\rest}{{\restriction}}
\newcommand{\dom}{{\rm dom}} 
\newcommand{\exf}{{\rm exf}} 
 
\newcommand{\sfb}{{\rm sfb}} 
 
\newcommand{\fsb}{{\rm fsb}} 
 
\newcommand{\cm}{{\rm cm}} 
\newcommand{\gel}{{\rm gl}} 
\newcommand{\gem}{{\rm gm}} 
\newcommand{\mx}{{\rm mx}} 
 
\newcommand{\usl}{{\rm usl}} 
\newcommand{\ulf}{{\rm ulf}} 
\newcommand{\nsp}{{\rm nsp}} 
\newcommand{\ged}{{\rm gd}} 

\newcommand{\atdf}{{\rm atdf}} 
\newcommand{\Cent}{{\rm Cent}}

\newcommand{\sch}{{\rm sch}}

\newcommand{\Card}{{\rm Card}}

\newcommand{\lf}{{\rm lf}}
\newcommand{\gl}{{\rm gl}}
\newcommand{\qf}{{\rm qf}}
\newcommand{\bK}{{\rm bK}}
\newcommand{\NF}{{\rm NF}}
\newcommand{\exlf}{{\rm exlf}}
\newcommand{\cfl}{{\rm cfl}}

\newcommand{\docl}{{\rm docl}}

\newcommand{\Th}{{\rm Th}}
\newcommand{\tp}{{\rm tp}}
\newcommand{\tr}{{\rm tr}}

\newcommand{\gmm}{{\rm gm}}

\newcommand{\cp}{{\rm cp}}
\newcommand{\at}{{\rm at}}
\newcommand{\ec}{{\rm ec}}
\newcommand{\Ord}{{\rm Ord}}

\newcommand{\deef}{{\rm def}}

\newcommand{\aut}{{\rm aut}}

\newcommand{\excl}{{\rm excl}}
\newcommand{\bs}{{\rm bs}}

\newcommand{\df}{{\rm df}}

\newcommand{\cg}{{\rm cg}}
\newcommand{\ab}{{\rm ab}}
\newcommand{\gr}{{\rm gr}}
\newcommand{\clf}{{\rm clf}}
\newcommand{\id}{{\rm id}}
\newcommand{\olf}{{\rm olf}}
\newcommand{\plf}{{\rm plf}}

\newcommand{\LST}{{\rm LST}}

\newcommand{\Sym}{{\rm Sym}}

\newcommand{\sel}{{\rm sl}}

\newcommand{\fin}{{\rm fin}}
\newcommand{\rfin}{{\rm rfin}}

\newcommand{\Dom}{{\rm Dom}}
\newcommand{\Rang}{{\rm Rang}}

\newcommand{\wilog}{{\rm without loss of generality}}
\newcommand{\Wilog}{{\rm Without loss of generality}}

\newcommand{\then}{{\underline{then}}}
\newcommand{\when}{{\underline{when}}}
\newcommand{\where}{{\underline{where}}}
\newcommand{\Then}{{\underline{Then}}}
\newcommand{\If}{{\underline{if}}}
\newcommand{\Iff}{{\underline{iff}}}
\newcommand{\mn}{{\medskip\noindent}}
\newcommand{\sn}{{\smallskip\noindent}}

\newcommand{\cA}{{\mathscr A}}

\newcommand{\cB}{{\mathscr B}}

\newcommand{\gC}{{\mathfrak C}}

\newcommand{\gs}{{\mathfrak s}}

\newcommand{\cH}{{\mathscr H}}

\newcommand{\bbL}{{\mathbb L}}
\newcommand{\bbN}{{\mathbb N}}

\newcommand{\gS}{{\mathfrak S}}

\newcommand{\gk}{{\mathfrak k}}

\newcommand{\bbQ}{{\mathbb Q}}

\newcommand{\bbR}{{\mathbb R}}

\newcommand{\cS}{{\mathscr S}}
\newcommand{\cT}{{\mathscr T}}
\newcommand{\gt}{{\mathfrak t}}

\newcommand{\cU}{{\mathscr U}}
\newcommand{\cW}{{\mathscr W}}

\newcommand{\cf}{{\rm cf}}

\newcount\skewfactor
\def\mathunderaccent#1#2 {\let\theaccent#1\skewfactor#2
\mathpalette\putaccentunder}
\def\putaccentunder#1#2{\oalign{$#1#2$\crcr\hidewidth
\vbox to.2ex{\hbox{$#1\skew\skewfactor\theaccent{}$}\vss}\hidewidth}}

\newenvironment{PROOF}[2][\proofname.]
 {\begin{proof}[#1]\renewcommand{\qedsymbol}{\textsquare$_{\mbox{#2}}$}}
 {\end{proof}}

\usepackage[hidelinks]{hyperref}
\begin{document}
\makeatletter\def\shfiuwefootnote{\gdef\@thefnmark{}\@footnotetext}\makeatother\shfiuwefootnote{Version 2021-08-26. See \url{https://shelah.logic.at/papers/312/} for possible updates.}

\title {Existentially closed locally finite groups \\
 Sh312}
\author {Saharon Shelah}
\address{Einstein Institute of Mathematics\\
Edmond J. Safra Campus, Givat Ram\\
The Hebrew University of Jerusalem\\
Jerusalem, 91904, Israel\\
 and \\
 Department of Mathematics\\
 Hill Center - Busch Campus \\ 
 Rutgers, The State University of New Jersey \\
 110 Frelinghuysen Road \\
 Piscataway, NJ 08854-8019 USA}
\email{shelah@math.huji.ac.il}
\urladdr{http://shelah.logic.at}
\thanks{Partially supported by the ISF, Israel Science Foundation. \newline
I would like to thank Alice Leonhardt for the beautiful typing. \newline
First Typed - 02/Oct/18}


\subjclass{Primary: 03C55, 03C50, 03C60; Secondary: 03C98, 03C45}

\keywords {model theory, applications of model theory, groups, locally
finite groups, canonical closed extension, endo-rigid group}

\date{June 9, 2017}

\begin{abstract}
We investigate this class of groups originally called 
ulf (universal locally finite groups) of uncountable cardinality.
We prove that for every locally finite group $G$ 
there is a canonical existentially closed extension
of the same cardinality, unique up to isomorphism and
increasing with $G$.
Also we get, e.g. existence of complete members 
(i.e. with no non-inner automorphisms) in many cardinals (provably
in ZFC).  The main point here is having a parallel to 
stability theory in the sense of
investigating definable types though the class is very unstable.
\end{abstract}

\maketitle
\numberwithin{equation}{section}
\setcounter{section}{-1}
\newpage

\section {Introduction} 

\subsection {Background} \
\bigskip

On lf (locally finite) groups and exlf (existentially closed locally
finite) groups, see the book by Kegel-Wehrfritz \cite{KeWe73}; exlf groups
were originally called ulf (= universal locally finite) groups, we
change as the word ``universal" has been used in this context with a 
different meaning, see Definition \ref{z12} and Claim \ref{z2}.

\noindent
Recall
\begin{definition}
\label{x2}  
1) $G$ is a lf \underline{(locally finite) group} \If \, $G$ is a group and 
every finitely generated subgroup is finite. 

\noindent
2) $G$ is an exlf (\underline{existentially closed} lf) group 
(in \cite{KeWe73} it
is called
ulf, universal locally finite group) \when \, $G$ is a locally finite
group and for any finite groups 
$K \subseteq L$ and embedding of $K$ into $G$, the embedding can be extended to
an embedding of $L$ into $G$.

\noindent
3) Let $\mathbf K_{\lf}$ be the class of lf (locally finite)
   groups (partially ordered by $\subseteq$, being a subgroup) and let
$\mathbf K_{\exlf}$ be the class of existentially closed $G \in \mathbf K_{\lf}$.
\end{definition}
\bigskip

\noindent
In particular there is one and only one $\exlf$ group of cardinality
$\aleph_0$.  Hall proved that every lf 
group can be extended to an exlf group, as follows.  It suffices for a
given lf group $G$ to
find $H \supseteq G$ such that if $K \subseteq L$ are finite and $f$
embeds $K$ into $G$, then some $g \supseteq f$ embed $L$ into $H$.  To
get such $H$,  for finite $K
\subseteq G$ let $E_{G,K} = \{(a,b):a,b \in G$ and 
$aK = bK\}$ and let $G^\oplus$ be
the group of permutations $f$ of $G$ such that for some finite $K
\subseteq G$ we have $a \in G \Rightarrow a E_{G,K} f(a)$; now $b \in G$ should
be identified with $f_b \in G^\oplus$ where $f_b$ is defined by
$f_b(x) = xb$ hence $f_b \in G^\oplus$ because if 
$b \in K \subseteq G$ then $a \in G \Rightarrow f_b(a) = ab \in
abK=aK$ and $f_{b_2} \circ f_{b_1}(x) = (xb_1)b_2 = x(b_1 b_2) 
= f_{b_1 b_2}(x)$.  Now
$H=G^\oplus$ is essentially as required.

The proof gives a canonical extension.  This means 
for example that every automorphism of
$G$ can be extended to an automorphism of $G^\oplus$ and, moreover, 
we can do it uniformly so preserving isomorphisms.  
Still we may like to have more; 
(for a given lf infinite group $G$) the extension $G^\oplus$ defined above is 
of cardinality $2^{|G|}$ rather than the minimal
value - $|G| + \aleph_0$ (not to mention having to repeat this
$\omega$ times in order to get an $\exlf$ extension).  Also if $G_1 \subseteq
G_2$ then the connection between $G^\oplus_1$ and $G^\oplus_2$ is not
clear, i.e. failure of ``naturality".
A major point of the present work is a construction of a canonical
existentially closed extension of $G$ which has those two additional
desirable properties, see e.g. \ref{a78}.

Note that in model theoretic terminology the exlf groups are the
$(\mathbf D,\aleph_0)$-homogeneous groups, with $\mathbf D$ 
the set of isomorphism types of
finite groups or more exactly complete
qf (= quantifier free) types of finite tuples 
generating a finite group, see e.g. \cite[\S2]{Sh:88r}.  
We use quantifier free types as we use
embeddings (rather than, e.g. elementary embeddings).  Let $\mathbf D(G)$ 
be the set of $\qf$-complete types of finite sequences from the group $G$.

Let $\mathbf K_{\exlf}$ be the class of exlf groups.
By Grossberg-Shelah \cite{Sh:174}, if 
$\lambda = \lambda^{\aleph_0}$ then no $G \in \mathbf K^{\exlf}_\lambda
:= \{H \in K_{\exf}:|H| = \lambda\}$
is universal in it, i.e., such that every other member is
embeddable into it.  But if $\kappa$ is a compact cardinal and $\lambda >
\kappa$ is strong limit of cofinality $\aleph_0$ then there is a
universal exlf in cardinality $\lambda$, (this is a 
special case of a general theorem).

Wehrfritz asked about the categoricity of the class of exlf groups 
in any $\lambda > \aleph_0$.  This was answered by 
Macintyre-Shelah \cite{Sh:55} which proved that in every $\lambda >
\aleph_0$ there are $2^\lambda$ non-isomorphic members of
$\mathbf K^{\text{exlf}}_\lambda$.  This was disappointing in some sense: in
$\aleph_0$ the class is categorical, so the question was perhaps
motivated by the hope that also general
structures in the class can be understood to some extent.

A natural and frequent question on a class of structures is the
existence of rigid members, i.e. ones with no non-trivial
automorphism.  Now any exlf group $G \in \mathbf K_{\exlf}$ has non-trivial
automorphisms - the inner automorphisms (recalling it has a trivial center).
So the natural question is about complete members where a group is
called complete \Iff \, it has no non-inner automorphism.

Concerning the existence of a complete, locally finite group of
cardinality $\lambda$: Hickin \cite{Hi78} proved one exists 
in $\aleph_1$ (and more, e.g. he finds a family of $2^{\aleph_1}$ 
such groups pairwise far apart, i.e. no
uncountable group is embeddable in two of them).  Thomas \cite{Th86}
assumed G.C.H. and built one in every successor cardinal (and
more, e.g. it has no Abelian or just solvable subgroup of the same
cardinality).  Related are Giorgetta-Shelah \cite{Sh:83},
 Shelah-Zigler \cite{Sh:96}, which investigate $\mathbf K_{G_*}$
 getting similar results where
\mn
\begin{enumerate}
\item[$(*)$]   assume $G_*$ an existentially closed countable group we
  let
\sn
\begin{enumerate}
\item[(a)]  $\mathbf K_{G_*}$ is the class of groups $G$ such that 
every finitely generated subgroup of $G$ is embeddable into $G_*$
\sn
\item[(b)]  $\mathbf K^{\excl}_{G_*}$ is the class of groups $G$ which are
$\bbL_{\infty,\aleph_0}$-equivalent to $G_*$ ($\excl$ stands for
existentially closed); equivalently $G \in \mathbf K_{G_*}$, every
finitely generated subgroup of $G_*$ is embeddable into $G$ and if $\bar
a,\bar b \in {}^n G$ realize the same qf type in $G$ then some inner
automorphism of $G$ maps $\bar a$ to $\bar b$
\end{enumerate}
\sn
\item[$(**)$]   we can replace ``group $G_*$" by any other structure.
\end{enumerate}
\mn
Giorgetta-Shelah \cite{Sh:83} build in cardinality 
continuum $G \in \mathbf K_{\exlf}$ with
no uncountable Abelian subgroup and similarly for $\mathbf
K^{\excl}_{G_*},G_*$ as in $(*)$ and also for the similarly defined
$\mathbf K^{\excl}_{F_*},F_*$ an existentially closed countable 
fixed division ring.  Shelah-Zigler \cite{Sh:96} build, for $G_*$ as
in $(*)$ and $\lambda > \aleph_0;N^\ell_\lambda \in \mathbf
K^{\exlf}_{G_*}$ of cardinality $\lambda$ for $\ell=1,2$ such that
$N^1_\lambda$ has no Abelian group of cardinality $\lambda$ and every
subgroup of cardinality $\lambda$ has a free subgroup of the same
cardinality; moreover, there are $2^\lambda$ pairwise non-isomorphic
$N$ like $N^\ell_\lambda$.

In 1985 the author wrote notes (in Hebrew) for proving that
there are anti-prime constructions and 
complete exlf groups when, e.g., $\lambda =
\mu^+,\mu^{\aleph_0} = \mu$; using black boxes and ``anti-prime"
construction, i.e. using definable types as below; here we exclusively
use qf (quantifier free) types; this was announced in 
\cite[pg.418]{Sh:300}, but the work
was not properly finished.  To do so is our aim here.

Meanwhile Dugas-G\"obel \cite[Th.2]{DgGb93} prove that for $\lambda =
\lambda^{\aleph_0}$ and $G_0 \in \mathbf K^{\lf}_{\le \lambda}$ there is a
complete $G \in \mathbf K^{\exlf}_{\lambda^+}$ extending $G_0$; moreover
$2^{\lambda^+}$ pairwise non-isomorphic ones.  Then
Braun-G\"obel \cite{BrGb03} got better results for complete locally finite
$p$-groups.  Those constructions build an increasing continuous chain
$\langle G_\alpha:\alpha < \lambda^+\rangle$, each $G_\alpha$ of
cardinality $\lambda$, such that $G_{\alpha +1}$ is the wreath product
of $G_\alpha$ and suitable Abelian locally finite groups, $G =
\{G_\alpha:\alpha < \lambda^+\}$ is the desired group.  This gives a
tight control over the group and implies, e.g. that only few
(i.e. $\le \lambda$) members commute with $G_0$.  Here we are
interested in groups $G'$ which are ``more existentially closed",
e.g. ``for every $G' \subseteq G$ of cardinality $<|G|$, there are
$|G|$ elements commuting with it"; such properties are called ``being
full", note that fullness implies that a restriction on the cardinal
is necessary and not so without it, see \ref{p88}.

We show that though the class $\mathbf K_{\text{exlf}}$ is very ``unstable"
there is a large enough set of definable types so we can imitate
stability theory and have reasonable control in building exlf groups, using
quantifier free types.  This may be considered a ``correction" to the
non-structure results discussed above.

In \S1 we present somewhat abstractly our results relying on the existence
of a dense and closed so called $\gS$, a set of schemes of definitions
of the relevant types.
So before we turn to explaining our results we deal with the so called
schemes, needed for explaining them.
\bigskip

\noindent
\subsection {Schemes} \
\bigskip

\noindent
We deal with a class $\mathbf K$ of structures, usually it is the class
of locally finite groups, but some of the results holds for suitable
universal classes, see \S6.

Central here are so-called schemes.  For \underline{models theorists}
they are for a given $G \in \mathbf K_{\lf}$ and finite sequence $\bar a
\subseteq G$ (realizing a suitable quantifier free type) a definition
of a complete (quantifier free) type over $G$ so realized in some
extension of $G$ from $\mathbf K_{\lf}$, which does not split over $\bar a$;
alternatively you may say that they are definitions of a complete-free type
quantifier over $G$ which does not split over $\bar a$ and its restriction.

For \underline{algebraists} they are our replacement of free products
$G_1 *_{G_0} G_2$, but $\mathbf K_{\lf}$ is not closed under free product, in
fact, fail amalgamation.  So we are interested in replacements in the
cases $G_0$ is finite, also we waive symmetry.

\begin{convention}
\label{1.0}
1) $\mathbf K$ a universal class of structures (i.e. all of the same
vocabulary, closed under isomorphisms and $M \in \mathbf K$ iff every
finite generated substructure belongs to $\mathbf K$; usually $\mathbf K =
\mathbf K_{\lf}$).

\noindent
2) $G,H,\ldots \in \mathbf K$.
\end{convention}

\begin{definition}
\label{a2}  
For $H \in \mathbf K,n < \omega$, a set $A \subseteq H$ and 
$\bar a \in {}^n H$ let tp$(\bar a,A,H) = \text{ tp}_{\bs}(\bar
a,A,H)$ be the basic type of $\bar a$ in $H$ over $A$, that is:

\begin{equation*}
\begin{array}{clcr}
\{\varphi(\bar x,\bar b):&\varphi \text{ is a basic (atomic or
negation of atomic) formula in the variables} \\
  &\bar x = \langle x_\ell:\ell < n \rangle \text{ and the parameters }
\bar b, \text{ a finite sequence from } A, \\
  &\text{ which is satisfied by } \bar a \text{ in } H\}.
\end{array}
\end{equation*}

\mn
So if $\mathbf K$ is a class of groups withough loss of generality $\varphi$ is 
$\sigma(\bar x,\bar b)=e$ or 
 $\sigma(\bar x,\bar b) \ne e$ 
for some group-term $\sigma$, a so called ``word",
(for $\mathbf K_{\text{\rm olf}}$ we also have 
$\sigma_1(\bar x,\bar b) < \sigma_2(\bar x,\bar b)$)
\underline{but} we may write $p(\bar y) = \text{ tp}_{\bs}(\bar
b,A,H)$ or $p(\bar z) = \text{ tp}_{\bs}(\bar c,A,H)$ or just
$p$ when the sequence of variables is clear from the context.

\noindent
2) We say $p(\bar x)$ is an $n-bs$-\underline{type} over $G$ \when \, it is a 
set of basic formulas in the
variables $\bar x = \langle x_\ell:\ell < n \rangle$ and
parameters from $G$, such that $p(\bar x)$ is consistent, which means:
if $K \subseteq G$ is f.g. and $q(\bar x)$ is a finite subset of
$p(\bar x)$ and $q(\bar x)$ is over $K$ (i.e. all the
parameters appearing in $q(\bar x)$ are from $K$) \then \, $q(\bar x)$
is realized in some $L \in \mathbf K$ extending $K$.  
We say $\bar a$ realizes $p$ in $H$ if $G
\subseteq H$ and $\varphi(\bar x,\bar b) \in p \Rightarrow H \models
\varphi[\bar a,\bar b]$. 

\noindent
3) $\mathbf S^n_{\bs}(G) = \{\text{tp}_{\bs}(\bar a,G,H):
G \subseteq H,H$ is from $\mathbf K$ and $\bar a \in {}^n H\}$ and
$\mathbf S_{\bs}(G) = \bigcup\limits_{n} \mathbf S_{\bs}(G)$; if $\mathbf K$
is not clear from the context we should write $\mathbf S^n_{\bs}
(G,\mathbf K),\mathbf S_{\bs}(G,\mathbf K)$.
\end{definition}

\begin{observation}
\label{a4}
For every $p \in \mathbf S^n_{\bs}(M)$ and $M \in \mathbf K$ there are
 $N,\bar a$ such that $M \subseteq N \in \mathbf K,\bar a \in {}^n N$
   realizes $p,G_N = c \ell(G_M + \bar a,N)$ \underline{and}
   if $M \subseteq N' \in \mathbf K$ and $\bar a'$ realizes $p$ in $N'$
   then there is $N'' \subseteq N'$ and an isomorphism $f$ from $N$ onto
   $N''$ extending {\rm id}$_M$ such that $f(\bar a) = \bar a'$.
\end{observation}

\begin{remark}
\label{a5}
0) In \ref{a4} we shall later use the convention of \ref{z4}(1),(3).

\noindent
1) We are particularly interested in types which are definable in some
   sense over small sets.

\noindent
2) We can define ``$p \in \mathbf S^n_{\bs}(M)$" syntactically, because
   for a set $p$ of basic formulas $\varphi(\bar x,\bar a),\bar a$
from $M$ which is complete (i.e. if $\varphi(\bar x,\bar a)$ is an
   atomic formula over $M$ then $\varphi(\bar x,\bar a) \in p$ or
   $\neg \varphi(\bar x,\bar a) \in p$), 
 we have $p \in \mathbf S^n_{\bs}(M)$ iff for every
   f.g. $N \subseteq M$ we have $p \rest N := \{\varphi(\bar x,\bar a)
   \in p:\bar a \subseteq N\} \in \mathbf S^n_{\bs}(N)$. 

\noindent
3) Why do we use below types which do not split over a finite subgroup
   and the related set of schemes?  As we like to get a canonical
   extension of $M \in \mathbf K$ it is natural to use a set of types
closed under automorphisms of $M$, and as their number is preferably
   $\le \|M\|$,  it is natural to demand that any such type is, in some sense,
   definable over some finite subset of $M$.
\end{remark}

\noindent
As in \cite{Sh:3}:
\begin{definition}
\label{a6}  
We say that $p = \text{tp}_{\bs}(\bar a,G,H) 
\in {\mathbf S}^n_{\bs}(G)$ \underline{does not split} over $K \subseteq G$ 
\when \,  for every $m< \omega$ and $\bar b_1,\bar b_2 \in
{}^m G$ satisfying tp$_{\bs}(\bar b_1,K,G) = 
\text{ tp}_{\bs}(\bar b_2,K,G)$ we have $\text{tp}_{\bs}
(\bar b_1 \char 94 \bar a,K,H) = \text{ tp}_{\bs}
(\bar b_2 \char 94 \bar a,K,H)$.
\end{definition}

\begin{definition}
\label{a12}
1) Let $\mathbf D(\mathbf K) = \bigcup\limits_{n} \mathbf D_n(\mathbf K)$,
   where $\mathbf D_n(\mathbf K) = \{\text{tp}_{\bs}(\bar a,\emptyset,M):
\bar a \in {}^n M$ and $M \in \mathbf K\}$.
 
\noindent
2) Assume\footnote{This is used to define the set $\gS$ of schemes; 
for this section the case $p(\bar x) = p'(\bar x)$ is enough as we can consider
all the completions but the general version is more natural in
counting a set $\gS$ of schemes and in considering actual examples.}
$p(\bar x)$ is a $k$-type, that is, $\bar x = \langle x_\ell:\ell <
k\rangle$ and for some $p'(\bar x)$ we have 
$p(\bar x) \subseteq p'(\bar x) \in \mathbf D_k(\mathbf K)$ and 
$m < \omega$.  We let $\mathbf D_{p(\bar x),m}(\mathbf K) = 
\mathbf D_m(p(\bar x),\mathbf K)$ be
 the set of $q(\bar x,\bar y) \in \mathbf D_{k+m}(\mathbf K)$ such that
$q(\bar x,\bar y) \supseteq p(\bar x)$, which means that there is $M
   \in \mathbf K$ and $\bar a \in {}^k M$ realizing $p(\bar x)$ and
   $(\bar a,\bar b)$ realizing $q(\bar x,\bar y)$ in $M$, i.e. $\ell
   g(\bar a) = k,\ell g(\bar b) = m$ and $\bar a \char 94 \bar b$
   realizes $q(\bar x,\bar y)$.

\noindent
3) In part (2) let $\mathbf D_{p(\bar x)}(\mathbf K) = 
\cup\{\mathbf D_m(p(\bar x),\mathbf K):m < \omega\}$.
\end{definition}

\begin{remark}
\label{a13}
Below $\gs \in \Omega_{n,k}[\mathbf K]$ is a scheme to fully 
define a type $q(\bar z)
\in \mathbf S^n_{\bs}(M)$ for a given parameter $\bar a \in {}^k M$ such
that $q(\bar z)$ does not split over $\bar a$.  Sometimes $\gs$ is not
unique but if, e.g., $M \in \mathbf K_{\text{exlf}}$ it is.
\end{remark}

\begin{definition}
\label{a14}
1) Let $\Omega[\mathbf K]$ be the set of schemes, i.e.
$\cup\{\Omega_{n,k}[\mathbf K]:k,n< \omega\}$
where $\Omega_{n,k}[\mathbf K]$ is the set of $(k,n)$-schemes 
$\gs$ which means, see below.

\noindent
1A) We say $\gs$ is a $(k,n)$-\underline{scheme} when for some $p(\bar x) =
 p_{\gs}(\bar x_{\gs})$ with $\ell g(\bar x_{\gs}) = k$, (and $k_{\gs} =
k(\gs) = k,n_{\gs} = n(\gs) = n$) we have:
\mn
\begin{enumerate}
\item[$(a)$]  $\gs$ is a function with domain $\mathbf D_{p(\bar x)}(\mathbf K)$
such that for each $m$ it maps $\mathbf D_{p(\bar x),m}(\mathbf K)$ into
$\mathbf D_{k+m+n}(\mathbf K)$
\sn
\item[$(b)$]  if $s(\bar x,\bar y) \in \mathbf D_{p(\bar x),m}(\mathbf K)$ and
$r(\bar x,\bar y,\bar z) = \gs(s(\bar x,\bar y))$ \then \, $r(\bar
x,\bar y,\bar z) \rest (k+m) = s(\bar x,\bar y)$; that is, if $(\bar a,\bar
b,\bar c)$, i.e. $\bar a \char 94 \bar b \char 94 \bar c$, realizes
$r(\bar x,\bar y,\bar z)$ in $M \in \mathbf K$ so $k = \ell g(\bar a),m
= \ell g(\bar b),n = \ell g(\bar c)$, \then \, $\bar a \char 94 \bar
b$ realizes $s(\bar x,\bar y)$ in $M$; see \ref{a15}(1)
\sn
\item[$(c)$]  in clause (b), moreover if 
$\bar b' \in {}^{\omega >}M$, Rang$(\bar b')
\subseteq \text{ Rang}(\bar a \char 94 \bar b)$ \then \, $\bar a \char 94
\bar b' \char 94 \bar c$ realizes the type $\gs(\tp_{\bs}(\bar a \char
94 \bar b',\emptyset,M))$; this is to avoid $\gs$'s which define
contradictory types\footnote{But some $\gs$'s satisfying clauses
  (a),(b) of \ref{a14}(1A) but failing clause (c) this may give a
consistent type in an interesting class of cases.}.
\end{enumerate}
\mn
2) Assume $\gs \in \Omega_{n,k}[\mathbf K]$ and $M \in \mathbf K$ and $\bar a \in
{}^k M$ realizes $p_{\gs}(\bar x_{\gs})$; we let $q_{\gs}(\bar a,M)$ be the
unique $r(\bar z) = r(z_{\gs}) \in \mathbf S^n_{\bs}(M)$ such that for any
$\bar b \in {}^{\omega >}M$ letting $r_{\bar b}(\bar x,\bar y,\bar z) 
:= \gs(\tp_{\bs}(\bar a \char 94 \bar b,\emptyset,M))$ we have 
$r_{\bar b}(\bar a,\bar b,\bar z) \subseteq r(\bar z)$.

\noindent
3) We call $\gs$ \underline{full} 
when $p_{\gs}(\bar x) \in \mathbf D_{k(\gs)}(\mathbf K)$.

\noindent
4) For technical reasons we allow $\bar x_{\gs} = \langle
   x_{\gs,\ell}:\ell \in u\rangle,u \subseteq \bbN,|u| = k_{\gs}$ and
   in this case
${}^{k(\gs)}M$ will mean ${}^u M = \{\langle a_\ell:\ell \in u\rangle:a_\ell
   \in M$ for $\ell \in u\}$ and we do not pedantically distinguish
   between $u$ and $k_{\gs}$.  Similarly for $n_{\gs}$ and $\bar z$, the reason is
   \ref{a14f}, \ref{a21}(4).
\end{definition}

\begin{convention}
\label{a14d}
$\gS$ will denote a subset of $\Omega[\mathbf K]$.
\end{convention}
\bigskip

\noindent
\subsection {The Results} \
\bigskip
In particular (in the so-called first avenue,
see below):
\begin{theorem}
\label{y2}
Let $\lambda$ be any cardinal $\ge |\gS|$.

\noindent
1) For every $G \in \mathbf K^{\lf}_{\le \lambda}$ there
is $H_G \in K^{\exlf}_\lambda$ which is $\lambda$-full over $G$ (hence
over any $G' \subseteq G$; see Definition \ref{a33}) and $\gS$-constructible
over it (see \ref{a37}).

\noindent
2) If $H \in \mathbf K^{\lf}_{< \lambda}$ is
$\lambda$-full over $G( \in \mathbf K^{\lf}_{\le \lambda})$ \then \,
$H_G$ from above can be embedded into $H$ over $G$, see \ref{a41}(4).
\end{theorem}
\bigskip

\noindent
This is proved by \ref{a41} + \S2.
So in some sense $H_G$ is prime over $G$, that is, it is prime but not
among the members of $\mathbf K^{\exlf}_\lambda$, 
i.e. for a different class.  Still 
we would like to have canonicity so uniqueness.  There are some 
additional avenues helpful toward this.

The second avenue tries 
to get results which are nicer by assuming $\gS$ is so called
symmetric which is the parallel of being stable in this context.
Under this assumption we prove the existence of a canonical closure of a
locally finite group to an exlf one.  This is done in \ref{a25}, \ref{a26}.

The third avenue is without assuming ``$\gS$ is symmetric" but using a more
complicated construction,  for which we have similar, somewhat 
weaker results using special linear orders.  The failure of symmetry
seems to draw you to order the relevant pairs $(\gs,\bar a)$ for $G$.
That is, trying to repeat the construction in \ref{a25}(2), without
symmetry we have to well order or at least linearly order $\deef(G) = 
\deef_{\gS}(G)$ which is essentially the set of relevant complete
quantifier types over $G$ over a finite set of parameters, see
Definition \ref{a14f};
this suffices by \ref{a22}(9).  At first glance we have to linearly order
$\deef(G)$, but we take a list of def$(G)$, with each appearing $\lambda$
times and linearly order it such that it does not induce a linear
order of $\deef(G)$.  See below.

\noindent
So we prove (in \ref{a58}, \ref{a62}, \ref{a68})
\begin{theorem}
\label{y8}
1) We can for every $\lf$ group $G$, define $G^{c \ell}$ such that:
\mn
\begin{enumerate}
\item[$(a)$]  if $G \in K^{\lf}_{\le\lambda}$ then $G \subseteq G^{c \ell}
  \in K^{\exlf}_\lambda$
\sn
\item[$(b)$]  $G^{c \ell}$ is unique up to isomorphism over $G$.
\end{enumerate}
\mn
2) Also\footnote{See on this in \ref{a76}.} essentially it commutes
with extensions, i.e. $G_1 \subseteq G_2 \Rightarrow G^{c \ell}_1
\subseteq G^{c \ell}_1$, pedantically
\mn
\begin{enumerate}
\item[$(c)$]  if $G_1 \subseteq G_2$ and $G^{c \ell}_\ell$ is as above
  \then \, there is an embedding $h$ of $G^{c \ell}_1$ into $G^{c
    \ell}_2$ such that $h(G^{c \ell}_1) \cap G_2 = G_1$
\sn
\item[$(c)'$]  restricting ourselves to $\{G \in \mathbf K_{\lf}$: every
$x \in G$ is a singleton$\}$ we have:
\sn
\item[$(b)''$]  $G^{c \ell}$ is really unique
\sn
\item[$(c)''$]   $G_1 \subseteq G_2 \Rightarrow G^{c \ell}_1 \subseteq
  G^{c \ell}_2$.
\end{enumerate}
\end{theorem}

To stress the generality in addition to the class
$\mathbf K_{\text{lf}}$ of lf-groups we use $\mathbf K_{\text{olf}}$, the
class of ordered locally finite groups (see \ref{z4}); for 
them the proof of the existence of a suitable $\gS$ is easier.
Naturally for $\mathbf K_{\text{olf}}$ we certainly
do not have a symmetric $\gS$.

In \S2 we show that $\gS$ as needed in \S1 exists, but not necessarily
symmetric and define and investigate some specific schemes used
later; also we define and investigate $\NF$, a relative of free
amalgamation.  In \S3 we find a fourth avenue which is more specific to the
class of lf groups.  We show that we can induce symmetry, i.e. define
symmetric constructions even for non-symmetric $\gS$ hence get
somewhat better results, see \ref{a78}.  
In particular we construct reasonable closures.

In \S4(A), we show that we can find amalgamation preserving
commuting and so can get a new relative NF$^3$ of NF.  In \S4(B) we deal
with some related schemes (of types).  In \S4(C) we deal with
types with infinitely many variables.

In \S5 we prove the existence of a complete group $G_* \in  \mathbf
K^{\text{exlf}}_\lambda$ when $\lambda = \mu^+,\mu = \mu^{\aleph_0}$.
Moreover, we prove the existence of a complete extension $G_* \in \mathbf
K^{\text{exlf}}_\lambda$ of an arbitrary $G \in 
\mathbf K^{\text{lf}}_{\le \mu}$.

Some of the definitions and claims work also in quite a general
framework, but it is not clear at present how interesting this is.  
Still we consider some expansions of $\mathbf K_{\text{lf}}$, and 
comment on them in \S6. 

We here also consider the partial order $\le_{\gS}$ on $\mathbf K$, where $G_1
\le_{\gS} G_2$ means that every finite $\bar a \subseteq G_2$ realizes over
$G_2$ a type from $\deef_{\gS}(G_1)$.
Note that on $(\mathbf K,\le_{\gS})$ we may generalize stability theory,
in particular when $\gS$ is symmetric (see \S1) or when we use the
symmetrized version (see \S3).  In particular, we can investigate
orthogonality, parallelism, super-stability, and indiscernible sets
which $\Delta$-converge (\cite{Sh:300} or \cite{Sh:300a}).  A class
somewhat similar to $\mathbf K_{\lf}$, for an existentially
closed countable group $L$ is $\mathbf K_L$, the class of groups $G$
such that every f.g. subgroup is embeddable into $L$.  We further
investigate $\mathbf K_{\lf}$ in \cite{Sh:1098} and in more general
direction in a work in preparation with G. Paolini.

We thank Omer Zilberboim and Gianluca Paolini 
for some help in the proofs of this paper and
a referee for many useful comments to clarify this paper.
\bigskip

\noindent
\subsection {Preliminaries}\
\bigskip

\begin{definition}
\label{z1}
1) Let $\mathbf K^{\lf}_\lambda$ be the class of $G \in \mathbf K_{\lf}$ 
of cardinality $\lambda$, let $\mathbf K^{\exlf}_\lambda$ be the 
class of $G \in \mathbf K_{\exlf}$ of cardinality $\lambda$; see
Definition \ref{x2}.

\noindent
2) Let $\fsb(M)$ be the set of f.g. (finitely generated) sub-structures
of $M$.  
\end{definition}

\noindent
Note that $\mathbf K_{\exlf}$ is the same $\mathbf K_{\ulf}$ as defined by
Hall as proved in Macintyre-Shelah \cite{Sh:55}, Wood \cite{Wd72};
that is:
\begin{claim}
\label{z2}
The following conditions on a locally finite group $G$ are equivalent:
\mn
\begin{enumerate}
\item[$(A)$]  $G$ is $\ulf$ which means:
\sn
\begin{enumerate}
\item[$(a)$]  every finite group is embeddable into $G$
\sn
\item[$(b)$]  if $H_1,H_2$ are isomorphic finite subgroups of $G$,
  \then \, for some $x \in G$, conjugation by $x$ maps $H_2$ onto
  (here equivalently into) $H_2$, i.e. $x^{-1} H_1x = H_2$
\end{enumerate}
\sn
\item[$(B)$]  $G \in \mathbf K_{\exlf}$.
\end{enumerate}
\end{claim}

\begin{PROOF}{\ref{z2}}
\underline{$(B) \Rightarrow (A)$}

\underline{Clause (A)(a)}:  let $H$ be a finite group, let $H_1 = \{e_H\}
\subseteq H$ so a sub-group of $H$ 
and let $H_2 = H$ and let $h_1:H_1 \rightarrow G$
be defined by $h_2(e_H) = e_G$.  So by clause (B) there is an extension $h_2$
of $h_1$ embedding $H_2=H$ into $G$, so $h_2(H)$ is as required.
\medskip

\noindent
\underline{Clause (A)(b)}:  let $H_1,H_2 \subseteq G$ be finite
sub-groups and let $H_3 \subseteq G$ be the finite subgroup which $H_1 \cup
H_2$ generates.  There is a finite group $H_4$ extending $H_3$ such
that: any partial automorphism of $H_3$ is included in some
conjugation in $H_4$.  Let $h_3:H_3 \rightarrow H_3 \subseteq G$ be the
identity, hence by Clause (B) recalling $G \in \mathbf K_{\exlf}$, 
there is an embedding $h_4$ of $H_4$ into $G$ extending $h_3$.

So in $h_4(H_4) \subseteq G$ there is a conjugation as required.
\medskip

\noindent
\underline{$(A) \Rightarrow (B)$}:

Let $H_1 \subseteq H_2$ be finite groups and $h_1$ be an embedding of
$H_1$ into $G$.  Let $H_4 \supseteq H_2$ be a finite group such that
any automorphism of $H_1$ is included in an inner automorphism of
$H_1$.  By Clause (A)(a) there is an embedding $h_4$ of $H_4$ into
$G$.  By Clause (A)(b) there is $x \in G$ such that $H'_4 := x^{-1}
h_1(H_1) x \subseteq G$ is equal to $h_4(H_1)$.

Recalling \ref{z22}(7) 
$h'_4 = (\square_x \rest h_4(H_4)) \circ h_4$ embeds $H_4$ into $G$
and maps $H_1$ onto $h_1(H_1)$; but the embedding $h'_4$ does not
necessarily extend $h_1$.  However, by clause (A)(b), 
for some $y \in h'_4(H_4),\square_y h'_4$ embeds $H_4$ 
(hence $H_2$) and extends $h_1$ as required.
\end{PROOF}

\noindent
We may use the class $\mathbf K_{\olf}$ of linearly ordered lf groups,
it is closely related and some issues are more transparent for it;
$\mathbf K_{\olf}$ is defined as follows.
\begin{definition}
\label{z4}
1) Let $\mathbf K_{\text{\rm olf}}$ be the class of structures $M$ which
   are an expansion of a lf group $G = G_M$ by a linear order
$<_M$, also this class is partially ordered by $M_1 \subseteq M_2,M_1$ a
   sub-structure of $M_2$.

\noindent
2) We say that $M \in \mathbf K_{\olf}$ is \underline{existentially closed} as
in \ref{z1}(2) and define $\mathbf K^{\text{olf}}_\lambda$ as in \ref{x2}(2). 

\noindent
3) If $M \in \mathbf K_{\lf}$ then we let $G_M=M$. 
\end{definition}

\begin{remark}
\label{z5}
For $\mathbf K_{\lf}$ conceivably there is a symmetric dense
$\gS$, hence a very natural canonical exlf-closure.  Without it we
can either use a somewhat less natural one (using linear orders, 
see end of \S1) or ``make it
symmetric by brute force" (see \S3).  But for the class 
$\mathbf K_{\olf}$ we can use only the linear orders, so
 every $M$ has a canonical existentially closed extension,
but it is more difficult to make it 
unique up to isomorphism.  We shall in \ref{n7} introduce 
another class, $\mathbf K_{\text{\rm clf}}$, locally finite
 groups with choice.
\end{remark}

\begin{convention}
\label{z8}
1)  Except in \S6, $\mathbf K$ is the class 
$\mathbf K_{\text{lf}}$ of locally finite
groups or $\mathbf K_{\text{olf}}$ of ordered locally finite groups
(we may use $\le_{\mathbf K}$ but here $\mathbf K$ is partially ordered by 
$\subseteq$, being a substructure) and see \ref{z5}.

\noindent
2) Let xlf-group mean a member of $\mathbf K$.  Let $\mathbf K_{\ec}$ be
   the class of existentially closed members of $\mathbf K$.

\noindent
3) In \S2, \S3, \S4, \S5 we use only $\mathbf K_{\lf}$; in \S1 you can 
restrict yourself to $\mathbf K = 
\mathbf K_{\text{lf}}$ but in \S6 we have further cases on which we comment.
\end{convention}

\noindent
The following definition is for the more general framework.
\begin{definition}
\label{z6}
1) For $M,N \in \mathbf K$ let $M \le_{\text{fsb}} N$ mean that if $K
\subseteq L$ are f.g., $K \subseteq M,L \subseteq N$, \then \, there
   is an embedding of $L$ into $M$ over $K$.

\noindent
2) For $M,N \in \mathbf K$ let $M \le_{\Sigma_1} N$ \underline{means}
that $M \subseteq N$ and if $\bar a \in {}^{\ell g(\bar y)}M,
\bar b \in {}^{\ell g(\bar x)}N$ and
   $\varphi(\bar x,\bar y) \in \bbL(\tau_{\mathbf K})$ is quantifier
   free and $N \models \varphi[\bar b,\bar a]$ \ then \, for some
$\bar b' \in {}^{\ell g(\bar x)}M$ we have $M \models 
\varphi[\bar b',\bar a]$. 

\noindent
3) Let $M_\ell \in \mathbf K,\bar a_\ell \in {}^{n(\ell)}(M_\ell)$ for
$\ell=1,2$.  We say that a relation on $M_1 \times M_2$ is
quantifier-free definable in $(M_1,\bar a_1,M_2,\bar a_2)$ 
\when \, it is a Boolean combination of finitely many
simple ones, where $R$ is a simple $n$-place relation
on $M_1 \times M_2$ \when \, $R$ is the set of $n$-tuples
$((b_0,c_0),\dotsc,(b_{n-1},c_{n-1}))$ such that $b_i \in M_1,c_i \in
M_2$ for $i < n$ and

\[
M_1 \models \varphi_1[b_0,\dotsc,b_{n-1},\bar a_1]
\]

\[
M_2 \models \varphi_2[c_0,\dotsc,c_{n-1},\bar a_2]
\]
\mn
for some quantifier-free formulas $\varphi_1,\varphi_2$ in
$\bbL(\tau_{\mathbf K})$ and finite sequences $\bar a_1,\bar a_2$ from
$M_1,M_2$ respectively.
\end{definition}

\begin{remark}
1) Note that \ref{z6}(3) is not actually used, but just indicate the form of
definability used.

\noindent
2) Note that $\le_{\Sigma_1}$ for $\mathbf K_{\text{lf}}$ and
$\mathbf K_{\text{olf}}$ is the same as $\le_{\text{fsb}}$.  For other
classes, see \S6, if the vocabulary is finite and we deal with
locally finite structures they are still the same.  Otherwise, by
our choice of ``does not split" we have to use $\le_{\text{fsb}}$.
But if we prefer to use $\le_{\Sigma_1}$ we have to strengthen the
definition of ``does not split" to make the proof of \ref{a23}(1) work.
\end{remark}

\begin{convention}
\label{z9}
Let $M_1,M_2 \in \mathbf K,M_1 \subseteq M_2$ and 
$\bar a \in {}^n(M_2)$, so $\bar a = (a_0,a_1,a_2,\dotsc,a_{n-1})$.

\noindent
1) Denote by $cl(M_1 + \bar a,M_2)$ the sub-structure generated by 
$M_1 \cup \bar a = M_1 \cup
\{a_0,a_1,\dotsc,a_{n-1}\}$ in $M_2$.

\noindent
2) For a group $G$ and $A \subseteq G$ let
\mn
\begin{enumerate}
\item[$\bullet$]   $\mathbf C_G(A) = \{g \in G:G \models 
``ag = ga"$ for every $a \in A\}$
\sn
\item[$\bullet$]   $\mathbf Z(G) = \mathbf C_G(G)$
\sn
\item[$\bullet$]  $\mathbf N_G(A) = \{c \in G:c^{-1} Ac = A\}$.
\end{enumerate}
\mn
4) For a group $G$, aut$(G)$ is the group of automorphisms of $G$ and
   inner$(G)$ is the normal subgroup of aut$(G)$ consisting of the
inner automorphisms of $G$.
\end{convention}

\noindent
A side issue here is:
\begin{definition}
\label{z12}
1) For a class $\mathbf K$ of structures (of a fixed vocabulary) we say
   $M \in \mathbf K$ is $\lambda$-\underline{universal} in 
$\mathbf K$ \when \, every
   $N \in \mathbf K$ of cardinality $\lambda$ can be embedded into it.

\noindent
2) We say $M \in \mathbf K$ is $(\le \lambda)$-\underline{universal} in
$\mathbf K$ \when \, every $N \in \mathbf K$ of cardinality $\le \lambda$
can be embedded into $M$.

\noindent
3) We say $M \in \mathbf K$ is \underline{universal} \when \, it is
$\lambda$-universal for $\lambda =\|M\|$.

\noindent
4) Assume $\gk = (K_{\gk},\le_{\gk}),K_{\gk}$ as a class of
   $\tau$-structures (for some vocabulary $\tau = \tau_{\gk}$), 
closed under isomorphism, and $\le_{\gk}$ a partial order on $K_{\gk}$
   preserved under isomorphisms.  Above ``$M \in K_{\gk}$ is
   $\lambda$-universal in $\gk$" means that if $N \in K_{\gk}$ has
   cardinality $\lambda$ then there is a $\le_{\gk}$-embedding $f$ of
   $N$ into $M$, i.e. $f$ is an isomorphism from $N$ onto some $N'
   \le_{\gk} M$.  Similarly in the other variants.
\end{definition}

\noindent
The problem of the existence of universal members of $\mathbf K^{\lf}_\lambda$
is connected to
\begin{question}
\label{z19}
Fixing $\kappa$ and an ideal $J$ on $\kappa$, what is
$\lambda_{\mu,\kappa}(J,\mathbf K)$, which is the minimal cardinal (or
$\infty$) $\lambda$ which is $> \mu$ and there is no sequence
$\langle (G_\alpha,\bar a_\alpha):\alpha < \lambda\rangle$ such that
$G_\alpha \in \mathbf K_{\le \mu},\bar a_\alpha \in
{}^\kappa(G_\alpha)$ and there are no $H \in \mathbf K$ and
$\alpha < \beta < \lambda$ and embeddings $f_1,f_2$ of $G_\alpha,G_\beta$
respectively into $H$ such that $\{i < \kappa:f(a_{\alpha,i}) \ne
a_{\beta,i}\} \in J$.
\end{question}

\begin{notation}
\label{z22}
1) Let $G,H,K$ denote members of $\mathbf K$.

\noindent
2) Let $p,q,r$ and $s$ denote types.

\noindent
3) $\gs$ denotes a scheme of defining types, here $\qf$.

\noindent
4) $t$ denotes a member of some $\deef(G)$, i.e. a pair $(\gs,\bar a)$ 
which defines a type in $\mathbf S^{n(\gs)}_{\bs}(G)$.

\noindent
5) For $A \subseteq M$ let $c \ell(A,M) = \langle A \rangle_M$ be
the closure of the set $A$ under the functions of $M$, i.e. the
sub-structure of $M$ which $A$ generates when $M$ is, as usual, a group.

\noindent
6) We may write, e.g. $A+B,A+\bar a,\sum\limits_{i < \alpha} \bar a_i$
instead of $A \cup B,A \cup \Rang(\bar a),\bigcup\limits_{i < \alpha}
\Rang(\bar a_i)$. 

\noindent
7) For a group $G$ and $x \in G$ let $\square_x$ be conjugation by
$x$, that is, the mapping $y \mapsto x^{-1} y x$ for $y \in G$.
\end{notation}
\newpage

\section {Definable types} 

What is accomplished in \S1 and under what assumptions?  We have to
assume that there are dense $\gS \in \Omega[\mathbf K]$ to get existentially
closed $H$ (see \S2).  
Still there are $\gS$'s and any $\gS$ can be extended to a
closed one, preserving density.  For any $\gS$, the partial order
$\le_{\gS}$ on $\mathbf K$ is quite reasonable: not fully so called
a.e.c. still close enough.  In $(\mathbf K,\le_{\gS})$ for regular
$\lambda$ we can find over any $G \in K_{\le \lambda}$ a prime $H$
among the $H \in \mathbf K_\lambda$ extending $G$ 
which are so-called $(\lambda,\gS)$-full over it, see
\ref{a41}.  Also we can find such $H$ quite definable in three ways.
First avenue is to allow a well order.
Second avenue is to assume $\gS$ is symmetric, then $H$ is canonical 
and commutes with extensions (\ref{a26}, \ref{a34}, \ref{a41},
\ref{a45}).  Third avenue relies on linear order.  We still get
uniqueness, but rely on linear ordering of $\deef(G)$ and the
commutation with extension is problematic.  However, we may use pair
$(I,E),I$ a linear order, $E$ an equivalence relation on $I$ and ``dedicate"
each equivalence class to some $t \in \deef(G)$, so can avoid linearly
ordering $\deef(G)$, see \ref{a58}, \ref{a68}; see more in \S3.
\bigskip

\subsection {The Framework}\
\bigskip

\begin{definition}
\label{a14f}
1) For $G \in \mathbf K$ let def$(G)$ be the set of pairs $t = (\gs,\bar
a) = (\gs_t,\bar a_t)$ such that $\gs \in \Omega[\mathbf K]$ 
and $\bar a \in {}^{\omega >}G$ 
realizes $p_{\gs}(\bar x_{\gs})$ and let $q_t(G) = q_{\gs_t}(\bar
a_t,G)$ and $p_t(\bar x_t) = p_{\gs}(\bar x_{\gs}),k(t) = k(\gs),n(t)
= n(\gs)$.  

\noindent
2)  We say $\gs_1,\gs_2$ are disjoint when $\bar x_{\gs_1},\bar
x_{\gs_2}$ are \underline{disjoint} as well as $\bar z_{\gs_1},\bar
z_{\gs_2}$ recalling \ref{a14}(4).
Similarly for $t_1,t_2 \in \text{ def}(G)$.

\noindent
3) We say $\gs_1,\gs_2$ are \underline{congruent}, 
written $\gs_1 \equiv \gs_2$ \when
 \,  we get $\gs_2$ from
$\gs_1$ by replacing $\bar x_{\gs_1},\bar z_{\gs_1}$ by other
sequences of variables, $\bar x_{\gs_2},\bar z_{\gs_2}$ 
(again with no repetitions, of the same length respectively, of course).
Similarly for $t_1,t_2 \in \deef(G)$ (the aim is to be able to get 
\underline{disjoint} congruent copies; we do not always
remember to replace a scheme by some congruent copy).

\noindent
4) We say $\gS$ is \underline{invariant} \when \,: 
if $\gs_1,\gs_2 \in \Omega[\mathbf K]$
are congruent then $\gs_1 \in \gS \Leftrightarrow \gs_2 \in \gS$.

\noindent
5) The invariant closure of $\gS$ is defined naturally.  Let $|\gS|$ 
mean its cardinality up to congruency, that is, $|\gS/\equiv|$; 
if not said otherwise we use invariant $\gS$.

\noindent
6) We define the (equivalence) relation $\approx_G$ on def$(G)$ by
   $t_1 \approx_G t_2$ \underline{iff} $t_1,t_2 \in \text{ def}(G)$
   and $q_{t_1}(G) = q_{t_2}(G)$.
\end{definition}

\begin{claim}
\label{a15}
1) If $\gs \in \Omega_{n,k}[\mathbf K]$ and $G \in \mathbf K,\bar a \in {}^k M$
\then \, indeed $q_{\gs}(\bar a,G) \in \mathbf S^n_{\bs}(G)$
so exist and is unique and does not split over $\bar a$, see 
Definition \ref{a14}(2); if $\bar a$ is
empty, i.e. $k_{\gs} = 0$ we may write $q_{\gs}(G)$.

\noindent
1A)  If $G_1 \subseteq G_2 \subseteq \mathbf K$ and $t \in \deef(G_1)$
\then \, $t \in \deef(G_2)$ and $q_t(G_1) \subseteq q_t(G_2)$.

\noindent
2) Assume $G \subseteq H \in \mathbf K$ and $G$ is 
existentially closed or just $G
\le_{\Sigma_1} H \in \mathbf K$.  If $t_1,t_2 \in$ {\rm def}$(G)$ 
\then \, $q_{t_1}(G) = q_{t_2}(G)$ iff $q_{t_1}(H) = q_{t_2}(H)$.

\noindent
3) Let $K \subseteq G \in \mathbf K,G$ be existentially closed
or just every $r \in \mathbf S^{< \omega}_{\bs}(K)$ 
is realized in $G,K$ is finite, and $p \in {\mathbf S}^n_{\bs}(G)$.

The type $p$ does not split over $K$ \Iff \, there are
$\gs \in \Omega[\mathbf K]$ and a finite sequence $\bar a$ from $K$
(even listing $K$) realizing 
$p_{\gs}(\bar x)$ such that $p = q_{\gs}(\bar a,M)$.

\noindent
4) If $G \subseteq H,\gs \in \Omega[\mathbf K],\bar a \in {}^{k(\gs)}G$
   realizes $p_{\gs}(\bar x_{\gs})$ and $\bar c \in {}^{n(\gs)}H$
   realizes $q_{\gs}(\bar a,G)$ in $H$ and $\sigma(\bar z_{\gs},\bar
   x_{\gs})$ is a group-term \then \, $\sigma^H(\bar c,\bar a) \in G
   \Rightarrow \sigma^H(\bar c,\bar a) \in c \ell(\bar a,G)$.

\noindent
4A) In (4), if $\bar a' = \bar a \char 94 \bar a''$ then
$\sigma^H(\bar c,\bar a'') \in G \Rightarrow \sigma^H(\bar c,\bar a'')
\in c \ell(\bar a'',G)$ because $p$ also does not split over $\bar
a^*$ if $\bar a^* \subseteq G,\bar a \subseteq c \ell(\bar a^*,H)$.
\end{claim}

\begin{PROOF}{\ref{a15}}
1) Let $K_* \subseteq G$ be the subgroup of $G$ generated by $\bar a$.  

First, there are $H$ and $\bar c$ such that:
\mn
\begin{enumerate}
\item[$(*)^1_{H,\bar c}$]  $G \subseteq H \in \mathbf K$ and $\bar c \in
  {}^n H$ such that $\tp(\bar c,G,H) = q_{\gs}(\bar a,G)$.
\end{enumerate}
\mn
Why?  For every $K \in \mathbf K^* := \{K \subseteq G:K$ finite extending
$K_*\}$ we can choose a pair $(H_K,\bar c_K)$ such that: $K \subseteq
H_K \in \mathbf K,H_K$ is finite, $\bar c_K \in {}^n(H_K),H_K$ is
generated by $K \cup \bar c_K$ and for some $\bar b$ 
listing $K,\tp_{\bs}(\bar a \char 94 \bar b \char 94 \bar c_K,
\emptyset,H_K) = \gs(\tp_{\bs}(\bar a \char 94 \bar b,\emptyset,G))$.

[Why? By Definition \ref{a14}(1A)(b).]  Now for every $K_1 \subseteq
K_2$ from $\mathbf K_*$ we can choose an embedding $f_{K_2,K_1}$ from
$H_{K_1}$ into $H_{K_2}$ extending $\id_{K_1}$ and mapping $\bar
c_{K_1}$ to $\bar c_{K_2}$.  [Why?  By Definition \ref{a14}(1A)(c).]

As $H_{K_1}$ is generated by $K_1 \cup \bar c$, this mapping is
unique.  Now if $K_1 \subseteq K_2 \subseteq K_3$ are from $\mathbf K_*$
then $f_{K_3,K_2} \circ f_{K_2,K_1}$ is an embedding of $H_{K_1}$ into
$H_{K_3}$ extending $\id_{K_1}$ and mapping $\bar c_{K_1}$ to $\bar
c_{K_3}$; hence by the previous sentence $f_{K_3,K_2} \circ
f_{K_2,K_1} = f_{K_3,K_1}$.  Hence $\langle H_{K_1},f_{K_2,K_1}:K_1
\subseteq K_2$ are from $\mathbf K_* \rangle$ has a direct limit, i.e. we
can find a group $H$ and $\bar f = \langle f_K:K \in \mathbf K_* \rangle$
such that $f_K$ embed $H_K$ into $H$ and for every $K_1 \subseteq
K_2$ from $\mathbf K_*$ 
we have $f_{K_1} = f_{K_2} \circ f_{K_2,K_1}$. 
\Wilog \, $H = \cup\{f_K(H_K):K \in \mathbf K_*\}$ hence $H$ is a
locally finite group and $\{f_K:K \in \mathbf K_*\}$ 
embeds $G$ into $H$, so \wilog \, $G \subseteq H$ and $f_K \rest K =
\id_K$ for $K \in \mathbf K_*$.  Letting $\bar c = f_K(\bar c_K)$ for any $K \in
\mathbf K_*$, clearly $(H,\bar c)$ is as required in $(*)^1_{H,\bar c}$.
\mn
\begin{enumerate}
\item[$(*)_2$]  $\tp_{\bs}(\bar c,G,H)$ belongs to $\mathbf S^n_{\bs}(G)$.
\end{enumerate}
\mn
[Why?  By the definitions of $\mathbf S^n_{\bs}(G)$ because $G \subseteq
H \in \mathbf K$ and $\bar c \in {}^n H$.]
\mn
\begin{enumerate}
\item[$(*)_3$]  $q_{\gs}(\bar a,G)$ is unique and does not split over
    $\bar a$.
\end{enumerate}
\mn
[Why?  See Definition \ref{a14}(1A)(c).]

\noindent
1A) See Definition \ref{a14}(2).

\noindent
2) For $\ell=1,2$ we have $q_{t_\ell}(G) \subseteq q_{t_\ell}(H)$,
moreover, $q_{t_\ell}(G) = \{\varphi(\bar z_{n(t)},\bar b) \in 
q_{t_\ell}(H):\bar b \subseteq G\}$.  For the other direction, note
that $\bar a_{t_1},\bar a_{t_2} \subseteq G$ and assume $q_{t_1}(H)
\ne q_{t_2}(H)$, hence there are $m$ and $\bar b \in {}^m H$ and a basic
formula $\varphi(\bar y_{m},\bar z_{n})$ such that $\varphi(\bar
b,\bar z_{n}) \in q_{t_1}(H),\neg \varphi(\bar b,\bar z_{n}) \in
q_{t_2}(H)$.  Now there is $\bar b' \in {}^m G$ such that $\tp_{\bs}(\bar
b',\bar a_{t_1} \char 94 \bar a_{t_2},G) = \tp_{\bs}(\bar b,\bar a_{t_1}
\char 94 \bar a_{t_2},H)$ because $G \le_{\Sigma_1} H$ and our choice
of $\mathbf K$.  As $q_{t_\ell}(H)$ does not split over $\bar
a_{t_\ell}$, clearly $\varphi(\bar b',\bar z_{n}) \in q_{t_\ell}(H)
\Leftrightarrow \varphi(\bar b,\bar z_{n}) \in q_{t_\ell}(H)$ for $\ell=1,2$.

Together with an earlier sentence, $\varphi(\bar b',\bar z_{n}) \in
q_{t_1}(H),\neg\varphi(\bar b',\bar z_{n}) \in q_{t_2}(H)$ hence by
the first sentence in the proof of \ref{a15}(2) we have $\varphi(\bar
b',\bar z_{n}) \in q_{t_1}(G)$ and $\neg\varphi(\bar b',\bar
z_{n}) \in q_{t_2}(G)$ hence $q_{t_2}(G) \ne q_{t_2}(G)$ so we are
also done with the ``other" direction.

\noindent
3) The implication ``if" holds by \ref{a15}(1).  For the other
direction assume $p$ does not split over $K$.  As $K$ is finite, let
$k = |K|$ and let $\bar a \in {}^n K \subseteq {}^n G$ list $K$.

We now define $\gs$ by:
\mn
\begin{enumerate}
\item[$(a)$]  $p_{\gs} = \tp_{\bs}(\bar a,\emptyset,K)$ so $k_{\gs} = k$
\sn
\item[$(b)$]  $q = \gs(s(\bar x,\bar y))$ \Iff \, for some $\bar b \in
  {}^m G$ we have:
\sn
\begin{enumerate}
\item[$\bullet$]  $s(\bar x,\bar y) = \tp(\bar a \char 94 \bar
  b,\emptyset,G)$
\sn
\item[$\bullet$]  $q = \tp_{\bs}(\bar a \char 94 \bar b \char 94 \bar c,
\emptyset,G)$ for some $\bar c \in {}^n G$ realizing $p
  \rest (\bar a \char 94 \bar b)$.
\end{enumerate}
\end{enumerate}
\mn
Now $\gs$ is well defined because on the one hand $p$ does not split
over $\bar a$, and on the other hand $G$ is existentially closed or just every
$r \in \mathbf S^{< \omega}_{\bs}(K)$ is realized in $G$.

\noindent
4) By \ref{a15}(2) \wilog \, $G$ is existentially closed, assume $\sigma^H(\bar
c,\bar a) \in G$ and let $b = \sigma^H(\bar c,\bar a)$.  If $b
\notin c \ell(\bar a,G)$ there is $b' \in G \backslash \{b\}$
realizing $\tp_{\bs}(b,K,G)$ because $\mathbf K$ has disjoint amalgamation for
finite members.  As $q_{\gs}(\bar a,G)$ does not split over
$\bar a$ and $b',b \in G$ realize the same type over $\bar a$ it follows
that $H \models ``(\sigma(\bar c,\bar a)=b) \equiv (\sigma(\bar c,\bar
a) = b')"$, an obvious contradiction.

\noindent
4A) Should be clear.
\end{PROOF}

\begin{example}
\label{a16}
There is $\gs \in \Omega[\mathbf K]$ such that:
\mn
\begin{enumerate}
\item[$(a)$]   $k_{\gs} = 0$ and $n_{\gs}=1$;
\sn
\item[$(b)$]  if $G \subseteq H \in \mathbf K$ and $a \in H$,
  \then\,: $a$ realizes $q_{\gs}(<>,G)$ \Iff \, $a \in H \backslash G$ 
has order 2 and commute with every member of $G$.
\end{enumerate}
\end{example}

\begin{definition}
\label{a18}
1) For $\gS \subseteq \Omega[\mathbf K]$ we define the two place relation
$\le_{\gS}$ on $\mathbf K$ as follows: $M \le_{\gS} N$ iff $M \subseteq
N$ (and they belong to $\mathbf K$) and for every $n < \omega$ and $\bar c \in
{}^n N$ we can find $k < \omega$ and $\bar a \in {}^k M$ and $\gs \in
\gS$ such that $p_{\gs}(\bar x) \subseteq \tp_{\bs}(\bar a,\emptyset,M) 
\in \mathbf D_k(\mathbf K)$ and $\tp_{\bs}(\bar c,M,N) =
q_{\gs}(\bar a,M)$ recalling $q_{\gs}(\bar a,M) \in \mathbf S^n_{\bs}(M)$ .

\noindent
2) For $M \in \mathbf K$ and $\gS \subseteq \gS[\mathbf K]$ let
\mn
\begin{enumerate}
\item[$(a)$]   $\mathbf S^n_{\gS}(M) = 
\{q_{\gs}(\bar a,M):\gs \in \gS$ satisfies $n_{\gs}
   = n$ and $\bar a \in {}^{k(\gs)}M$ realizes $p_{\gs}(\bar
x_{\gs})\}$
\sn
\item[$(b)$]  $\deef_{\gS}(M) = \{t \in$ {\rm def}$(M):\gs_t \in \gS\}$
\sn
\item[$(c)$]  $\mathbf S_{\gS}(M) = \cup\{\mathbf S^n_{\gS}(M):n <
  \omega\}$.
\end{enumerate}
\mn
3) We say $M \in \mathbf K$ is $\gS$-\underline{existentially closed}  
\when \, for
 every $\gs \in \gS$, finite\footnote{For general $\mathbf K$: we use
 finitely generated $G \subseteq M$; generally this change is needed.} 
$G \subseteq M$ and $\bar a \in {}^{\omega >} G$ 
realizing $p_{\gs}(\bar x)$ the type 
$q_{\gs}(\bar a,G)$ is realized in $M$; (this is equivalent to being  
existentially closed if $\gS$ is dense, see Definition \ref{a21}(2) below).
\end{definition}

\begin{definition}
\label{a20}
We say $\gS \subseteq \Omega[\mathbf K]$ is \underline{symmetric} 
\when \, : if $\gs_1,\gs_2
\in \gS,M \subseteq N$ are from $\mathbf K$ and $\bar c_\ell \in
{}^{n(\gs_\ell)}N$ realizes $q_{\gs_\ell}(\bar a_\ell,M)$ in $N$ (so $\bar
a_\ell \in {}^{k(\gs_\ell)}M$ realizes $p_{\gs_\ell}(\bar x_{\gs_\ell}))$
and $M_\ell = c \ell(M + \bar c_\ell,N) \subseteq N$ for $\ell=1,2$
\then \, $\bar c_1$ realizes
$q_{\gs_1}(\bar a_1,M_2)$ in $N$ iff $\bar c_2$ realizes 
$q_{\gs_2}(\bar a_2,M_1)$ in $N$.
\end{definition}

\begin{definition}  
\label{a21}
1) We say $\gS$ is closed \when \, it is dominating-\underline{closed} and
   composition-closed, see below and invariant of course.

\noindent
1A) $\gS$ is \underline{composition-closed} \when \, if $H_0 \subseteq H_1
\subseteq H_2 \in \mathbf K,\bar a_\ell \in {}^{n(\ell)}(H_\ell)$ for
$\ell=0,1,2$ and tp$_{\bs}(\bar a_{\ell +1},
H_\ell,H_{\ell +1}) = q_{\gs_\ell}(\bar a_\ell,H_\ell) \in \mathbf
S^{n(\ell +1)}_{\gS}(H_\ell)$ and
$H_{\ell +1} = c \ell(H_\ell + \bar a_\ell,H_{\ell +1}),
\gs_\ell \in \gS$ for $\ell=0,1$ \then \, tp$_{\bs}(\bar a_1 
\char 94 \bar a_2,H_0,H_2) = q_{\gs}(\bar
a_0,H_0)$ for some $\gs \in \gS \cap 
\Omega_{n(1)+n(2),n(0)}[\mathbf K]$.

\noindent
1B) $\gS$ is \underline{dominating-closed} \when \,: if 
$H_0 \subseteq H_1 \in \mathbf K,\bar a_1 \in {}^{k(1)}(H_0),
\bar c_1 \in {}^{n(1)}(H_1)$, tp$_{\bs}(\bar c_1,H_0,H_1) =
q_{\gs}(\bar a_1,H_0) \in \mathbf S^{n(1)}_{\gS}(H_0)$ and $\bar c_2 \in
{}^{n(2)}(H_1)$ and $\bar a_2 \in {}^{k(2)}(H_0)$, Rang$(\bar a_2)
\supseteq \text{Rang}(\bar a_1)$ and $\bar c_2 \subseteq 
c \ell(\bar a_2 + \bar c_1,H_1)$ \then \, tp$(\bar
c_2,H_0,H_1) = q_{\gs}(\bar a_2,H_0)$ for some $\gs \in \gS$.

\noindent
2) We say $\gS$ is weakly dense \when \,: every $\gS$-existentially
closed $G \in K$ is existentially closed.

\noindent
3) We say $\gS$ is \underline{dense} when: 
for every $G_0 \subseteq H \in \mathbf
   K,G_0 \subseteq G_1 \in \mathbf K,G_0,G_1$ are finite 
and $\bar c \in {}^n(G_1)$ 
\underline{there is} $p(\bar z) \in \mathbf S^n_{\gS}(H)$ which 
extends $\tp_{\bs}(\bar c,G_0,G_1)$.  Moreover $p(\bar z) = 
q_{\gs}(\bar a,H)$ for some $\gs \in \gS$ and $\bar a$ from $G_0$.

\noindent
4) For disjoint $\gs_1,\gs_2 \in \gS$ define $\gs = \gs_1 \oplus \gs_2$ with
$p_{\gs}(\bar x_{\gs}) = p_{\gs_1}(\bar x_{\gs_1}) \cup p_{\gs_2}
(\bar x_{\gs_2})$, recalling $\bar x_{\gs_1},\bar x_{\gs_2}$ are
disjoint, as follow: if $G_0 \subseteq G_1 \subseteq G_2$ are from
$\mathbf K$ and $\bar a_\ell \in {}^{k(\gs_\ell)}G_0,\bar a_\ell$ 
realizes $p_{\gs_\ell}(\bar x_{\gs_\ell})$ in $G_0 
\in \mathbf K$ and $\bar c_\ell \in 
{}^{n(\gs_\ell)}(G_{\ell +1})$ realizes $q_{\gs_\ell}(\bar
a_\ell,G_\ell)$ for $\ell=1,2$ \then \, $\bar c_1 \char 94 \bar c_2$ realizes 
$q_{\gs}(\bar a_1 \char 94 \bar a_2,G_0)$ in $G_2$.

\noindent
4A) For (disjoint) $t_1,t_2 \in \text{ def}(G)$ we define $t_1 \oplus
t_2 = t_1 \oplus_G t_2$ similarly.

\noindent
5) We define $\bigoplus\limits_{k < m} \gs_k,\bigoplus\limits_{k <
m} t_k$ similarly using associativity, see \ref{a22}(5).

\noindent
6) Let $\gs_1 \le \gs_2$ means: if $G \in \mathbf K,\bar a_2
\in {}^{u(2)}G$ realizes $p_{\gs_2}(\bar x_{\gs_2}),G \subseteq
   H,\bar c_2 \in {}^{n(\gs_2)}H$ realizes $q_{\gs_2}(\bar a_2,G)$
   \then \, $\dom(\bar x_{\gs_1}) \subseteq u(2)$ and
$\bar c_2 \rest \dom(\bar z_{\gs_2})$ realizes $q_{\gs_1}(\bar a_2
\rest k(\gs_1),G)$ and $p_{\gs_2}(\bar x_{\gs_2}) \rest 
\bar x_{\gs_1} = p_{\gs_1}(\bar x_{\gs_1})$.

\noindent
7) Let $\gs_1 \le_{\bar h} \gs_2$ means that $\bar h = (h',h''),h'$ is
   a one-to-one function from dom$(\bar x_{\gs_1})$ into dom$(\bar x_{\gs_2})$
   and $h''$ is a one-to-one function from dom$(\bar z_{\gs_1})$ into 
dom$(\bar z_{\gs_2})$ such that: if tp$_{\bs}(\bar
 c_2,G,H) = q_{\gs_2}(\bar a_2,G)$ and $\bar a_1 = \langle
   a_{2,h''(\ell)}:\ell \in \text{ dom}(\bar a_1)\rangle$ and $\bar
   c_1 = \langle c_{2,h(\ell)}:\ell \in \text{ dom}(\bar c_2)\rangle$
   then tp$_{\bs}(\bar c_1,G,H) = q_{\gs_1}(\bar a_1,G,H)$.
   Similarly $t_1 \le_{\bar h} t_2$ for $t_1,t_2 \in \text{ def}(G)$.
   If $h' \cup h''$ is well defined we may write $h' \cup h''$ instead
   of $\bar h$.
\end{definition}

\begin{remark}
\label{a21m}
0) Concerning \ref{a21}(7) the point of disjoint $\gs_1,\gs_2$ and
   congruency is to avoid using it.  So we may ignore it as well as
\ref{a22f}(2),(3), \ref{a56}(3), \ref{a57}(4), \ref{a59}(5).

\noindent
1) Note that the operation $\gs_1 \oplus \gs_2$ is not necessarily
commutative, e.g. for $\mathbf K_{\text{olf}}$ it cannot be.

\noindent
2) In e.g. Definition \ref{a21}(1A), in general $\gs$ is not uniquely
determined by the relevant information $\tp_{\bs}(\bar a_1 
\char 94 \bar a_2 \char 94 \bar c_1 \char 94 \bar c_2,H_0,H_2)$ 
and the lengths of $\bar a_1,
\bar a_2,\bar c_1,\bar c_2$ \underline {but} if 
$H_1$ is existentially closed, it is.  We could 
have written the definition in a computational form.

\noindent
3) So $\gs_1 \le \gs_1$ means $\gs_1 \le_{\bar h} \gs_2$ with $h_\ell$
the identity for $\ell=1,2$.
\end{remark}

\begin{dc}
\label{a22}
1) For any $\gS \subseteq \Omega[\mathbf K]$ we can define its closure as the
minimal closed (and \underline{invariant}, of course) 
$\gS_1 \subseteq \Omega[\mathbf K]$ which includes it, see \ref{a21}(1);
we denote it by $c \ell(\gS) = c \ell(\gS;\mathbf K)$.

\noindent
2) Similarly for \underline{dominating-closure} $\docl(\gS)$ 
and \underline{composition-closure} cocl$(\gS)$.

\noindent
3) Those closures preserve density and countability (and being invariant),
 and have the obvious closure properties.

\noindent
4) Also dominating-closure preserve being composition closed.

\noindent
5) The operation $\oplus$ on $\Omega[\mathbf K]$ is well defined and
associative. If $\gS \subseteq \Omega[\mathbf K]$ is closed under
$\oplus$, for transparency, \then \, $\gS$ is symmetric (see
\ref{a20}) \Iff \, the operation $\oplus$ on $\gS$ is 
commutative (when defined).  Similarly for def$_{\gS}(G)$.

\noindent
6) $\Omega[\mathbf K]$ has cardinality $\le 2^{\aleph_0}$; generally
$\le 2^{|\tau(\mathbf K)| + \aleph_0}$.

\noindent
7) $\le_{\gS}$ is a transitive relation on $\mathbf K$, if $\gS \subseteq
\Omega[\mathbf K]$ is closed.

\noindent
8) If $H_0 \subseteq H_1 \subseteq H_2,\gs \in \Omega[\mathbf K]$ 
and tp$_{\bs}(\bar c,H_1,H_2) = q_{\gs}(\bar a,H_1)$ 
and $\bar a \in {}^{k(\gs)}H_0$
   \then \, Rang$(\bar c) \cap H_1 = \text{ Rang}(\bar c) \cap H_0$. 

\noindent
9) Assume $\gS$ is dense and closed. 
If $G \subseteq H \in \mathbf K$ and $G$ is finite then $G \le_{\gS} H$.

\noindent
10) If $\gs = \gs_0 \oplus \ldots \oplus \gs_{n-1}$ and $i(0) < \ldots
   i(k-1) < n$ and $\gs' = \gs_{i(0)} \oplus \ldots \oplus \gs_{i(k-1)}$
 \then \, $\gs' \le \gs$.
\end{dc}

\begin{PROOF}{\ref{a22}}  Natural, noting that (8) is specific for our present
 $\mathbf K$, see \ref{a15}(4).
\end{PROOF}

\begin{claim}
\label{a22f}
0) The operation $\oplus$ is well defined, that is:
\mn
\begin{enumerate}
\item[$(a)$]  if $\gs_1,\gs_2 \in \Omega[\mathbf K]$ are disjoint \then \,
  $\gs_1 \oplus \gs \in \Omega[\mathbf K]$ is well defined;
\sn
\item[$(b)$]  if $t_1,t_2 \in \deef(G)$ are disjoint \then \, $t_1 \oplus
  t_2 \in \deef(G)$.
\end{enumerate}
\mn
1) The operation $\oplus$ on disjoint pairs from {\rm def}$(G)$ 
respects congruency, see Definition \ref{a14f}(3).  
If $\gs_1,\gs_2 \in \Omega[\mathbf K]$ \then \,
$(\gs_1/\equiv) \oplus (\gs_2/\equiv)$ is well defined, i.e. if
$\gs'_\ell,\gs''_\ell$ are congruent to $\gs_\ell$ for $\ell=1,2$ and
$\gs' = \gs'_1 \oplus \gs'_2,\gs'' = \gs''_1 \oplus \gs''_2$ are well
defined (equivalently for $\ell=1,2$ the two schemes
$\gs'_\ell,\gs''_\ell$ are disjoint) then $\gs',\gs''$ are congruent.  (So we
may forget to be pedantic about this.)

\noindent
2) If $(\gs,\bar a) = (\gs_1,\bar a_1) \oplus_G (\gs_2,\bar a_2)$
\then \, $(\gs_\ell,\bar a_\ell) \le (\gs,\bar a)$.

\noindent
3) If in {\rm def}$(G)$ we have $(\gs_\ell,\bar a_\ell) \le_{h_\ell}
(\gs'_\ell,\bar a'_\ell)$ for $\ell=1,2$ and {\rm Dom}$(h_1) \cap
   \Dom(h_2) = \emptyset$, {\rm Rang}$(h_1) \cap \Rang(h_1) 
= \emptyset$ \then \, $(\gs_1,\bar a_1)
\oplus (\gs_2,\bar a_2) \le_{h_1 \cup h_2} (\gs'_1,\bar a_1) \oplus
(\gs'_2,\bar a_2)$.  Similarly for $\bar h_1,\bar h_2$.
\end{claim}

\begin{PROOF}{\ref{a22f}}
Straightforward.
\end{PROOF}

\begin{claim}
\label{a23}
Assume $\gS \subseteq \Omega[\mathbf K]$ is dominating-closed and $G_0
\subseteq G_1 \in \mathbf K$ and $G_0 \le_{\gS} G_2$ and, for
transparency, $G_1 \cap G_2 = G_0$ and\footnote{If $G_2 = \langle G_0
\cup A \rangle,A$ finite then for part (1) this is not necessary.} 
$G_0 \le_{\Sigma_1} G_2$ (holds if $G_0$ is existentially closed in
$\mathbf K$).

\noindent
1) There is $G_3 \in \mathbf K$ such that $G_1 \le_{\gS} G_3$ and $G_2
   \subseteq G_3$ and $G_3 = \langle G_1 \cup G_2\rangle_{G_3}$ and
$G_1 \le_{\Sigma_1} G_3$.

\noindent
2) Above $G_3$ above is unique up to isomorphism over $G_1 \cup G_2$.

\noindent
3) If $\gS$ is symmetric and $G_0 \le_{\gS} G_1$ in part (1)
 \then \, also $G_2 \le_{\gS} G_3$. 
\end{claim}

\begin{PROOF}{\ref{a23}}
Straightforward, e.g.

\noindent
1) Let $\bar c = \langle c_\alpha:\alpha <
 \alpha(*)\rangle$ list the elements of $G_2$, and for every finite $u
 \subseteq \alpha(*)$ let $\bar x_u = \langle x_\alpha:\alpha \in u
 \rangle$ and $p^0_u(\bar x_{u}) = \text{ tp}_{\bs}(\bar c \rest
 u,G_0,G_2)$ hence by assumption, there is $\gs_u \in \gS$ (up to
 congruency) and $\bar a_u \in {}^{k(\gs_u)}(G_0)$ such that
 $p^0_u(\bar x) = q_{\gs_u}(\bar a_u,G_0)$ so $\dom(\bar
 x_{\gs_u})=u$.  We define $p^1_u(\bar x_u) \in 
\mathbf S(G_1)$ as $q_{\gs_u}(\bar a_u,G_1)$.  We define $G_3$
 as a group extending $G_1$ generated by $G_1 \cup \{c_\alpha:\alpha <
 \alpha(*)\}$ such that $\bar c \rest u$ realizes $p^1_u(\bar x_u)$
 for every finite $u \subseteq \alpha(*)$.  But for this to work we
 have to prove that for finite $u \subseteq v \subseteq \alpha(*)$ we
 have $p^1_u(\bar x_u) \subseteq p^1_v(\bar x_v)$.  This is straightforward 
recalling \ref{a15}(1A).

Lastly, $G_1 \le_{\Sigma_1} G_3$ is easy, too.
\end{PROOF}

\begin{remark}
\label{a23f}
1) We may consider an alternative definition of $\le_{\gS}$:
\mn
\begin{enumerate}
\item[$\bullet_1$]  $G \le_{\gS} H$ \Iff \, for every finite 
$A \subseteq H$ there are $\bar c \in {}^{\omega >} H,
\bar a \in {}^{\omega >}G$ and $\gs \in \gS$ such that:
$\bar a$ realizes $p_{\gs}(\bar x_{\gs}),\bar c$ realizes
$q_{\gs}(\bar a,G)$ in $H$ and $A \subseteq \Rang(\bar c)$.
\end{enumerate}
\mn
An even weaker version is:
\mn
\begin{enumerate}
\item[$\bullet_2$]  as in $\bullet_1$ but ``$A \subseteq \Rang(\bar
  c)$" is replaced by $A \subseteq c \ell(G \cup \bar c,H)$.
\end{enumerate}
\mn
2) But, e.g. for $\bullet_1$, to prove 
$\le_{\gS}$ is transitive we need a stronger version of
   composition-closed: if $G_0 \subseteq G_1 \subseteq G_2$ and for
$\ell=0,1,\bar c_\ell \in {}^{n(\ell)}(G_{\ell +1})$ realizes
   $q_{\gs_\ell}(\bar a_\ell,G_\ell)$ and Rang$(\bar b_0) \subseteq
   \text{ Rang}(\bar a_1)$ \then \, for some $\gs \in
\gS,p_{\gs}(\bar x_{\gs}) = p_{\gs_0}(\bar x_{\gs})$ and $\bar a_1
\char 94 \bar a_2$ realizes $q_{\gs}(\bar a_0,G_0)$.

\noindent
3) In any case for closed $\gS$ the three definitions are equivalent,
i.e. those in $\bullet_1$, in $\bullet_2$ and in \ref{a18}(1).

\noindent
4) Does the operation $\oplus_G$ respect $\approx_G$, see Definition
\ref{a14f}, i.e. if $t_1 \approx_G t'_1$ and $t_2 \approx_G t'_2$ \then \, $t_1 
\oplus_G t_2 \approx_G t'_1 \oplus_G t'_2$?; all this 
assuming the operations are well
defined, i.e. the disjointness demands from \ref{a21}(4) are satisfied.  
We do not see a reason for this to hold.
\end{remark}
\bigskip

\subsection {Constructions}\
\bigskip

Before we present the more systematic construction from \cite[Ch.IV]{Sh:c}, we
give a self-contained direct definition and proof for the existence of
 a canonical existentially
closed extension of $G \in \mathbf K$ when $\gS$ is symmetric, i.e. the
``second avenue" in \S0B.  
We shall deal with the non-symmetric case later.

\begin{definition}
\label{a25}
Assume $\gS \subseteq \Omega[\mathbf K]$ is symmetric.

\noindent
1) We say $H$ is a $\gS$-\underline{closure} of $G$ \when \, 
there is a sequence $\langle G_n:n < \omega\rangle$ such that 
$G_0 = G,H = \cup\{G_n:n < \omega\}$ and 
$G_{n+1}$ is a one-step $\gS$-closure of $G_n$, see below.

\noindent
2) We say that $H$ is a \underline{one-step} $\gS$-\underline{closure}
of $G$ \when \,:
\mn
\begin{enumerate}
\item[$(a)$]  $G \subseteq H$ are from $\mathbf K$;
\sn
\item[$(b)$]  $S := \text{ def}(G) = 
\{(\gs,\bar a):\gs \in \gS$ and $\bar a \in
{}^{\omega >} G$ realizes $p_{\gs}(\bar x_{\gs})\}$ and let $t =
(\gs_t,\bar a_t) = (\gs(t),\bar a(t))$ for $t \in S$;
\sn
\item[$(c)$]  $\bar c_t \in {}^{n(\gs(t))}H$ realizes $q_{\gs_t}
(\bar a_t,G)$ for $t \in S$;
\sn
\item[$(d)$]  $H$ is generated by $\{\bar c_t:t \in S\} \cup G$;
\sn
\item[$(e)$]  $\bar c_t$ realizes $q_{\gs_t}(\bar c_t,c \ell(\cup\{\bar
c_s:s \in S \backslash \{t\}\} \cup G,H)$ inside $H$ for every $t \in S$.
\end{enumerate}
\end{definition}

\begin{claim}
\label{a26}
Let $\gS \subseteq \Omega[\mathbf K]$ be symmetric.

\noindent
1) For every $G \in \mathbf K$ there is a one-step $\gS$-closure $H$ of
   $G$.

\noindent
2) For every $G \in \mathbf K$ there is an $\gS$-closure $H$ of $G$.

\noindent
3) In both parts (1) and (2) we have $|G| \le |H| \le |G| + |\gS| +
   \aleph_0$.

\noindent
4) In both parts (1) and (2), $H$ is unique up to isomorphism over $G$.

\noindent
5) If the pair $(G_\ell,H_\ell)$ is as in part (1), or as in 
part (2) for $\ell=1,2$ and $G_1 \subseteq G_2$ \then \, $H_1$ 
can be embedded into $H_2$ over $G_1$.

\noindent
6) In both parts (1) and (2) there is a set theoretic class function 
$\mathbf F$ computing $H$ from $G$, pedantically for every $G \in \mathbf
K$ and ordinal $\alpha$ not in the transitive closure $\tr - c
\ell(G)$ of $G,\mathbf F_\alpha(G)$ is well defined such that:
\mn
\begin{enumerate}
\item[(A)]
\begin{enumerate}
\item[(a)]  $\mathbf F_\alpha(G) \in \mathbf K_{\lf}$ is of cardinality
  $\le |G| + \aleph_0 + |\alpha|$
\sn
\item[(b)]  if $\alpha = 0$ then $\mathbf F_\alpha(G)=G$
\sn
\item[(c)]  the sequence $\langle \mathbf F_\beta(G):\beta \le
  \alpha\rangle$ is increasing continuous
\sn
\item[(d)]  $\mathbf F_{\alpha +1}(G)$ is a one step closure of $\mathbf
  F_\alpha(G)$
\end{enumerate}
\sn
\item[(B)]  if $G_1 \subseteq G_2 \wedge G_2 \cap \mathbf F_\alpha(G_1) 
= G_1 \wedge \bigwedge\limits^{2}_{\ell=1} \emptyset = (\alpha +1) 
\cap \tr - c \ell(G_\ell)$ 
then $\mathbf F_\alpha(G_1) \subseteq \mathbf F_\alpha(G_2)$; 
this is ``naturality"; an alternative is \ref{y8}(2).
\end{enumerate}
\mn
7) In fact we do not have to use the axiom of choice.
\end{claim}

\begin{PROOF}{\ref{a26}}
Should be clear (alternatively, below we do more).
\end{PROOF}

\begin{remark}
\label{a28}
 Similarly in \S3. 
\end{remark}

\begin{definition}
\label{a33}
1) We say $N$ is $(\lambda,\gS)$-\underline{full} over $M$ 
when: $M \subseteq N$ and
\underline{if} $M \subseteq M_1 \subseteq N$ and 
$M_1 = c \ell(M+A,N)$ for some $A
\subseteq M_1$ of cardinality $< \lambda$ and $\gs \in \gS$ and $\bar
a \in {}^{k(\gs)}M_1$ realizes $p_{\gs}(\bar x_{\gs})$ in $M_1$ 
then $q_{\gs}(\bar a,M_1)$ is realized in $N$.

\noindent
2) We may write ``$N$ is $\gS$-full over $M$" when $\lambda = \|N\|$ is regular
\underline{or}, in general, when\footnote{For the case $\gS$ is not
  symmetric and $\lambda$ is singular, if we like to have ``prime",
(as in \ref{a34}(5)) we should add: for every pair $t = (\gs,\bar a)$
  as in \ref{a33}(2), for every large enough $\mu < \lambda$, for
  every $\alpha < \mu^+$ for some $\bar c \subseteq M_{\mu^+}$
  requires $q_{\gs}(\bar a,M_\alpha)$ is realized in $M_{\mu^+}$; 
also we can in \ref{a41}(1) have
  such $\cA$, i.e. strengthen (d) there as here so weakens the
 assumption in \ref{a41}(4).}
 there is a list $\langle a_\alpha:\alpha < \|N\|\rangle$ of $N$ 
such that for every $\alpha <
\|N\|$ and $\gs \in \gS$ we have:  if $M_\alpha = c \ell(M + \{a_\beta:\beta <
\alpha\},N)$ and $\bar a \in {}^{k(\gs)}M_\alpha$ realizes
$p_{\gs}(\bar x_{\gs})$ then the type $q_{\gs}(\bar a,M_\alpha)$ 
is realized in $N$ by $\|N\|$ elements.

\noindent
3)  We may omit $\gS$ when $\gS = \Omega[\mathbf K]$.
\end{definition}

\begin{claim}
\label{a34}  Let $\gS$ be symmetric.

\noindent
1) If $\gS \subseteq \gS(\mathbf K)$ is closed (see \ref{a21}(1)) \then \, 
$(\mathbf K,\le_{\gS})$ is a weak 
a.e.c. with amalgamation\footnote{Not enough for quoting results.}
 (even canonical), see \cite[1.2]{Sh:88r} or \cite[Ch.I]{Sh:h},
 i.e. in the Definition of a.e.c. we have Ax 0,(I),(II),(III),(V)
but $\LST(\mathbf K,\le_{\gS})$ may be $\infty$ and
we omit Ax(IV), see \ref{a35} below.

\noindent
2) If $\gS \subseteq \Omega[\mathbf K]$ is dense and closed (see \ref{a21}) \then
   \, for every $M \in \mathbf K_\lambda$ there is an existentially closed
   $N \in \mathbf K_\lambda$ which $\le_{\gS}$-extends it, in fact any
   $\gS$-closure of $M$ can serve.

\noindent
3) If $N$ is $(\lambda,\gS)$-full over $M_1$ and $M_0 \subseteq M_1$,
   \then \, $N$ is $(\lambda,\gS)$-full over $M_0$; also in Definition
\ref{a33} \wilog \, $\bar a$ is from $M \cup A$, i.e. $\bar a \in {}^{k(\gs)}(M
\cup A)$.

\noindent
4) If $M \in \mathbf K_{\le \lambda}$ \then \, there
is a model $N,(\lambda,\gS)$-full over $M$ of cardinality $\le \lambda +
\|M\| = \lambda$; moreover if $\gS$ is dense, 
then $M$ is existentially closed.

\noindent
5) In (4), we can add: if $N' \in \mathbf K$ 
is $(\lambda,\gS)$-full over $M$ \then \, we
can find an embedding of $N$ into $N'$ over $M$.
\end{claim}

\begin{PROOF}{\ref{a34}}
1) Easy.

\noindent
2) Easy by \ref{a26} and see more below.

\noindent
3) Easy.

\noindent
4) We choose $G_n \in \mathbf K$ by induction on $n$ such that:
\mn
\begin{enumerate}
\item[$(a)$]  $G_0 = M$;
\sn
\item[$(b)$]  $G_{n+1} \supseteq G_n$ is as in Definition \ref{a25}
  but each $t$ appears $\lambda$ times, i.e.
\sn
\begin{enumerate}
\item[$\bullet$]  $G_{n+1} = c\ell(\cup\{\bar c^n_{t,\alpha}:t \in
  \deef_{\gS}(G_n)$ and $\alpha < \lambda\} \cup G_n,G_{n+1})$ where
\sn
\item[$\bullet$]  $\tp_{\bs}(\bar c^n_{t,\alpha},G_{n,t,\alpha},G_{n+1}) =
  q_t(\bar a_t,G_{n,t,\alpha})$ where
\sn
\item[$\bullet$]  $G_{n,t,\alpha} = c \ell(\cup\{\bar c^n_{t_1,\alpha_1}:
t_1 \in \deef_{\gS}(G_n),\alpha_1 < \lambda$ but
  $(t_1,\alpha_1) \ne (t,\alpha)\} \cup G_n,G_{n+1})$.
\end{enumerate}
\end{enumerate}
\mn
Let $\hat G = \bigcup\limits_{n} G_n$ and we shall show that $\hat G$
is $(\lambda,\gS)$-full over $G$.  We can ignore the case $\lambda =
\aleph_0$ being obvious.  
Assume $A \subseteq \hat G,|A| < \lambda$ and $t_* \in
\deef_{\gS}(\hat G)$ and $\bar a_{t_*} \subseteq c \ell(G_0 +A,\hat G)$,
hence we can find $\bar S$ such that:
\mn
\begin{enumerate}
\item[$(*)$]  
\begin{enumerate}
\item[(a)]  $\bar S = \langle S_n:n < \omega\rangle$;
\sn
\item[(b)]  $S_n \subseteq \deef_{\gS}(G_n) \times \lambda$ and
  $\bigcup\limits_{m} S_m$ has cardinality $< |A|^+ + \aleph_0$;
\sn
\item[(c)]  if $(t,\alpha) \in S_n$ then $\bar a_t
  \subseteq c \ell(\cup\{\bar c^m_{t_1,\alpha_1}:m < n$ and 
$(t_1,\alpha_1) \in S_m\} \cup G_0,G_n)$;
\sn
\item[(d)]  $A \subseteq \bigcup\limits_{n} A_n \cup G_0$
where $A_n = \cup\{\bar c^m_{t,\alpha}:(t,\alpha) \in S_m$ and $m < n\}$;
\sn
\item[(e)]  for some $n_*,(t_*,0) \in S_{n_*}$.
\end{enumerate}
\end{enumerate}
\mn
We have to prove that $q_{t_*}(c \ell(A \cup G,\hat G))$ is realized
in $M$.  Choose $\alpha_*$ such that $(t_*,\alpha_*) \notin S_{n_*}$
and prove by induction on $n \ge n_*$ that 
$\bar c^{n_*}_{t_*,\alpha_*}$ realizes $q_{t_*}(c \ell(A_n \cup G_0,\hat
G))$.

For $n=n_*$ this is obvious, so assume this holds for $n$ and we shall
prove for $n+1$.

For this it suffices to prove, for every finite $u \subseteq S_n$ that
$\bar c^{n(*)}_{t_*,\alpha_*}$ realizes $q_{t_*}(c \ell(A_n \cup G_0
\cup \{\bar c^n_{t,\alpha}:(t,\alpha) \in u\},\hat G)$; we prove this by
induction on $|u|$.  Now if $|u|=0$ this holds by the induction
hypothesis on $n$ and if $|u| >0$, let $\beta \in u$ and use the
induction for $u' = u \backslash \{\beta\}$ and $\gS$ being symmetric.

\noindent
5) We can find a list $\langle (n_\zeta,t_\zeta,\alpha_\zeta):\zeta <
\lambda\rangle$ of $\{(n,t,\alpha):n < \omega$ and $(t,\alpha) \in
S_n\}$ such that $\bar a_{t_\zeta} \subseteq c \ell(\cup\{(\bar
c^{n_\xi}_{t_\xi,\alpha_\xi}:\xi < \zeta\} \cup M,N)$.

Now choose $f(\bar c^{n_\zeta}_{t_\zeta,\alpha_\zeta}) \subseteq N'$
by induction on $\zeta$.
\end{PROOF}

\begin{discussion}
\label{a45}  
1) So by \ref{a26}(2), \ref{a34}(2) if there is a symmetric closed 
dense $\gS$ \then \, for every lf group $G$ there is a ``nice"
extension of $G$ to an existentially closed one $\hat G$, that is we have:
\mn
\begin{enumerate}
\item[$(a)$]  uniqueness (by \ref{a26}(4))
\sn
\item[$(b)$]  cardinality $\le |\theta| + |\gS|$ (by \ref{a26}(3))
\sn
\item[$(c)$]  extending $G$ (see \ref{a25}(1))
\sn
\item[$(d)$]  being existentially closed (see \ref{a34}(2)).
\end{enumerate}
\mn
2) Fixing $\lambda$ and demanding $G \in \mathbf K_{\le \lambda}$ we can add
\mn
\begin{enumerate}
\item[$(e)$]  $\hat G$ is $(\lambda,\gS)$-full over $M$
\sn
\item[$(f)$]  if $H \supseteq G$ is $(\lambda,\gS)$-full \then \,
  there is an embedding of $\hat G$ into $H$ over $G$.
\end{enumerate}
\end{discussion}

\begin{discussion}
\label{a35}
Concerning \ref{a34}(1), if we 
assume $\langle G_\alpha:\alpha \le \delta + 1 \rangle$ is
$\subseteq$-increasing continuous and $\alpha <  \delta \Rightarrow G_\alpha
\le_{\gS} G_{\delta +1}$, does it follow that $G_\delta \le_{\gS} G_{\delta
+1}$?   This is Ax(IV) of the definition of a.e.c.  
Well, if $\delta$ has uncountable cofinality and 
each $G_\alpha$ is existentially closed then yes.  The point is that the
relevant types do not split over \underline{finite} sets.  If we deal
with ``not split over countable sets" we need cf$(\delta) \ge
\aleph_2$, etc.

So $(\mathbf K,\le_{\gS})$ is not an a.e.c. in general failing
Ax(IV); in fact, e.g. we may prove for the maximal $\gS$ that this
axiom fails, see the proof of \ref{p73}.

Now we turn to constructions not necessarily assuming ``$\gS$ is
symmetric" presenting the ``first avenue" in \S0(B).
\end{discussion}

\begin{definition}
\label{a37}  
1) We say that ${\cA} = \langle G_i,\bar a_j,w_j,K_j:i \le \alpha,
j < \alpha \rangle$ is an 
$\mathbf F^{\sch}_{\aleph_0}-\gS$-\underline{construction} 
(for $\mathbf K$) \when \, :
\mn
\begin{enumerate}
\item[(a)]   $G_i$ for $i \le \alpha$ is an $\le_{\gS}$-increasing continuous
sequence of members of $\mathbf K$;
\sn
\item[(b)]  $G_{i+1}$ is generated by $G_i \cup \bar a_i,\bar
a_i$ a finite sequence;
\sn
\item[(c)]   $w_i$ is a finite subset of $i$;
\sn
\item[(d)]   $K_i \subseteq G_i$ is finite;
\sn
\item[(d)$^+$]   moreover $K_i \subseteq \langle G_0 +
\sum\limits_{j \in w_i} \bar a_j\rangle_{G_i}$; we may add ``$K_j$
generated by $\cup \{\bar a_j:j \in w_i\} \cup (K_i \cap G_i)$;
\sn
\item[$(e)$]   tp$_{\bs}(\bar a_i,G_i,G_{i+1}) \in \mathbf S^{\ell
g(\bar a_i)}_{\gS}(G_i)$ as
witnessed by $K_i$, i.e. it is $q_{\gs}(\bar a,G_i)$ for some $\bar a
\in {}^{\omega >}K_i$ realizing $p_{\gs}$ for some $\gs \in \gS$.
\end{enumerate}
\mn
2) We may say above that $G_\alpha$ is $\mathbf
F^{\sch}_{\aleph_0}-\gS$-constructible over $G_0$; and may also say that
 ${\cA}$ is an $\gS$-construction over $G_0$.  We let $\alpha =
\ell g({\cA}),G_i = G^{\cA}_i,\bar a_i = \bar a^{\cA}_i,
w_j = w^{\cA}_j,K_j = K^{\cA}_j$.

\noindent
3) We say above that $\cA$ is a \underline{definite} 
$\mathbf F^{\sch}_{\aleph_0}-\gS$-\underline{construction} \when \, for every 
$j < \alpha$ we have 
also $t_j = t^{\cA}_j \in \deef(G^{\cA}_j)$ such that
$\bar a_{t_j} \in {}^{\omega >}(K_j)$ and $\bar a^{\cA}_j$ realizes
$q_{t_j}(G_j)$ (note that in \ref{a37}(1)(e) we have ``for some
$\gs_j$", so $\cA$ does not determine the $\gs$'s (or here the $t_j$;
so every $\mathbf F^{\sch}_{\aleph_0}-\gS$-construction
can be expanded to a definite one, but not necessarily uniquely).

\noindent
4) We say $\cA$ is a $\lambda$-\underline{full definite} $\mathbf
   F^{\sch}_{\aleph_0}-\gS$-\underline{construction} \when \, $\alpha$ is
   divisible by $\lambda$ and for every $i < \alpha$ and $t \in 
\deef(G_i)$, the set $\{j:j \in (i,\alpha)$ and $t^{\cA}_j=t\}$ is
   an unbounded subset of $\alpha(*)$ of order type divisible by $\lambda$.
\end{definition}

\begin{discussion}
\label{a38}
We may replace \ref{a37}(1)(e) by ``$\tp_{\bs}(\bar a_i,G_i,G_{i+1})$ does
   not split over $K_i$", this is like the case 
$\mathbf F^p_{\aleph_0}$ in \cite[Ch.IV,Def.2.6,pg.168]{Sh:c} and
\cite[Ch.IV,Lemma 2.20,pg.168]{Sh:c} and is equal to $\mathbf
F^{\nsp}_{\aleph_0}$ in \cite[\S1,1.1-1.12]{Sh:900}, both for first
order theories, but we seemingly lose the following:
\end{discussion}

\begin{observation}
\label{a39}
1)  If $\cA$ is a $\mathbf
F^{\sch}_{\aleph_0}-\gS$-construction and $G^{\cA}_0 \subseteq G$ and
$G \cap G^{\cA}_{\ell g(\cA)} = G^{\cA}_0$ \then \, there is an $\mathbf
F^{\sch}_{\aleph_0}-\gS$-construction $\cB$ with $G^{\cB}_0 = G,\ell
g(\cB) = \ell g(\cA)$ and $G^{\cB}_{\ell g(\cB)} = \langle
G^{\cA}_{\ell g(\cA)} \cup G \rangle_{G^{\cA}_{\ell g(\cA)}}$.

\noindent
2) Like (1) but with definite $F^{\sch}_{\aleph_0}-\gS$-constructions 
and then add in the end $t^{\cB}_j = t^{\cA}_j$ for $j < \ell g(\cA)$.

\noindent
3) For the definite version, see \ref{a37}(3), we get even uniqueness
in (2).
\end{observation}

\begin{discussion}
\label{a40}
In \ref{a43} below, we may consider (see \cite[Ch.IV,\S1]{Sh:f}):
\medskip

\noindent
\underline{Ax(V.1)}:  If $(q,G,L) \in \mathbf F,G \subseteq H \in \mathbf K;
\bar a,\bar b \in {}^{\omega >}H$; $q = \tp_{\bs}(\bar a \char 94
\bar b,G,H)$ and $p = \tp_{\bs}(\bar a,\langle G + \bar b\rangle_H,H)$
\then \, $(p,\langle G + \bar b \rangle_H,L) \in \mathbf F$.
\medskip

\noindent
\underline{Ax(V.2)}:  A notational variant of (V1) so we ignore it.
\end{discussion}

\noindent
The following claim (together with \S2, the existence of countable
dense $\gS$) proves Theorem \ref{y2}.
\begin{claim}
\label{a41}  
1) If $G \in \mathbf K$ is of cardinality $\le \lambda$ and $\gS
 \subseteq \Omega[\mathbf K]$ is closed and dense and of cardinality $\le
 \lambda$ (if $\lambda \ge 2^{\aleph_0}$ this follows) \then \, there is an
$\mathbf F^{\sch}_{\aleph_0}-\gS$-construction ${\cA}$ such that:
\mn
\begin{enumerate}
\item[$(a)$]  $\alpha^{\cA} = \lambda$;
\sn
\item[$(b)$]  $G^{\cA}_0 = G$;
\sn
\item[$(c)$]  $G^{\cA}_\lambda \in \mathbf K$ is existentially closed 
of cardinality $\lambda$;
\sn
\item[$(d)$]  $\cA$ is $\lambda$-full, that is
for every $\gs \in \gS$ and $\bar a \in
{}^{k(\gs)}(G^{\cA}_\lambda)$ realizing $p_{\gs}(\bar x)$, for 
$\lambda$ ordinals $\alpha < \lambda$ we have: 
{\rm tp}$_{\bs}(\bar a_\alpha,G^{\cA}_\alpha,G^{\cA}_{\alpha +1}) =
q_{\gs}(\bar a,G^{\cA}_\alpha)$.
\end{enumerate}
\mn
2) Assume $\lambda \ge \|G\| + |\gS|$ is regular.  \Then \, we can find $H
\in \mathbf K_\lambda$ which is $\mathbf
F^{\sch}_{\aleph_0}-\gS$-constructible over $G$, is
$(\lambda,\gS)$-full over $M$ and is embeddable over $M$ into any $N'$
which is $(\lambda,\gS)$-full over $G$, in fact $G_\lambda$ from part
(1) is as required.

\noindent
3) If $\gS$ is symmetric and is closed and $H_1,H_2$ are 
$\mathbf F^{\sch}_{\aleph_0}-\gS$-constructible over $G$ and
$(\lambda,\gS)$-full over $G$ and of cardinality $\lambda$ \then
 \, $H_1,H_2$ are isomorphic over $G$.

\noindent
4) If $\lambda \ge \|G\|$ and $\cA$ is an $\mathbf
   F^{\sch}_{\aleph_0}-\gS$-construction of $H$ over $G$ and $\ell
   g(\cA) = \lambda$ \then \, for every $H' \in \mathbf K$ which is 
$(\lambda,\gS)$-full over $G$, we have $H$ is embeddable into $H'$ over $G$.
\end{claim}

\begin{PROOF}{\ref{a41}}  
By \cite[Ch.VI,\S3]{Sh:c} as all the relevant 
axioms there apply (see below or
\cite[Ch.IV,\S1,pg.153]{Sh:c}) \underline{or} just check directly.  
Of course, we can use a monster $\gC$ for
groups, but use only sets $A$ such that $c \ell(A,\gC) = \langle A
\rangle_{\gC}$ is locally finite, and we use quantifier-free types. 
\end{PROOF}

\noindent
Now we make the connection to \cite[Ch.IV]{Sh:c}.
\begin{dc}
\label{a43}
1) Let $\gS \subseteq \Omega[\mathbf K]$ be closed and below let $\lambda =
   \lambda(\mathbf F_{\gS})$ be $\aleph_0$.  Then $\mathbf F = 
\mathbf F_{\gS}$ is defined as the set of triples $(p,G,A)$ such that $A$ is
finite, for some $B \subseteq G \in \mathbf K$ we have 
$A \subseteq B,c \ell(B) = c \ell(B,G) = G \in \mathbf K$, $p
\in \mathbf S^{< \omega}_{\gS}(c \ell(B))$ is $q_{\gs}(\bar b,c
   \ell(B))$ for some $\gs \in \gS,\bar b \subseteq c \ell(A)$ over
$A$; we may restrict ourselves to the case $B = c \ell(B,G) = G$.  Note
that: as here we do not have a monster model $\gC$ we can either 
demand $B \in \mathbf K$ 
\underline{or} demand $B \subseteq G \in \mathbf K$ but then
it is more natural to write $(p,G,A)$ instead of $(p,A)$.

\noindent
2) $\mathbf F$ satisfies the axioms (from \cite[Ch.IV,\S1]{Sh:c} 
written below in the present notation) except 
possibly V, VI, VIII, X.1, X.2, XI.1, XI.2.

\noindent
3) If $\gS$ is symmetric \then \, $\mathbf F$ satisfies also Ax(VI).

\noindent
4) If $\gS$ is dense \then \, $\mathbf F$ satisfies also Ax(X.1).
\end{dc}

\begin{remark}
\label{a44}
If $\gS$ is compact (see \ref{a21}(5)), \then \, 
$\mathbf F$ satisfies Ax(VIII), i.e. 

\underline{Ax(VIII) when $\gS$ is compact}:  
If $\langle G_i:i \le \delta +1 \rangle$ is
$\subseteq$-increasing continuous in $\mathbf K,L \subseteq G_0$ finite,
$p \in \mathbf S_{\gS}(G_\lambda)$ and $i < \delta \Rightarrow (p \rest
G_i,G_i,L) \in \mathbf F$ \then \, $(p,G_\delta,L) \in \mathbf F$.
\medskip

\noindent
[Why?  By the Definition; also holds when cf$(\delta) > \aleph_0$.]
\end{remark}

\begin{PROOF}{\ref{a43}}

\noindent
\underline{Isomorphism - Ax(I)}:  preservation under isomorphism.

Obvious.
\medskip

\noindent
Concerning trivial $\mathbf F$-types:

\noindent
\underline{Ax(II1)}:  If $K \subseteq L \subseteq G \in 
\mathbf K,|L| < \lambda,K$ is finite, $\bar a \in {}^{\omega >}K$ and
$p = \tp_{\bs}(\bar a,L,G)$ then $(p,G,K) \in \mathbf F$.

[Why?  Trivially; recall $\lambda = \aleph_0$.]
\medskip

\noindent
\underline{Axiom(II2)-(II3)-(II4)}:  irrelevant here.
\medskip

\noindent
Concerning monotonicity:

\noindent
\underline{Ax(III1)}:  If $L \subseteq G_1 \subseteq G_2$ and $(p,G_2,L)
\in \mathbf F$ then $(p \rest G_1,G_1,L) \in \mathbf F$.
\medskip

\noindent
[Why?  Because if $\bar a \in {}^{\omega >} L,L \subseteq G_1 
\subseteq G_2 \in \mathbf K$ and
$q_{\gs}(\bar a,G_2)$ is well defined and equal to $p$, \then \, 
$q_{\gs}(\bar a,G_1) =
q_{\gs}(\bar a,G_2) \rest G_1)$, see Claim \ref{a15}(1A).]
\medskip

\noindent
\underline{Ax(III2)}:  If $L \subseteq L_1 \subseteq G,|L_1| <
\lambda$, i.e. $L_1$ is finite and 
$(p,G,L) \in \mathbf F$ \then  \, $(p,G,L_1) \in \mathbf F$.
\medskip

\noindent
[Why?  By the definition.]
\medskip

\noindent
\underline{Ax(IV)}:   If $\bar a,\bar b \in {}^{\omega >}H,
L \subseteq G \subseteq H,(\text{tp}_{\bs}(\bar b,G,H),G,L) \in
\mathbf F$ and Rang$(\bar a) \subseteq \text{ Rang}(\bar b)$ \then \,
$(\text{tp}_{\bs}(\bar a,G,H),G,L) \in \mathbf F$.
\medskip

\noindent
[Why?  Straightforward as $\gS$ is domination closed, see Definition
\ref{a21}(1B).]
\medskip

\noindent
Concerning transitivity and symmetry:
\medskip

\noindent
\underline{Ax(VI)}:  ($\gS$ is symmetric).  If $G \subseteq H \in \mathbf
K,\bar a,\bar b \in {}^{\omega >}H$ and $L_1,L_2 \subseteq G$ are finite and
$(\tp_{\bs}(\bar b,\langle G + \bar a\rangle_H,H),\langle G + \bar a
\rangle_H,L_1) \in \mathbf F$ and $(\tp_{\bs}(\bar a,G,H),G,L_2) \in
\mathbf F$ \then \, ($\tp_{\bs}(\bar a,\langle G + \bar b\rangle_H,H),\langle G +
\bar b \rangle_H,L_1) \in \mathbf F$. 
\medskip

\noindent
[Why?  By $\gS$ being symmetric when we claim this axiom, i.e. in 
\ref{a43}(3).]
\medskip

\noindent
\underline{Ax(VII)}:   
If $G \subseteq H \in \mathbf K,\bar a,
\bar b \in {}^{\omega >}H,(\tp_{\bs}(\bar a,\langle G + \bar
b\rangle_H,H),\langle G + \bar b \rangle_H,L) \in \mathbf F$ and
$(\tp_{\bs}(\bar b,G,H),G,L) \in \mathbf F$ hence $L \subseteq G$ is
finite, \then \, $(\tp_{\bs}(\bar a
\char 94 \bar b,G,H),G,L) \in \mathbf F$.
\medskip

\noindent
[Why?  By $\gS$ being composition-closed, see Definition \ref{a21}(1A).]
\medskip

\noindent
Concerning continuity:
\medskip

\noindent
\underline{Ax(IX)}:  irrelevant as $\lambda = \aleph_0$.
\medskip

\noindent
Concerning existence:

\noindent
\underline{Ax(X.1)}:  If $L_1 \subseteq G \in \mathbf K,L_1 \subseteq
L_2$ finite, $\bar a \in {}^{\omega >}(L_2)$ \then \, for some $p$ extending
tp$_{\bs}(\bar a,L_1,L_2)$ and finite $L \subseteq G$ we have 
$(p,G,L) \in \mathbf F$, moreover \wilog \, $L=L_1$.
\medskip

\noindent
[Why?  By $\gS$ being dense.] 
\medskip

\noindent
\underline{Ax(X.2)}:  irrelevant and follows by the moreover in
 Ax(X.1).
\medskip

\noindent
\underline{Ax(XI.1)}:  If $p \in \mathbf S_{\bs}(G_1),(p,G_1,L) \in \mathbf
F$ hence $p \in \mathbf S^n_{\gS}(G_1)$ for some $n$ and 
$G_1 \subseteq G_2$ \then \, there is $q \in \mathbf S^n_{\bs}(G_2)$ extending
$p$ such that $(q,G_1,L_2) \in \mathbf F$ for $L_2$, so $\bar q \in
\mathbf S^n_{\gS}(G_2)$; moreover, in fact, 
$L_2 = L$ is O.K.
\medskip

\noindent
[Why?  Use the same $\gs \in \gS$.] 
\medskip

\noindent
\underline{Ax(XI.2)}:  irrelevant and really follows by the moreover
in (XI.1).
\end{PROOF}

\begin{definition}
\label{p33}
A sequence $\mathbf I = \langle \bar a_s:s \in I\rangle$ in $G \in
\mathbf K$ is $\kappa$-convergent \when \, for some $m,s \in I
\Rightarrow \bar a_s \in {}^m G$ and for every 
finite $K \subseteq G$ and some $q \in
\mathbf S^m(K)$ for all but $< \kappa$ members $s$ of $\mathbf I,q = 
\tp_{\bs}(\bar a_s,K,G)$.
\end{definition}

\begin{remark}
\label{p35}
1) So $\mathbf F_{\gS}$-constructions preserve ``$\mathbf
I$ is $\kappa$-convergent".  Moreover, if $\mathbf I$ is
$\kappa$-convergent in $G \in \mathbf K$ and $G \le_{\gS} H$, where $\gS
\subseteq \Omega[\mathbf K]$ \then \, $I$ is $\kappa$-convergent in $H$.

\noindent
2) We can assume $I$ is a linear order with no last member and of
cofinality $\ge \kappa$ and replace ``all
but $< \kappa$ of the $s \in I$" by ``every large enough $s \in I$".
See more in \cite[\S(1C)]{Sh:950}.
\end{remark}
\bigskip

\subsection {Using Order}\
\bigskip

We now turn to the third avenue of \S(0B) to
deal with the general and not necessarily symmetric case.
Can we get uniqueness for non-symmetric $\gS$?  Can we get every automorphism
extendable, etc.?  The answer is that at some price, yes.  A
major point in the construction was the use of linear well ordered index set 
($\lambda$ in \ref{a41}(1) or $\alpha^{\cA}$ in general).
But actually we can use linear non-well ordered index sets, so those index sets
can have automorphisms which help us toward uniqueness.  The solution
here is not peculiar to locally finite groups.

\begin{definition}
\label{a47}
We say $(I,E)$ is $\lambda$-suitable \when \, (we may
omit $\lambda$ when $\lambda = |I|$, we may write $(I,P_i)_{i <
\lambda}$ with $\langle P_i:i < \lambda\rangle$ listing the
$E$-equivalence classes (with no repetitions)):
\mn
\begin{enumerate}
\item[$(a)$]  $I$ is a linear order;
\sn
\item[$(b)$]  $E$ is an equivalence relation on $I$ with $\lambda$
equivalence classes;
\sn
\item[$(c)$]  every permutation of $I/E$ is induced by some
automorphism of the linear order which preserves equivalence and
non-equivalence by $E$;
\sn
\item[$(d)$]  each $E$-equivalence class has cardinality $|I|$.
\end{enumerate}
\end{definition}

\begin{claim}
\label{a51}
Let $T = \Th(\bbR,<,E)$ where Th stands for ``the first order theory of",
$E := \{(a,b):a,b \in \bbR$ and $a-b \in \bbQ\}$; so $(A,<,E) \models
T$ iff $(A,<)$ is a dense linear order with neither first nor last
element, $E$ an equivalence relation with each equivalence class a
dense subset of $A$ and with infinitely many equivalence classes.

\noindent
1) If $\lambda = \lambda^{<\lambda}$ and $(I,E)$ is a
saturated model of $T$ of cardinality $\lambda$,
\then \, $(I,E)$ is suitable\footnote{By similar arguments, if 
$\lambda \ge 2^\mu$ \then \, there is a $\mu$-suitable
 $(I,P_i)_{i < \mu}$ but $|I/E| = \mu < \lambda$.  
We can use any model of cardinality $\lambda$ which
is strongly $\mu^+$-sequence homogeneous; this means that
every partial automorphism of cardinality $\le \mu$ can be extended to
an automorphism.}

\noindent
2) For every $\lambda$ the $(I,P_i)_{i < \lambda}$ from \cite[\S2]{Sh:E62}
(see history there) is $\lambda$-suitable and $|I|=\lambda$.

\noindent
3) There is a definable sequence 
$\langle (I_\lambda,P^\lambda_i)_{i < \lambda}:
\lambda$ an infinite cardinal$\rangle$ such that
 $(I_\lambda,P^\lambda_i)_{i < \lambda}$ is $\lambda$-suitable and is
   increasing with $\lambda$ and this definition is absolute.
\end{claim}

\begin{PROOF}{\ref{a51}}
1) Obvious.

\noindent
2),3) See there.
\end{PROOF}

\begin{claim}
\label{a58}  Assume
\mn
\begin{enumerate}
\item[(A)]   $G \in \mathbf K$ is of cardinality $\lambda$;
\sn
\item[(B)]   $\gS \subseteq \Omega[\mathbf K]$ is closed and dense;
\sn
\item[(C)]
\begin{enumerate}
\item[(a)]  $i(*) \le \lambda$ and 
$\cS = \{t_i = ({\gs}_i,\bar a_i):i < i(*)\}$ lists 
$\deef_{\gS}(G)$, i.e. the pairs
$({\gs},\bar a)$, as in clause (d) 
of \ref{a41}(1) or \ref{a41}(3);
\sn
\item[(b)]   each such pair appears exactly once;
\sn
\item[(c)]   let $t_i = (\gs_*,<>)$ for $i \in 
[i(*),\lambda)$ so $\gs_* \in \gS,k_{\gs_*} = 0,n_{\gs}=1,
i(*) = \|\deef_{\gS}(G)\|,\gs_*$ is from \ref{a16}; so $\bar a_i =
\langle \rangle$;
\end{enumerate}
\sn
\item[(D)]   $(I,P_i)_{i < \lambda}$ is $\lambda$-suitable, see
  Definition \ref{a47}.
\end{enumerate}
\mn
\Then \, we can find $H,\mathbf c = 
\langle \bar c_r:r \in I \rangle$ (the ordered one step 
$(\lambda,\gS)$-closure), such that:
\mn
\begin{enumerate}
\item[$(a)$]   $H \in \mathbf K$ is a $\le_{\gS}$-extension of $G$;
\sn
\item[$(b)$]  $\bar c_r \in {}^{n(t_i)} H$ if $r \in P_i,i < \lambda$;
\sn
\item[$(c)$]   $H$ is generated by $G \cup \{\bar c_r:r \in I\}$;
\sn
\item[$(d)$]   if $i < \lambda$ and 
$r \in P_i$ then $\bar c_r$ realizes in $H$ over
$c \ell(G \cup\{\bar c_s:s <_I r\},H)$ the type defined by
$(\gs_i,\bar a_i)$;
\sn
\item[$(e)$]  every automorphism of $G$ can be extended to an
automorphism of $H$.
\end{enumerate}
\end{claim}

\begin{PROOF}{\ref{a58}}
Straightforward; e.g. to define $H$ we should choose 
$q_{r_0,\dotsc,r_{n-1}}$ for every $r_0 <_I \ldots <_I r_{n-1}$
by induction on $n$ such that in the end $q_{r_0,\dotsc,r_{n-1}}
 = \tp_{\bs}(\bar c_{r_0} \char 94 \ldots
\char 94 \bar c_{r_{n-1}},G,H)$, by clause (d), and prove that:
\mn
\begin{enumerate}
\item[$(*)$]   if $m \le n$ and
$h:\{0,\dotsc,m-1\} \rightarrow \{0,\dotsc,n-1\}$ is increasing then
$q_{r_{h(0)},\dotsc,r_{h(m-1)}} \le_h q_{r_0,\dotsc,r_{n-1}}$.
\end{enumerate}
\mn
Note that clause (e) follows by clauses (a)-(d) above recalling clause
(c) of Definition \ref{a47}.

Why?  Let $\pi$ be an automorphism of $G$, for each $i < \lambda$ we
have $(\gs_i,\bar a_i) \in \mathscr{S}$ and also $(\gs_i,\pi(\bar a_i)) 
\in \mathscr{S}$, so by the choice of $\langle (\gs_i,\bar a_i):
i < \lambda\rangle$ there is a unique $j < \lambda$ such
that $i \ge i(*) \Rightarrow j = i$ and $(\pi(\bar a_i),\gs_i) 
= (\bar a_j,\gs_j)$, so let $j =
\hat\pi(i)$.  So $\hat\pi$ is a permutation of $\lambda$.  By
``$(I,P_i)_{i < \lambda}$ is $\lambda$-suitable" there is an
automorphism $\check\pi$ of the linear order $I$ such that $i <
\lambda \Rightarrow \check\pi(P_i) = P_{\hat\pi(j)}$.  Clearly there
is a unique automorphism $\dot\pi$ of $H$ such that $\pi = \dot\pi
\rest G$ and $\dot\pi(\bar c_i) = \bar c_{\hat\pi(i)}$. 
\end{PROOF}

\begin{definition}
\label{a62}
1) We say $H$ is an ordered one-step $(\lambda,\gS)$-closure of 
$G$, pedantically the ordered one step $(I,E)-\gS$-closure of $G$,
\when \, $G,H,\mathbf c$ are as in \ref{a58}.

\noindent
2) We say $H$ is an ordered $(\lambda,\gS)$-closure of $G$,
pedantically the ordered $(I,E)-\gS$-closure of $G$ when:
\mn
\begin{enumerate}
\item[$(a)$]  $H = \bigcup\limits_{n} H_n$
\sn
\item[$(b)$]  $H_0 = G$
\sn
\item[$(c)$]  $H_{n+1}$ is the one step $(I,E)-\gS$-closure of $H_n$.
\end{enumerate}
\end{definition}

\begin{remark}
\label{a61b}  
1) In what way is \ref{a58} weaker?  We have to choose the listing of
def$(G)$ in clause (C).  Also for $G_1 \subseteq G_2$ it is not clear
why $H_1 \subseteq H_2$, where $(G_\ell,H_\ell)$ is as above.  
But see \ref{a51}(3).

\noindent
2) On naturality see Paolini-Shelah \cite{Sh:1106}.
\end{remark}

\begin{conclusion}
\label{a68}
The parallel of parts (2)-(6) of \ref{a26} holds.
\end{conclusion}

\begin{PROOF}{\ref{a68}}
Straightforward, for part (6) of \ref{a26} use \ref{a51}(3).
\end{PROOF}
\newpage
 
\section {There are enough reasonable schemes}
\bigskip

\subsection {There is a Dense Set of Schemes}\
\bigskip

We like to find $\gS$'s as in \S1 for $\mathbf K_{\lf}$, in
particular to prove that there are dense $\gS$,  so we have to look 
in details at amalgamations of $\lf$-groups under special assumptions.

Recall the well known: for finite groups $G_0 \subseteq G_\ell \in \mathbf K$ for
$\ell=1,2$ we can amalgamate $G_1,G_2$ over $G_0$ by embedding 
into suitable finite permutations group; see the proof of the theorem of
Hall, explained in the second paragraph of \S(0A). 

Concerning the $\mathbf K_{\olf}$ versions of \ref{c3}, 
see later in \ref{n17}.

\begin{convention}
\label{c1}
$\mathbf K$ is $\mathbf K_{\lf}$.
\end{convention}

\begin{definition}
\label{c3}
1) Let $\mathbf X_{\mathbf K} = \mathbf X(\mathbf K)$, the 
set of amalgamation tries, be the set of
$\mathbf x$ such that: $\mathbf x$ is a quintuple 
$(G_0,G_1,G_2,\mathbf I_1,\mathbf I_2) = (G_{\mathbf x,0},
G_{\mathbf x,1},\ldots)$ satisfying:
\mn
\begin{enumerate}
\item[$(a)$]  $G_0 \subseteq G_\ell \in \mathbf K$ for $\ell=1,2$;
\sn
\item[$(b)$]  $\mathbf I_\ell$ is a set of representatives of the left
$G_0$-cosets in $G_\ell$, i.e. $\langle g G_0:g \in \mathbf I_\ell\rangle$
is a partition of $G_\ell$ (so without repetitions) for $\ell=1,2$;
\sn
\item[$(c)$]  $e_{G_{\mathbf x,0}} \in \mathbf I_{\mathbf x,1} \cap \mathbf
I_{\mathbf x,2}$.
\end{enumerate}
\mn
2) For $\mathbf x$ as above let
\mn
\begin{enumerate}
\item[$(a)$]  $\cU = \cU_{\mathbf x} = \{(g_0,g_1,g_2):
g_\ell \in G_\ell$ for $\ell=0,1,2$ and $g_1 \in \mathbf I_1,g_2
\in \mathbf I_2\}$;
\sn
\item[$(b)$]  for $\ell=1,2$ and $g \in \mathbf I_\ell$ let
$\cU^\ell_g = \cU^\ell_{\mathbf x,g} := \{(g_0,g_1,g_2) \in \cU_{\mathbf
x}:g_\ell = g\}$;
\sn
\item[$(c)$]  if $G_1 \cap G_2 = G_0$ then we let
$\mathbf j_{\mathbf x} = \mathbf j_{\mathbf x,1} \cup \mathbf
j_{\mathbf x,2}$, see below;
\sn
\item[$(d)$]  for $\ell = 0,1,2$ let $\mathbf j_\ell = \mathbf j_{\mathbf
x,\ell}$ be the following embedding of $G_\ell$ into per$(\cU_{\mathbf x})$, the
group of permutations of $\cU_{\mathbf x}$, so let $g \in G_\ell$ and we
should define $\mathbf j_\ell(g)$, so let $(g_0,g_1,g_2) \in \cU_{\mathbf
x}$ and we define $(g'_0,g'_1,g'_2) = (\mathbf
j_\ell(g))(g_0,g_1,g_2)$ from $\cU_{\mathbf x}$ as follows:
\end{enumerate}
\bigskip

\noindent
\underline{$\ell=0$}:  $g'_0 = g_0 g$ in $G_0$ and $g'_1 = g_1,g'_2= g_2$;
\medskip

\noindent 
\underline{$\ell=1$}:  $g'_1 g'_0 = g_1 g_0 g$ in $G_1$ and $g'_2 = g_2$;
\medskip

\noindent 
\underline{$\ell=2$}:  $g'_2 g'_0 = g_2 g_0 g$ in $G_2$ and $g'_1 = g_1$.
\medskip

\noindent
3) Let $G_{\mathbf x} = G_{\mathbf x,3}$
 be the subgroup of $\Sym(\cU_{\mathbf x})$ which
Rang$(\mathbf j_{\mathbf x,1}) \cup \text{ Rang}(\mathbf j_{\mathbf x,2})$
generates where $G_{\mathbf x} \models ``f_1 f_2 = f_3"$ means that for
every $u \in \cU_{\mathbf x},f_3(u) = f_2(f_1(u))$, i.e. we look at the
permutation as acting from the right.

\noindent
4) Let $\le_{\mathbf X(\mathbf K)}$ be the following partial order on
$\mathbf X_{\mathbf K}:\mathbf x \le_{\mathbf X(\mathbf K)} \mathbf y$ \Iff \,:
\mn
\begin{enumerate}
\item[(a)]  $\mathbf x,\mathbf y \in \mathbf X_{\mathbf K}$;
\sn
\item[(b)]  $G_{\mathbf x,0} = G_{\mathbf y,0}$;
\sn
\item[(c)]  $G_{\mathbf x,\ell} \subseteq G_{\mathbf y,\ell}$ for
  $\ell=1,2$;
\sn
\item[(d)]  $\mathbf I_{\mathbf x,\ell} = \mathbf I_{\mathbf y,\ell} \cap
  G_{\mathbf x,\ell}$ for $\ell=1,2$.
\end{enumerate}
\mn
5) We say $(f_1,f_2)$ embeds $\mathbf x \in \mathbf X_{\mathbf K}$ into
$\mathbf y \in \mathbf X_{\mathbf K}$ \when \,:
\mn
\begin{enumerate}
\item[(a)]  $f_\ell$ embeds $G_{\mathbf x,\ell}$ into $G_{\mathbf y,\ell}$
  for $\ell=1,2$;
\sn
\item[(b)]  $f_1 \rest G_{\mathbf x,0} = f_2 \rest G_{\mathbf y,0}$ maps
  $G_{\mathbf x,0}$ onto $G_{\mathbf y,0}$.
\end{enumerate}
\mn
6) We say $(f_1,f_2)$ is an isomorphism from $\mathbf x \in \mathbf
X_{\mathbf K}$ onto $\mathbf y \in \mathbf X_{\mathbf K}$ \when \, above
$f_\ell$ is onto $G_{\mathbf y,\ell}$ for $\ell=1,2$.
\end{definition}

\begin{observation}
\label{c6}
Let $\mathbf x$ be as in Definition \ref{c3}, i.e. it is an amalgamation try.

\noindent
0) If $G_0 \subseteq G_\ell \in \mathbf K$ for $\ell=1,2$ 
\then \, for some $\mathbf x \in \mathbf X_{\mathbf K}$ we have
   $G_{\mathbf x,\ell} = G_\ell$ for $\ell=0,1,2$.

\noindent
1) In Definition \ref{c3}(2), for $\ell=0,1,2$ if $g \in G_{\mathbf
 x,\ell}$ \then \, $\mathbf j_{\mathbf x,\ell}(g)$ is a permutation of 
$\cU_{\mathbf x}$, in fact, its restriction to $\cU^{3-\ell}_{g_1}$ is a
permutation for each $g_1 \in G_{3-\ell}$.

\noindent
2) Moreover in part (1) the mapping $\mathbf j_{\mathbf x,\ell}$ embeds
   the group $G_{\mathbf x,\ell}$ into the group of permutation of
   $\cU_{\mathbf x}$ hence into $G_{\mathbf x}$.

\noindent
3) The mapping $\mathbf j_{\mathbf x,0}$ is equal to $\mathbf j_{\mathbf x,1} \rest
   G_{\mathbf x,0}$ and also to $\mathbf j_{\mathbf x,2} \rest G_{\mathbf x,0}$.

\noindent
4) If $G_{\mathbf x,\ell}$ is finite for $\ell=0,1,2$ \then \, $|G_{\mathbf
   x}| \le (|G_{\mathbf x,1}| \times |G_{\mathbf x,2}|/|G_{\mathbf x,0}|)!$ 

\noindent
5) If $\mathbf x$ is an amalgamation try and $G_{\mathbf x,0} \subseteq
   G'_\ell \subseteq G_{\mathbf x,\ell}$ so $G'_\ell$ is a subgroup of
   $G_{\mathbf x,\ell}$, for $\ell=1,2$ \then \, for one and
   only one amalgamation try $\mathbf y$ we have $G_{\mathbf y,0} =
   G_{\mathbf x,0},G_{\mathbf y,\ell} = G'_\ell$ for $\ell=1,2$ and $\mathbf
I_{\mathbf y,\ell} = \mathbf I_{\mathbf x,\ell} \cap G'_\ell$ so $\mathbf y
\le_{\mathbf X(\mathbf K)} \mathbf x$.

\noindent
6) Moreover in part (5), if $\mathbf z$ is an amalgamation try with $(G_{\mathbf
z,0},G_{\mathbf z,1},G_{\mathbf z,2}) = (G_{\mathbf x,0},G'_1,G'_2)$
   \then \, for some $\mathbf x'$, the pair $(\mathbf x',\mathbf z)$ 
is like $(\mathbf x,\mathbf y)$ in (5) and 
$(G_{\mathbf x',0},G_{\mathbf x',1},G_{\mathbf x',2}) = 
(G_{\mathbf x,0},G_{\mathbf x,1},G_{\mathbf x,2})$.

\noindent
7) In part (5) there is a unique homomorphism $f$ from 
$\langle \mathbf j_{\mathbf x,1}(G'_1) \cup \mathbf j_{\mathbf x,2}(G'_2)
   \rangle_{\Sym(\cU_{\mathbf x})}$ onto $G_{\mathbf y}$ such that $\ell \in
\{1,2\} \wedge g \in G'_\ell \Rightarrow \mathbf j_{\mathbf y,\ell}(g)
 = f(\mathbf j_{\mathbf x,\ell}(g))$.

\noindent
8) In part (5), if $G'_1,G'_2$ are finite then
$\langle \mathbf j_{\mathbf x,1}(G'_1) \cup \mathbf j_{\mathbf
   x,2}(G'_2)\rangle_{G_{\mathbf x}}$ has at most $(n_*!)^{m_*}$ members where
$n_* = |G'_1| \times |G'_2| \times |G_{\mathbf
  x,0}|^3$ and $m_* = (n_*!)^{|G'_1| + |G'_2|}$.
\end{observation}

\begin{PROOF}{\ref{c6}}  
Straightforward.  E.g.:

\noindent
2) E.g. let $\ell=1$ and $f,h \in G_1$.  For 
$(g_0,g_1,g_2) \in \cU_{\mathbf x}$ let $(\mathbf j_1(f))(g_0,g_1,g_2) =
(g'_0,g'_1,g'_2)$ and $(\mathbf j_1(h))(g'_0,g'_1,g'_2) =
(g''_0,g''_1,g''_2)$ hence
\mn
\begin{enumerate}
\item[$(*)_1$]  $(\mathbf j_1(h))(\mathbf j_1(f))(g_0;g_1,g_2) =
  (g''_0,g''_1,g''_2)$.
\end{enumerate}
\mn
Then $g_2 = g'_2$ and $g'_2 = g''_2$ and in
$G_1$ we have 
$g_1 g_0 f = g'_1 g'_0$ and $g'_1 g'_0 h = g''_1 g''_0$, hence $g_2
   = g''_2$ and $g''_1 g''_0 = g'_1 g'_0 h = (g_1 g_0 f)h = 
(g_1 g_0)(fh)$, so by the definition of $\mathbf j_1(f_h)$ we have
\mn
\begin{enumerate}
\item[$(*)_2$]   $\mathbf j_1(fh)(g_0,g_1,g_2) = (g''_0,g''_1,g''_2)$.
\end{enumerate}
\mn
By $(*)_1 + (*)_2$ we have
\mn
\begin{enumerate}
\item[$(*)_3$]  $\mathbf j_1(f_h)(g_0,g_1,g_2) =
(\mathbf j_1(h))(\mathbf j_1(f))(g_0,g_1,g_2)$.
\end{enumerate}
\mn
As this holds for every $(g_0,g_1,g_2) \in \cU_{\mathbf x}$ we have
$G_{\mathbf x} \models ``\mathbf j_1(fh) = \mathbf j_1(f) \mathbf j_1(h)"$.

\noindent
4) Clearly $|G_{\mathbf x,\ell}| = |\mathbf I_{\mathbf x,\ell}| \times
|G_{\mathbf x,0}|$ for $\ell=1,2$ hence $|\cU_{\mathbf x}| = |\mathbf I_{\mathbf x,1}|
\times |\mathbf I_{\mathbf x,2}| \times |G_{\mathbf x,0}| = (|G_{\mathbf
  x,1}|/|G_{\mathbf x,0}) \times (|G_{\mathbf x,2}|/|G_{\mathbf x,0}| \times
|G_{\mathbf x,0}| = |G_{\mathbf x,1}| \times |G_{\mathbf x,2}|/|G_{\mathbf x,0}|$.

Hence $|G_{\mathbf x}| \le |\Sym(\cU_{\mathbf x})| = (|\cU_{\mathbf x}|)! 
= (|G_{\mathbf x,1}| \times |G_{\mathbf x,2}|/|G_{\mathbf x,0}|)!$ as stated).

\noindent
7) First, why there is such a homomorphism?  If $b \in \mathbf j_{\mathbf
  x,1}(G'_1) \cup \mathbf j_{\mathbf x,2}(G'_2)$ then $b$ is a permutation of
$\cU_{\mathbf x}$ which maps the set $\cU_{\mathbf y} = G_0 \times \mathbf
I_{\mathbf y,1}  \times \mathbf I_{\mathbf y,2}$ onto itself.  It follows
that every $b \in G' := \langle \mathbf j_{\mathbf x,1}(G'_1) \cup \mathbf
j_{\mathbf x,2}(G'_2)\rangle_{\Sym(\cU_{\mathbf x})}$ maps the set $\mathbf
U_{\mathbf y} = G_0 \times \mathbf I_{\mathbf y,1} \times \mathbf I_{\mathbf
  y,2}$ onto itself.  Hence the mapping with domain $G'$ defined by
$f(b) = b \rest \cU_{\mathbf y}$ is a homomomorphism from $G'$ into
$\Sym(\cU_{\mathbf y})$.  However, for each $b \in \mathbf j_{\mathbf
  x,1}(G'_1) \cup \mathbf j_{\mathbf x,2}(G'_2)$ we have $b \rest
\cU_{\mathbf y}$ belongs to $G_{\mathbf y,3}$ so $b \in G' \Rightarrow
f(b) \in G_{\mathbf y,3}$, hence $f$ is as required.

Second, why $f$ is unique?  Because $\mathbf j_{\mathbf x,1}(G'_1) \cup
\mathbf j_{\mathbf x,2}(G'_2)$ generates $G'$ and on it $f$ is determined.

\noindent
8) Let $G_0 = G_{\mathbf x,0}$. 
We define $E = \{((g'_0,g'_1,g'_2),(g''_0,g''_1,g''_2))) \in
   \cU_{\mathbf x} \times \cU_{\mathbf x}:G_0 g'_1 G'_1 = G_0 g''_1 G'_1$ and
   $G_0 g'_2 G'_2 = G_0 g''_2 G'_2\}$, this is an equivalence relation on
$\cU_{\mathbf x}$, each equivalence class has $\le (|G'_1| \times |G'_2|
   \times |G_{\mathbf x,0}|^3) = n_*$ members.

[Why?  As if $(g'_0,g'_1,g'_2) \in (g_0,g_1,g_2)/E$ then $g'_0 \in
  G_0,g'_1 \in G_0 g_1 G'_1,g'_2 \in G_0 g_2 G'_2$ and $|G_0 g_\ell
 G'_\ell| \le |G_0| \times |G'_\ell|$.]

Also each of the permutations of $\cU_{\mathbf x}$ from
 $\mathbf j_{\mathbf x,1}(G'_1) \cup \mathbf j_{\mathbf
   x,2}(G'_2)$ maps each $E$-equivalence class onto itself.  Hence for
$n \in [1,n_*]$ there are $\le m^*_n := n!^{|G'_1|+|G'_2|-1}$ 
isomorphism types of structures of the form: 
$N = (|N|,F^N_f)_{f \in G'_1 \cup G'_2}$, where
   $|N|$, the universe, has exactly $n$ elements and 
is an $E$-equivalence class, and
   for each $f \in G'_1 \cup G'_2$ we have: $F^N_f$ is a permutation of this
   equivalence class and $F^N_{e(G_0)}$ is the identity.  Clearly as
$\sum\limits_{n \le n_*} (n!)^{|G'_1|+|G'_2|-1} \le 
 (n_*!)^{|G'_1|+|G'_2|} = m_*$, the subgroup 
$\langle \mathbf j_{1,\mathbf x}(G'_1) \cup \mathbf j_{2,\mathbf
   x}(G'_2)\rangle_{G_{\mathbf x}}$ of $G_{\mathbf x}$ has at most
$(n_*!)^{m_*}$ members.  Of course\footnote{See more in
  \cite{Sh:1098}.}, the argument gives better bounds,
   e.g. the number of relevant $N$'s is much smaller and using a finer $E$.
\end{PROOF}

\begin{claim}
\label{c8}
In Definition \ref{c3}, $\mathbf j_{\mathbf x,1}(G_1) \cap \mathbf j_{\mathbf
x,2}(G_2) = \mathbf j_{\mathbf x,\ell}(G_0)$.
\end{claim}

\begin{PROOF}{\ref{c8}}
Assume that $a_\ell \in G_\ell$ and $b_\ell = \mathbf j_{\mathbf x,\ell}
(a_\ell)$ for $\ell=1,2$.  It suffices to show that: if 
$b_1 = b_2$ \then \, $a_1,a_2 \in G_{\mathbf x,0}$ and $a_1 = a_2$.  We 
check to what $b_\ell$ maps the triple $(e,e,e) \in
\cU_{\mathbf x}$: by the definition of $\mathbf j_{\mathbf x,1},\mathbf
j_{\mathbf x,2}$ we have:
\mn
\begin{itemize}
\item  $b_1((e,e,e)) = (g'_0,g_1,e) \in\cU_{\mathbf x}$
where $G_1 \models g_1 g'_0 = b_1$;
\sn
\item  $b_2((e,e,e)) = (g''_0,e,g_2) \in
\cU_{\mathbf x}$ where $G_2 \models g_2 g''_0 = b_2$.
\end{itemize}
\mn
So if $b_1=b_2$ then $(g'_0,g_1,e) = b_1((e,e,e) = b_2((e,e,e,)) =
(g''_0,e,g_2)$, hence $g'_0 = g''_0 \wedge 
g_1 = e \wedge e = g_2$; this implies that 
$g'_0 = b_1,g''_0 = b_2$ hence $g'_0
= g''_0$, also $g''_0 \in G_0$ together $a_1 = a_2$ so we are done.
\end{PROOF}

\begin{definition}
\label{c10}
1) Let\footnote{NF stands for non-forking.}
NF$_{\text{rfin}}(G_0,G_1,G_2,G_3)$ means that $G_\ell
   \subseteq G_3 (\in \mathbf K)$ for $\ell < 3$ and
   NF$_{\text{fin}}(G_0,G_1,G_2,\langle G_1 \cup G_2\rangle_{G_3})$,
   see below.

\noindent
2) Let NF$_{\text{fin}}(G_0,G_1,G_2,G_3)$ mean that:
\mn
\begin{enumerate}
\item[$(a)$]  $G_0 \subseteq G_\ell \subseteq G_3 \in \mathbf K$ 
are finite groups for $\ell=1,2$;
\sn
\item[$(b)$]  $G_3 = \langle G_1 \cup G_2\rangle_{G_3}$;
\sn
\item[$(c)$]  if $\mathbf x \in \mathbf X_{\mathbf K}$ and $G_0 = G_{\mathbf x,0},G_1
\subseteq G_{\mathbf x,1},G_2 \subseteq G_{\mathbf x,2}$ \then \, there is
a homomorphism $\mathbf f$ from $G_3$ into $G_{\mathbf x}$ such that
$\mathbf f \rest G_\ell = \mathbf j_{\mathbf x,\ell} \rest G_\ell$ for $\ell=1,2$;
\sn
\item[$(d)$]  if $a \in G_3 \backslash \{e_{G_3}\}$ \then \, for some
$\mathbf x,\mathbf f$ as above we have $\mathbf f(a) \ne e_{G_3}$.
\end{enumerate}
\end{definition}

\begin{remark}
Note the choice ``$G_\ell \subseteq G_{\mathbf x,\ell}"$ rather than
$G_\ell = G_{\mathbf x,\ell}$ in clause (c) of \ref{c10}.
\end{remark}

\noindent
Now the amalgamation in Definition \ref{c10} is very nice but do we
have existence, in $\mathbf K_{\lf}$ of course?  
The following Claim \ref{c12}(3) answers positively.
\begin{claim}
\label{c12}
1) In clause (c) of Definition \ref{c10}(2), the homomorphism $\mathbf f$ is
   unique.

\noindent
1A) If $\NF_{\fin}(G^\iota_0,G^\iota_1,G^\iota_2,G^\iota_3)$ for
$\iota = 1,2$ and $\mathbf f_\ell$ is an isomorphism from $G^1_\ell$ onto
$G^2_\ell$ such that $\mathbf f_0 \subseteq \mathbf f_\ell$ for
$\ell=0,1,2$ \then \, there is one and only one isomorphism $\mathbf
f_3$ from $G^1_3$ onto $G^2_3$ extending $\mathbf f_1 \cup \mathbf f_2$.

\noindent
2) In Definition \ref{c10}, necessarily $G_1 \cap G_2 = G_0$.

\noindent
3) If $G_0 \subseteq G_\ell \in \mathbf K$ are finite for $\ell=1,2$ \then \,
 we can find $\bar f,\bar H$ such that
\mn
\begin{enumerate}
\item[$(a)$]  $\bar f = \langle f_0,f_1,f_2\rangle$;
\sn
\item[$(b)$]  $\bar H = \langle H_\ell:\ell \le 3\rangle$;
\sn
\item[$(c)$]  {\rm NF}$_{\fin}(H_0,H_1,H_2,H_3)$;
\sn
\item[$(d)$]   $f_\ell$ is an isomorphism from $G_\ell$ onto $H_\ell$
for $\ell=0,1,2$;
\sn
\item[$(e)$]  $f_0 \subseteq f_1$ and $f_0 \subseteq f_2$.
\end{enumerate}
\end{claim}

\begin{PROOF}{\ref{c12}}
1), 1A)  Obvious.

\noindent
2) By Claim \ref{c8} recalling clause (c) of \ref{c10}(2).

\noindent
3)  Follows by \ref{c6}(8) but we elaborate. 
Let $\bar G = \langle G_\ell:\ell=0,1,2\rangle$ and
\mn
\begin{enumerate}
\item[$(*)_1$]   let $\mathbf X_{\bar G} := \{\mathbf x 
\in \mathbf X_{\mathbf x}:G_{\mathbf x,0} = G_0$ and
$G_{\mathbf x,\ell}$ is a $\lf$ group extending $G_\ell$ for
$\ell=1,2\}$;
\sn
\item[$(*)_2$]   for $\mathbf x \in \mathbf X_{\bar G}$ let: $\mathbf
  n_{\bar g}(\mathbf x) =$ the number of elements of $\langle \mathbf
  j_{\mathbf x,1}(G_1) \cup \mathbf j_{\mathbf x,2}(G_2)\rangle_{G_{\mathbf
      x,3}}$.
\end{enumerate}
\mn
We define $\mathbf X^{\mx}_{\bar G}$ as the set of $\mathbf x$'s such that:
\mn
\begin{enumerate}
\item[$(*)_3$]
\begin{enumerate}
\item[(a)]  $\mathbf x \in \mathbf X_{\bar G}$;
\sn
\item[(b)]  if $\mathbf y \in \mathbf X_{\bar G}$ and $\mathbf x \le \mathbf
  y$ then $n_{\bar G}(\mathbf x) = n_{\bar G}(\mathbf y)$;
\end{enumerate}
\sn
\item[$(*)_4$]  if $\mathbf x,\mathbf z \in \mathbf X^{\mx}_{\bar G}$ and 
$\mathbf x \le_{\mathbf X(\mathbf K)} \mathbf z$, \then \, $n_{\bar G}(\mathbf x) \le
n_{\bar G}(\mathbf z)$.
\end{enumerate}
\mn
[Why?  Because by \ref{c6}(7) there is a homomorphism from $G_{\mathbf
  z} = \langle \mathbf j_{\mathbf z,1}(G_1) \cup \mathbf
j_{\mathbf z,2}(G_2)\rangle$ onto $G_{\mathbf x} =
\langle \mathbf j_{\mathbf x,1}(G_1) \cup
\mathbf j_{\mathbf x,2}(G_2)\rangle$.]
\mn
\begin{enumerate}
\item[$(*)_5$]  for every $\mathbf x \in \mathbf X_{\bar G}$ there is
  $\mathbf y \in \mathbf X^{\mx}_{\bar G}$ such that $x \le \mathbf y$;
  hence $\mathbf X^{\mx}_{\bar G} \ne \emptyset$.
\end{enumerate}
\mn
[Why?  By $(*)_4$ and \ref{c6}(8).]
\mn
\begin{enumerate}
\item[$(*)_6$]  $(\mathbf X_{\bar G},\le_{\mathbf X[\mathbf K]})$ has
  amalgamation, that is
\sn
\begin{itemize}
\item  if $\mathbf x_0 \le_{\mathbf X[\mathbf K]} \mathbf x_\iota$ for $\iota
  = 1,2$ \then \, we can find $\mathbf x_3$ and $(f^\iota_1,f^\iota_2)$ for
$\iota = 1,2$ such that:
\begin{enumerate}
\item[(a)]  $\mathbf x_3 \in \mathbf X_{\mathbf K}$
\sn
\item[(b)]  $\mathbf x_0 \le_{\mathbf X[\mathbf K]} \mathbf x_3$
\sn
\item[(c)]  $(f^\iota_1,f^\iota_2)$ embeds $\mathbf x_\iota$ into $\mathbf
  x_3$ over $\mathbf x_0$
\newline
(over $\mathbf x_0$ means: $f^\iota_1 \rest G_{\mathbf x_0,1} =
 \id_{G_{\mathbf x_0,1}},f^\iota_2 \rest G_{\mathbf x_0,2} =
\id_{G_{\mathbf x_0,2}}$).
\end{enumerate}
\end{itemize}
\end{enumerate}
\mn
[Why?  For $\ell=1,2$, we use the disjoint amalgamation for finite
groups, i.e. find $(G_\ell,f^1_\ell,f^2_\ell)$ such that:
\mn
\begin{enumerate}
\item[$\bullet_1$]  $G_\ell$ is a finite group extending $G_{\mathbf x_0,0}$
\sn
\item[$\bullet_2$]  $f^1_\ell$ embeds $G_{\mathbf x_1,\ell}$ into
  $G_\ell$ over $G_{\mathbf x_0,0}$
\sn
\item[$\bullet_3$]  $f^2_\ell$ embeds $G_{\mathbf x_2,\ell}$ into
 $G_\ell$ over $G_{\mathbf x_0,0}$
\sn
\item[$\bullet_4$]  $f^1_\ell(G_{\mathbf x_1,\ell}) \cap
  f^2_\ell(G_{\mathbf x,\ell}) = G_{\mathbf x_0,0}$.
\end{enumerate}
\mn
Note that $f^1_\ell(\mathbf I_{\mathbf x_1,\ell}) \cap f^2_\ell(\mathbf
I_{\mathbf x_2,\ell}) = \{e_{G_{\mathbf x_0,0}}\}$, moreover, working
inside $G_\ell,\langle g G_{\mathbf x_0,0}:g \in 
f^1_\ell(\mathbf I_{\mathbf x_1,\ell}) \cup
f^2_\ell(\mathbf I_{\mathbf x_2,\ell})\rangle$ is a sequence of pairwise
disjoint sets.  Hence there is $\mathbf I_\ell \subseteq G_\ell$
extending $f^1_\ell(\mathbf I_{\mathbf x_1,\ell}) \cup f^2_\ell(\mathbf
I_{\mathbf x_2,\ell})$ such that $\langle g G_{\mathbf x_0,0}:g \in \mathbf
I_\ell\rangle$ is a partition of $G_\ell$.

Define $\mathbf x_3$ by:
\mn
\begin{enumerate}
\item[$\bullet'_1$]  $G_{\mathbf x_3,0} = G_{\mathbf x_0,0}$
\sn
\item[$\bullet'_2$]  $G_{\mathbf x_3,\ell} = G_\ell$ for $\ell =1,2$
\sn
\item[$\bullet'_3$]  $\mathbf I_{\mathbf x_3,\ell} = \mathbf I_\ell$ for
  $\ell=1,2$.
\end{enumerate}
\mn
Now check that $\mathbf x_3,(f^\iota_1,f^\iota_2)$ for $\iota = 1,2$ are
as required.]
\mn
\begin{enumerate}
\item[$(*)_7$]  if $\mathbf y \in \mathbf x^{\mx}_{\bar G}$ then
  $\NF_{\fin}(G_{\mathbf y,0},G_{\mathbf y,1},G_{\mathbf y,2},G_{\mathbf y,3})$.
\end{enumerate}
\mn
[Should be clear now.] 

Alternatively\footnote{Or see \cite{Sh:1098}.}, use 
\ref{c25} below and \ref{c6}(8).
\end{PROOF}

\noindent
We give now further basic properties, mainly connecting it to
non-splitting (in \ref{c14}(4)).
\begin{claim}
\label{c14}
Assume $\NF_{\rfin}(G_0,G_1,G_2,G_3)$ hence
$\NF_{\fin}(G_0,G_1,G_2,G_3) \Leftrightarrow G_3 = \langle G_1 \cup
G_2\rangle_{G_3}$.

\noindent
1) Symmetry:  Also $\NF_{\rfin}(G_0,G_2,G_1,G_3)$ holds.

\noindent
2) Monotonicity: If $G_0 \subseteq G'_\ell \subseteq G_\ell$ for $\ell=1,2$ and
$G'_1 \cup G'_2 \subseteq G'_3 \subseteq G_3$ \then \,
{\rm NF}$_{\rfin}(G_0,G'_1,G'_2,G'_3)$.

\noindent
3) Uniqueness: if $\NF_{\fin}(G'_0,G'_1,G'_2,G'_3)$ hence $G'_3 =
\langle G'_1 \cup G'_2\rangle_{G'_3},f_\ell$ is an
isomorphism from $G'_\ell$ into $G_\ell$ for $\ell=0,1,2$ such that
   $f_1 \rest G'_0 = f_0 = f_2 \rest G'_0$ and $f_0$ is onto $G_0$, 
\then \, there is an
embedding $f_3$ of $G'_3$ into $G_3$ extending $f_1 \cup f_2$
(unique, of course; it is onto if and only if 
$G_3 = \langle G_1 \cup G_2\rangle_{G_3}$ and $f_\ell$ is onto
$G_\ell$ for $\ell=1,2$).

\noindent
4) Definability: If $\bar a \in {}^{\omega >}(G_2)$ \then \, {\rm tp}$_{\bs}
(\bar a,G_1,G_3)$  does not split over $G_0$.
\end{claim}

\begin{PROOF}{\ref{c14}}  
Straightforward but we elaborate.

\noindent
1) Use the symmetry in the definition (recall that in \S2 we have
$\mathbf K = \mathbf K_{\text{lf}}$ not $\mathbf K_{\text{olf}}$!)

\noindent
2) By \ref{c6}(7) and use the uniqueness in \ref{c12}(1).
Alternatively use \ref{c25} below and \ref{c6}(8).

\noindent
3) Easily, too.

\noindent
4) Obvious by parts (2) and (3).
\end{PROOF}

Now above the restriction of $G_1,G_2$ to be finite is undesirable.
\begin{definition}
\label{c16}
Let NF$_f(G_0,G_1,G_2,G_3)$ or ``$G_1,G_2$ are NF$_f$-stably
amalgamated over $G_0$ inside $G_3$" mean that:
\mn
\begin{enumerate}
\item[$(a)$]  $G_\ell \in \mathbf K$ for $\ell \le 3$
\sn
\item[$(b)$]  $G_0$ is finite
\sn
\item[$(c)$]  $G_0 \subseteq G_\ell \subseteq G_3$ for $\ell=1,2$ and
  $G_1 \cap G_2 = G_0$
\sn
\item[$(d)$]  if $G'_1,G'_2$ are finite groups and $G_0 \subseteq
G'_\ell \subseteq G_\ell$ for $\ell=1,2$ and $G'_3 = \langle G'_1 \cup
G'_2\rangle_{G_3}$ then NF$_{\text{fin}}(G_0,G'_1,G'_2,G'_3)$.
\end{enumerate}
\end{definition}

\begin{claim}
\label{c18}
\underline{Stable Amalgamation over Finite Claim}
1) Existence: If $G_0 \in \mathbf K$ is finite and 
$G_0 \subseteq G_\ell \in \mathbf
K$ for $\ell=1,2$ and
for transparency $G_1 \cap G_2 = G_0$ \then \, for some $G_3$ we have
{\rm NF}$_f(G_0,G_1,G_2,G_3)$ and $G_3 = \langle G_1 \cup G_2\rangle_{G_3}$.

\noindent
2) Uniqueness: In part (1), $G_3$ is unique 
up to isomorphism over $G_1 \cup G_2$.

\noindent
3) Monotonicity: If $G_0 \subseteq G'_\ell \subseteq G_\ell$ for $\ell=1,2$ and
$\NF_\ell(G_0,G_1,G_1,G_2)$ \then \, $\NF_f(G_0,G'_1,G'_2,G_3)$.

\noindent
4) Symmetry: $\NF_f(G_0,G_1,G_2,G_3)$ holds \Iff \,
$\NF_f(G_0,G_2,G_1,G_3)$ holds.

\noindent
5) Definability: If $\NF_f(G_0,G_1,G_2,G_3)$, \then \, $G_1
\le_{\Omega[\mathbf K]} G_3$.
\end{claim}

\begin{PROOF}{\ref{c18}}  
1) Straightforward by \ref{c12}(1) and \ref{c18}(2),(3),
i.e. existence follows from the existence for finite $G_1,G_2$ using
uniqueness and monotonicity.

\noindent
2) - 5) easy too; holds by \ref{c14}(3),
\ref{c14}(1), \ref{c14}(1), \ref{c14}(3), i.e. uniqueness, 
monotonicity, symmetry and definability respectively.
\end{PROOF}

Now we go back to the major problem left in \S1.
\begin{claim}
\label{c19}
There is one and only one full $\gs \in \Omega[\mathbf K]$ such that
$q=q_{\gs}(\bar a,G_1)$ \when \,:
\mn
\begin{enumerate}
\item[$(a)$]  $G_1 \in \mathbf K$ is existentially closed
\sn
\item[$(b)$]   $q(\bar x) \in \mathbf S^n_{\gS_{\atdf}}(G)$ is 
$\tp_{\bs}(\bar c,G_1,G_3)$, see below
\sn
\item[$(c)$]  $\NF_f(G_0,G_1,G_2,G_3)$
\sn
\item[$(d)$]  $\bar c \in {}^{n(\gs)}(G_2)$ and $\bar a \in
  {}^{k(\gs)}(G_0)$ generate $G_0$.
\end{enumerate}
\end{claim}

\begin{PROOF}{\ref{c19}}
By \ref{c18} and \ref{a15}(3).
\end{PROOF}

\begin{definition}
\label{c20}
Let $\gS_{\df} \subseteq \Omega[\mathbf K]$ be the closure of
$\gS_{\atdf}$, see Definition \ref{a21} 
where $\gS_{\atdf} \subseteq \gS(\mathbf K)$ is the set of $\gs \in
\Omega[\mathbf K]$ as in \ref{c19}.
\end{definition}

\begin{claim}
\label{c22}
1) $\gS_{\text{\rm df}}$ is well defined, see Definition \ref{c20},
   \ref{c14}(3). 

\noindent
2) $\gS_{\text{\rm df}}$ is dense (see Definition \ref{a21}(2)), closed
and countable.
\end{claim}

\begin{PROOF}{\ref{c22}}
1) Obvious.

\noindent
2) \underline{$\gS_{\text{df}}$ is dense}: holds by \ref{c18} and \ref{c19}
   recalling Definition \ref{c16}, \ref{c20}.

\noindent
\underline{$\gS_{\text{df}}$ is closed}: by its definition.

\noindent
\underline{$\gS_{\text{df}}$ is countable}: as $\gS_{\text{atdf}}$ is
by \ref{c14}(3), \ref{c18}(2) recalling \ref{a22}(3).
\end{PROOF}

\begin{discussion}
\label{c23}
Is $\gS_{\df}$ symmetric?  Not clear, however, in the end of
\S1 we have circumvented this and we shall in \S3 circumvent this in
another way.
\end{discussion}

\begin{claim}
\label{c25}
1) Assume $G_0 \subseteq G_\ell \in \mathbf K$ and $G_\ell$ is
existentially closed for $\ell=1,2$  and $G_0$ finite.

\Then \, we can find $\mathbf x \in \mathbf X_{\mathbf K}$ such that $G_\ell =
G_{\mathbf x,\ell}$ for $\ell = 0,1,2$ and $(\mathbf j_{\mathbf x,0}(G_0)
\subseteq \mathbf j_{\mathbf x,\ell}(G_\ell) \le_{\gS_{\df}}
G_{\mathbf x})$ and $\NF_f(\mathbf j_{\mathbf x,0}(G_0),
\mathbf j_{\mathbf x,1}(G_1),\mathbf j_{\mathbf x,2}(G_2),G_{\mathbf x})$.

\noindent
2) Assume $G_0 \subseteq G'_\ell \subseteq G_\ell$ and $G'_\ell$
finite (or just $(G_\ell:G'_\ell) = |G_\ell|)$ and $\mathbf y \in \mathbf
X_{\mathbf K},G_{\mathbf y,0} = G_0$ and $G_{\mathbf y,\ell} = G'_\ell$ for
$\ell=1,2$.
\Then \, in part (1) we can demand that $\mathbf x$ extends $\mathbf y$.
\end{claim}

\begin{remark}
\label{c27}
If $G_0 \subseteq G_\ell \in \mathbf K$ for $\ell=1,2$ then we can find infinite
$G'_1,G'_2 \in \mathbf K$ extending $G_1,G_2$ respectively
 as $\mathbf K$ is closed under (finite) product (for
$\mathbf K_{\text{olf}}$ use lexicographic order).
\end{remark}

\begin{PROOF}{\ref{c25}}  1) By the definitions it is easy.  That is,
for $\ell=1,2$ we can choose $\mathbf I_\ell$ as in \ref{c3}(1)(b)
satisfying:
\mn
\begin{enumerate}
\item[$(*)$]  if $G'_\ell \subseteq G_\ell$ is finite and extends
$G_0$ and $\mathbf I' \subseteq G'_\ell$ is such that $e_{G_0} \in \mathbf I'$
and $\langle g G_0:g \in \mathbf I'\rangle$ is a partition of $G'_\ell$
\then \, we can find $g^* \in G_\ell$ such that $\{g^* g:g \in \mathbf
I'\} \subseteq \mathbf I_\ell$.
\end{enumerate}
\mn
Now think.

\noindent
2) Similarly.
\end{PROOF}
\bigskip

\subsection {Constructing Reasonable Schemes}\
\bigskip

We now give some examples of $\gs \in \Omega[\mathbf K]$.
\begin{definition}
\label{c50}
1) Let $\gs_{\text{cg}}$ be the $\gs$ from \ref{c52}(2) below.

\noindent
2) Let $\gs_{\gel}$ be the $\gs$ from \ref{c52}(3) below.
\end{definition}

\begin{claim}
\label{c52}
1) For every $G \in \mathbf K_{\lf}$ there are $G^+$ and $a$ such that
   $G \subseteq G^+ \in \mathbf K_{\lf},G^+ = \langle G \cup
   \{a\}\rangle_{G^+}$; in $G^+$ the element $a$ does not 
commute with any $b \in G \backslash \{e_G\},a$ has order 2 
and the sets $G,a^{-1} Ga$ commute in $G^+$ and their intersection is
$\{e_G\}$.

\noindent
2) There is unique $\gs \in \Omega[\mathbf K]$ such that $k_{\gs} =
   0,n_{\gs}=1,p_{\gs}$ is empty and in part (1) above 
{\rm tp}$_{\bs}(a,G,G^+)$ is $q_{\gs}(<>,G)$.

\noindent
3) There is $\gs \in \Omega[\mathbf K]$ with $k_{\gs} = 1,n_{\gs} = 4,
p_{\gs}(\bar x_{\gs}) = \{x_0 = x^{-1}_0 \wedge x_0 \ne e\}$ such that: if 
$G \in \mathbf K_{\lf}$ and $a \in G$ realizes $p_{\gs}(x_0)$
\then \, there are  $G^+,\bar c$ such that $G \subseteq G^+ =
   \langle G \cup \bar c\rangle_{G^+},\tp_{\bs}(\bar c,G,G^+) = q_{\gs}
(\langle a \rangle,G)$ and $c_\ell$ realizes 
$q_{\gs_{\cg}}(<>,G)$ in $G^+$ for $\ell < n_{\gs}$ 
and $a \in \langle \bar c \rangle_{G^+}$.
\end{claim}

\begin{PROOF}{\ref{c52}}
We first make a less specific construction for any $G \in \mathbf K$.

For $n \ge 2$ let $\cU_n = G \times n = \{(g,\iota):g \in G,\iota < n\}$.  
For finite $K \subseteq G$ let $E_K := \{((g_1,\iota_1),(g_2,\iota_2)):g_1,g_2
\in G$ and $\iota_1,\iota_2 < n$ and $g_1 K = g_2 K\}$, this is an equivalence
relation on $\cU_n$, each equivalence class has $\le n \times |K|$
elements.  For $\bar a \in{}^{\omega >} G$ let $E_{\bar a} = E_K$ when
$K = \langle\text{Rang}(\bar a)\rangle_G$ which is finite. 

For $\bar a \in {}^n G$ and $\pi$ a permutation of $\{0,\dotsc,n-1\}$ let
$h_{\bar a,\pi}$ be the following function from $\cU_n$ into $\cU_n$:
\mn
\begin{enumerate}
\item[$(*)_1$]  $h_{\bar a,\pi}((g,\iota)) = (ga_\iota,\pi(\iota))$.
\end{enumerate}
\mn
Clearly
\mn
\begin{enumerate}
\item[$(*)_2$]  $h_{\bar a,\pi}$ is a permutation of $\cU_n$ which
maps every $E_{\bar a}$-equivalence class onto itself.
\end{enumerate}
\mn
Let $H$ be the group of permutations of $\cU_n$ generated by
$\{h_{\bar a,\pi}:\bar a \in {}^n G$ and $\pi$ is a permutation of
$\{0,\dotsc,n-1\}\}$, now by $(*)_2$ it is easy to see that $H \in 
\mathbf K_{\text{\rm lf}}$ where, as in earlier cases, 
\mn
\begin{enumerate}
\item[$\bullet$]  $H \models ``h =h_1 h_2"$ iff $x \in \cU_n 
\Rightarrow h(x) = h_2(h_1(x))$.
\end{enumerate}
\mn
Now for $\iota < n$ let $\mathbf j_\iota$ be the following function from
$G$ into $H$:
\mn
\begin{enumerate}
\item[$(*)_3$]  $\mathbf j_\iota(a) = h_{\bar b,\pi}$ when $\pi =$ the
identity and $b_k$ is $a$ if $k = \iota$ and is $e_G$ otherwise.
\end{enumerate}
\mn
Now
\mn
\begin{enumerate}
\item[$(*)_4$]  for $\iota < n,\mathbf j_\iota$ is an embedding of $G$ into
$H$.
\end{enumerate}
\mn
[Why?  Check.]

We let $G^*,\mathbf j_*$ be such that $G^* \supseteq G$ and $\mathbf j_*$ 
is an isomorphism from $G^*$ onto $H$ extending $\mathbf j_0$.

For later use note:
\mn
\begin{enumerate}
\item[$(*)_5$]   for transparency we can use existentially closed $G$.
\end{enumerate}
\mn
Also
\mn
\begin{enumerate}
\item[$(*)_6$]
\begin{enumerate}
\item[(a)]  $G \le_{\Omega[\mathbf K]} G^+$ equivalently
  $\mathbf j_0(G) = \mathbf j_*(G) \le_{\Omega[\mathbf K]} H$;
\sn
\item[(b)]  if $A \subseteq G$ and for $\ell < m$ we
  have $\bar a_\ell \in {}^n A$ and $\pi_\ell$ is a permutation of
$\{0,\dotsc,n-1\}$ then $p=\tp_{\bs}
(\langle h_{\bar a_\ell,\pi_\ell}:\ell < m\rangle,\mathbf j_*(G),H)$ 
 does not split over $A$;
\sn
\item[(c)]   if above $A$ is finite and $\bar a$ lists
  $A$ \then \, for some $\gs \in \Omega[\mathbf K],p=q_s(\bar a,\mathbf j_*(G)$.
\end{enumerate}
\end{enumerate}
\mn
[E.g. why clause (c) holds?  By \ref{c18}(2) recalling $(*)_5$.]

Now we prove each part.

\noindent
1) Let $n=2$ and $\pi$ be the permutation of $\{0,1\}$ such that
   $\pi(0)=1,\pi(1) = 0$, and let $a = \mathbf
   j^{-1}_*(h_{<e_G,e_G>,\pi})$.

\noindent
2) Should be clear by $(*)_6(c)$.

\noindent
3) First note that
\mn
\begin{enumerate}
\item[$\oplus_1$]  $\mathbf j^{-1}_*(h_{\bar a,\pi})$ realizes
$q_{\gs_{\text{cg}}}(\langle \rangle,G)$ in $G^*$ \when \,
for some $k \in \{1,\dotsc,n-1\}$ we have
\begin{enumerate}
\item[$\bullet_1$]  $\pi$ is a permutation of $\{0,\dotsc,n-1\}$ and
has order two
\sn
\item[$\bullet_2$]  $\pi(0) = k$
\sn
\item[$\bullet_3$]  $\pi(k) = 0$
\sn
\item[$\bullet_4$]  $\bar a \in {}^n G$ satisfies 
$a_{\pi(\iota)} = a^{-1}_\iota$ for $\iota <n$
\sn
\item[$\bullet_5$]  if $\pi(\iota) \ne \iota,\iota < n$ 
then $a_\iota = e_G$, (or just
  $a_0$ belongs to the center of $G$).
\end{enumerate}
\end{enumerate}
\mn
[Why?  By $(*)_2$ and the choice of $H$ 
clearly $h_{\bar a,\pi} \in H$ and inspecting $(*)_1$, easily $h_{\bar
a,\pi}$ has order two.  By the choice of $\mathbf j_*,\pi$ as $\pi(0) =
k,\pi(k) = 0$ and $a_k = e_G = a_0$, for $g \in G$ we get $H \models 
``h^{-1}_{\bar a,\pi} \mathbf j_0(g) h_{\bar a,\pi} = 
\mathbf j_k(g)"$.   
However, for every $g_1,g_2 \in G$ the elements
$\mathbf j_0(g_1),\mathbf j_k(g_2)$ of $H$ 
commute as $h_{\bar a_1,\pi_1},
h_{\bar a_0,\pi_2}$ commute in $H$, e.g. when $\pi_1 = \text{ id}_n =
\pi_2$ and $\bigwedge\limits_{\ell < n} (a_{1,\ell} = e \vee
a_{2,\ell} = e$).  Lastly, $g_1,g_2 \in G \wedge \mathbf j_0(g_1) =
\mathbf j_0(g_2) \Rightarrow g_1 = e_G = g_2$.
Together we are done.]

Let $n=3$ and for $\ell < 4$ let $g_\ell \in H$ 
be $h_{\bar a_\ell,\pi_\ell}$ where
$\pi_\ell,\bar a_\ell$ are defined by (recall $a \in G$ is given and
has order 2):
\mn
\begin{enumerate}
\item[$\oplus_2$]  for $\ell < 4$ let $\pi_\ell$ be such that:

\underline{$\ell=0,3$}:  the orbits are $\{0,1\},\{2\}$;

\underline{$\ell=1,2$}:  the orbits are $\{0,2\},\{1\}$.
\sn
\item[$\oplus_3$]   let $\bar a_\ell = \langle a_{\ell,i}:i < 3\rangle$
be $\langle e,e,e\rangle,\langle e,a,e\rangle,\langle e,e,e\rangle,(e,e,e)$ 
for $\ell=0,1,2,3$.
\end{enumerate}
\mn
Now
\mn
\begin{enumerate}
\item[$\oplus_4$]   $c_\ell := \mathbf j^{-1}_*(h_{\bar
a_\ell,\pi_\ell})$ realizes $q_{\gs_{\text{cg}}}(\langle \rangle,G)$ 
for $\ell < 4$.
\end{enumerate}
\mn
[Why?  We apply $\oplus_1$ with $k$ being $1$ for $\ell = 0,3$ and $2$ for
$\ell=1,2$.  So we have to check $\bullet_1 -
\bullet_4$ for each $\ell$; now $\bullet_1 + \bullet_2 + \bullet_3$
holds by inspecting $\oplus_2$ and the choice of $k$ and of
$\pi_\ell$.  

Lastly, for $\bullet_4 + \bullet_5$ 
note that $a,e$ has order $2$ and $a_{\ell,0}
= e_G = a_{\ell,k}$ by inspecting $\oplus_3$.]
\mn
\begin{enumerate}
\item[$\oplus_5$]  $\tp_{\bs}(\langle c_0,c_1,c_2,c_3),G,G^*)$ does not
split over $\langle a \rangle$, moreover is 
$q_{\gt}(\langle a \rangle,G)$ for some $\gt \in \Omega[\mathbf K]$.
\end{enumerate}
\mn
[Why?  Just think recalling $(*)_6$.]

Lastly,
\mn
\begin{enumerate}
\item[$\oplus_6$]   $G^+ \models ``c_0 c_1 c_2 c_3 = a"$.
\end{enumerate}
\mn
[Why?  This is equivalent to $H \models h_{\bar a_0,\pi_0} h_{\bar
a_1,\pi_1} h_{\bar a_2,\pi_2} h_{\bar a_3,\pi_3} = \mathbf j_0(a)$.
By the definition of the product we check how each 
$(g,\ell) \in \cU_n$ is mapped (see above, so 
$h_{\bar a_0,\pi_0}$ is applied first) applying 
$h_{\bar a_\ell,\pi_\ell}$ in turn:

\[
(g,0) \mapsto (g e,1) \mapsto (gea,1) \mapsto (geae,1) \mapsto
(geaee,0) = (ga,0) = \mathbf j_0(a)(g,0)
\]

\mn
and

\[
(g,1) \mapsto (ge,0) \mapsto (gee,2) \mapsto (geee,0) \mapsto (geeee,1) 
= (g,1) = \mathbf j_0(a)(g,1)
\]

\[
(g,2) \mapsto (g e,2) \mapsto (gee,0) \mapsto (g eee,2) \mapsto (geeee,2)
= (g,2) = \mathbf j_0(a)(g,2).
\]
\mn
So we are done.]
\end{PROOF}

The following will be used in the proof of existence of complete existentially 
closed $G$.
\begin{claim}
\label{c64}
1) If (A) then (B) \where \,:
\mn
\begin{enumerate}
\item[(A)]  
\begin{enumerate}
\item[(a)]  $G_n \subseteq G_{n+1} \in \mathbf K$ for $n <
\omega$ and $I$ a set;
\sn
\item[(b)]  $a^t_n \in G_{n+1}$ and let $b^t_n = a^t_0 \ldots
a^t_n$ in $G_{n+1}$ for $n < \omega,t \in I$;
\sn
\item[(c)]  $\bar a_n = \langle
a^t_n:t \in I\rangle,\bar b_n = \langle b^t_n:t \in I\rangle$;
\sn
\item[(d)]
\begin{enumerate}
\item[$(\alpha)$]  $\tp_{\bs}(\bar a_n,G_n,G_{n+1})$ is
  increasing\footnote{So by (A)(e) this is equivalent to ``$\tp(\bar
    a_n,\emptyset,G_{n+1})$ is constant".} with $n$;
\sn
\item[$(\beta)$]  $c \ell(\bar a_n,G_{n+1}) \cap G_n = \{e_{G_n}\}$; if
$I=\{t\}$ and $a^t_n$ has order $k(t)$
this means that for every $i \in \{1,\dotsc,k(t)\}$ we have:  

\hskip25pt $G_{n+1} \models ``(a^t_n)^i = e_{G_1}"$ iff 
$(a^t_n)^i \in G_n$ iff $i=k(t)$;
\end{enumerate}
\sn
\item[(e)]  $a^t_n$ commutes with every $c \in G_n$;
\sn
\item[(f)]  $G_\omega = \cup\{G_n:n < \omega\}$ hence $\in \mathbf K$;
\end{enumerate}
\sn
\item[(B)]  for some $\bar b_\omega,G_{\omega +1}$ we have:
\sn
\begin{enumerate}
\item[(a)]  $G_{\omega +1} \supseteq G_\omega$ belongs to $\mathbf K$;
\sn
\item[(b)]  $\bar b_\omega = \langle b^t_\omega:t \in I\rangle$ and
$b^t_\omega \in G_{\omega +1}$;
\sn
\item[(c)]  $G_{\omega +1} = c \ell(G_\omega 
\cup\{b^t_\omega:t \in I\},G_{\omega +1})$;
\sn
\item[(d)]  if $n < \omega$ then
$p_n = \tp_{\bs}(\bar b_\omega,G_n,G_{\omega +1}) 
= \tp_{\bs}(\bar b_n,G_n,G_{n+1})$.
\end{enumerate}
\end{enumerate}
\mn
2) If we have (A) except omitting (A)(d)$(\beta)$, still we have:
\mn
\begin{enumerate}
\item[(B)$'$]  $(a)-(c) \quad$ as above;
\sn
\begin{enumerate}
\item[$(d)$]  $\bar b_\omega \rest u$ realizes $\tp_{\bs}
(\bar b_{n!},G_{n!} \rest u,G_{n+1})$ in $G_{\omega +1}$ when $u
\subseteq I$ is finite and $n$ is large enough.
\end{enumerate}
\end{enumerate}
\end{claim}

\begin{PROOF}{\ref{c64}}  
1) Letting $p_n(\bar x) = \tp_{\bs}(\bar b_n,
G_n,G_{n+1})$, it is enough to prove:
\mn
\begin{enumerate}
\item[$(*)_1$]  $p_n \subseteq p_{n +1}$.
\end{enumerate}
\mn
For this it is enough to prove, letting $\bar y = \langle y_t:t \in I\rangle$,
\mn
\begin{enumerate}
\item[$(*)_2$]   if $\sigma(\bar y,\bar z)$ is a group-term and $\bar
c \in {}^{(\ell g(\bar z))}(G_n)$ \then \, $G_{n+2} \models 
``\sigma(\bar b_{n+1},\bar c) =
e"$ iff $G_{n+1} \models ``\sigma(\bar b_n,\bar c) = e"$.
\end{enumerate}
\mn
Towards proving $(*)_2$ note:
\mn
\begin{enumerate}
\item[$\bullet_{2.1}$]  $\bar c$ and $\bar b_n$, and 
hence $\sigma(\bar b_n,\bar c)$ are from $G_{n +1}$,
\sn
\item[$\bullet_{2.2}$]   $a^t_{n+1}$ commutes with every 
$c_i(i < \ell g(\bar c))$ and with $b^s_n$ for $s \in I$.
\end{enumerate}
\mn
By clause (A)(b) of the assumption of the claim,
\mn
\begin{enumerate}
\item[$\bullet_{2.3}$]   $b^t_{n+1} = b^t_n a^t_{n+1}$ and $b^t_n =
  b^t_{n-1} a^t_n$ stipulating $b^t_{-1}=e$.
\end{enumerate}
\mn
Similarly,
\mn
\begin{enumerate}
\item[$\bullet_{2.4}$]   $\bar c$ and $\bar b_{n-1}$ are from $G_n$;
\sn
\item[$\bullet_{2.5}$]  $a^t_n$ commute with every $c_i(i < \ell g(\bar c))$
and with $b^s_{n-1}$ for $s \in I$.
\end{enumerate}
\mn
Hence for some group term $\sigma_*(\bar x)$:
\mn
\begin{enumerate}
\item[$\bullet_{2.6}$]   $G_{n+2} \models ``\sigma(\bar b_{n+1},\bar c) =
\sigma(\bar b_n,\bar c) \sigma_*(\bar a_{n+1})"$;
\sn
\item[$\bullet_{2.7}$]   $G_{n+1} \models ``\sigma(\bar b_n,\bar c) =
\sigma(\bar b_{n-1},\bar c) \sigma_*(\bar a_n)"$.
\end{enumerate}
\mn
Hence by clauses (A)(d)$(\alpha),(\beta)$:
\mn
\begin{enumerate}
\item[$\bullet_{2.8}$]  $\sigma_*(\bar a_n) \in G_n$ iff $\sigma_*(\bar a_n)
= e_{G_n}$ iff $\sigma_*(\bar a_{n+1}) = e_{G_n}$ iff $\sigma_*(\bar
a_{n+1}) \in G_{n+1}$;
\sn
\item[$\bullet_{2.9}$]  if $\sigma_*(\bar a_n) \notin G_n$, hence
$\sigma_*(\bar a_{n+1}) \notin G_{n+1}$,  
then both statements in $(*)_2$ fail because:
\sn
\begin{enumerate}
\item[$(\alpha)$]  $\sigma(\bar b_n,\bar c)$ is from $G_{n+1}$ and $\sigma_*
(\bar a_{n+1}) \notin G_{n+1}$ so $\sigma(\bar b_{n+1},
\bar c) \notin G_{n+1}$ and thus
$\sigma(\bar b_{n+1},\bar c) \ne e_{G_n}$; 
\sn
\item[$(\beta)$]   similarly $\sigma(\bar b_n,\bar c) \notin G_n$ and thus
$\sigma(\bar b_n,\bar c) \ne e_{G_n}$; 
\end{enumerate}
\item[$\bullet_{2.10}$]   if $\sigma_*(\bar a_n) \in G_n$ hence $\sigma_*(\bar
a_n) = e = \sigma_*(\bar a_{n+1})$, then $\sigma(\bar b_{n+1},\bar c) =
\sigma(\bar b_n,\bar c)$ and again we are done.
\end{enumerate}
\mn
Together $(*)_2$ holds.

\noindent
2) Similarly (and the same as part (1) when $G_n$ is existentially
closed for every $n$) but we elaborate.  \Wilog \, $I$ is finite; 
letting $p_n(\bar y) = \tp_{\bs}(\bar a_n,G_n)$, we need:
\mn
\begin{enumerate}
\item[$(*)_1$]  if $\bar c$ is a finite sequence from $G_\omega$ \then
\, the sequence $\langle \tp_{\bs}(\bar b_{n!} \char 94 \bar
c,\emptyset,G_{n!+1}):n < \omega \rangle$ is eventually constant.
\end{enumerate}
\mn
Let $K_n = c \ell(\bar a_n,G_{n+1})$, so by clause $(A)(d)(\alpha)$ of
the assumption $|K_n|$ is constant, finite and $K_n \cap G_n$ is
$\subseteq$-increasing with $n$.  Hence for some $K_*,n(*)$ we have $n
\ge n(*) \Rightarrow K_n \cap G_n = K_*$ and let $k(*) = |K_*|$.
\Wilog \, $n(*) \ge k(*)$; so it is enough to prove
\mn
\begin{enumerate}
\item[$(*)_2$]  if $\bar y = \langle y_t:t \in I\rangle$ and
$\sigma(\bar y,\bar z)$ is a group term, $\bar c \in {}^{\ell g(\bar
z)}(G_n)$ and $n \ge n(*)$, \then \, $G_{n+1} \models ``\sigma(\bar
b_n,\bar c) = e"$ \Iff \, $G_{n+k(*)+1} \models 
\sigma(\bar b_{n+k(*)},\bar c) = e"$.
\end{enumerate}
\mn
As in part (1) we can prove that for some group term $\sigma_*(\bar y)$
we have
\mn
\begin{enumerate}
\item[$\boxplus$]  if $n \ge n(*)$ then $G_{n+2} \models \sigma(\bar
b_{n+1},\bar c) = \sigma(\bar b_n,\bar c) \sigma_*(\bar a_{n+1})$.
\end{enumerate}
\medskip

\noindent
\underline{Case 1}:  $``\sigma_*(\bar a_n) \notin G_n$ for some,
equivalently every, $n \ge n(*)$.

In this case $G_{n+1} \models ``\sigma(\bar b_n,\bar c) \ne e"$ for
every $n \ge n(*)$.
\medskip

\noindent
\underline{Case 2}:  $``\sigma_*(\bar a_n) \in G_n$ for some,
equivalently every, $n \ge n(*)$.

In this case there is $b$ such that $\sigma_*(\bar a_n,\bar c) = b$
for every $n \ge n(*)$.  So for every $n \ge n(*)$ by induction on $m$
we can prove $\sigma(\bar b_{n+m},\bar c) = \sigma(\bar b_n,\bar c)
\cdot b^m$.  But necessarily $b \in K_*$ hence $b$ has order dividing
$|K_*| = k(*)$.  Hence $n \ge n(*) \Rightarrow 
\sigma(\bar b_{n+k(*)},\bar c) = \sigma(\bar b_n,\bar c)$ and thus 
$n_2 > n_1 \ge n(*) \wedge k(*)|(n_2-n_1) \Rightarrow
\sigma(\bar b_{n-2},\bar c) = \sigma(\bar b_{n_1},\bar c)$, and so we 
can finish easily.
\end{PROOF}

\begin{dc}
\label{c67}
1) For $k=2,3...$ let $\gs_{\ab(k)}$ be the unique 
$\gs \in \Omega[\mathbf K_{\lf}]$ such that:
\mn
\begin{enumerate}
\item[$(a)$]  $n(\gs)=1,k(\gs)=0$;
\sn
\item[$(b)$]   if $G \subseteq H$ and $c \in H$ realizes
$q_{\gs}(<>,G) = \tp_{\bs}(c,G,H)$ \then \, $c$
commutes with every $a \in G$;
\sn
\item[$(c)$]  also for every $m < \omega,a^m = e_H$ 
iff $a^m \in G$ iff $k|m$.
\end{enumerate}
\mn
2) Assume $K \in \mathbf K_{\lf}$ is finite and $\bar c \in
{}^{|K|}K$ list it.  \Then \, let $\gs = \gs_{\ab}(\bar c,K)$ be the
unique $\gs \in \Omega[\mathbf K_{\lf}]$ such that:
\mn
\begin{enumerate}
\item[$(a)$]  $n(\gs) = \ell g(\bar c),k(\gs) = 0$ so $p_{\gs}(\bar
x_{\gs}) = \emptyset$;
\sn
\item[$(b)$]  if $G \subseteq H \in \mathbf K_{\lf}$ and $\bar c' \in
{}^{n(\gs)}H$ \then \, the following are equivalent:
\sn
\begin{enumerate}
\item[$(\alpha)$]  $\tp(\bar c',G,H) = q_{\gs}(<>,G)$, 
\sn
\item[$(\beta)$]   $\bar c'$ commutes with $G$, realizes $\tp(\bar
  c,\emptyset,K)$ and $\langle \bar c' \rangle_H \cap G = \{e\}$.
\end{enumerate}
\end{enumerate}
\end{dc}

\begin{claim}
\label{c70}
Assume $\NF_f(G_0,G_1,G_2,G_3)$ and $a \in G_1 \backslash G_0,b
\in G_2 \backslash G_0$.  \Then \, $a,b$ commute in $G_3$ iff $a \in
\mathbf C_{G_1}(G_0),b \in \mathbf C_{G_2}(G_0)$ and $G_0$ is commutative.
\end{claim}

\begin{remark}
1) $\NF_f$ is from Definition \ref{c16}.

\noindent
2) Recall $g^{[a]} = a^{-1} g a$.
\end{remark}

\begin{PROOF}{\ref{c70}}
\Wilog \, $G_1,G_2$ are existentially closed (by
  monotonicity of $\NF_f$, see \ref{c18}(3) and existence of
  existentially closed extensions).

First assume
\mn
\begin{enumerate}
\item[$\oplus$]  $a \in \mathbf N_{G_1}(G_0)$ and 
$b \in \mathbf N_{G_2}(G_0)$.
\end{enumerate}
\mn
By \ref{c25}, \wilog \, we can find
 $\mathbf x \in \mathbf X_{\mathbf K}$ such that $G_3 = G_{\mathbf x},
G_{\mathbf x,\ell} = G_\ell$ for $\ell < 3$ and let 
$f_a = \mathbf j_{\mathbf x,1}(a),f_b = \mathbf j_{\mathbf x,2}(b)$; we shall
use the fact that: we have some freedom in the choice of 
$\mathbf x$, see \ref{c25}.

Let $(g_0,g_1,g_2) \in \cU_{\mathbf x}$ and we should see whether $f_b
\circ f_a((g_0,g_1,g_2)) = f_a \circ f_b((g_0,g_1,g_2))$; there are unique
$a',h_a,b',h_b$ such that:
\mn
\begin{enumerate}
\item[$(*)_0$]  
\begin{enumerate}
\item[(a)]  $g_1 a = a' h_a$ with $h_a \in G_0,a' \in \mathbf I_{\mathbf x,1}$;
\sn
\item[(b)]  $g_2 b = b' h_b$ with $h_b \in G_0,b' \in \mathbf I_{\mathbf x,2}$.
\end{enumerate}
\end{enumerate}
\mn
Now
\mn
\begin{enumerate}
\item[$(*)_1$]  $f_a((g_0,g_1,g_2)) = (h_a g^{[a]}_0,a',g_2)$.
\end{enumerate}
\mn
[Why?  As $g_1 g_0 a = g_1 a g^{[a]}_0 = a'(h_a g^{[a]}_0)$, noting
that $g^{[a]}_0 \in G_0$ because we are assuming that $a$ normalize
$G_0$ inside $G_1$.]
\mn
\begin{enumerate}
\item[$(*)_2$]  $f_b((h_a g^{[a]}_0,a',g_2)) = (h_b h^{[b]}_a
g^{[a][b]}_0,a',b')$.
\end{enumerate}
\mn
[Why?  As $g_2(h_a g^{[a]}_0)b = g_2 b(h^{[b]}_a g^{[a][b]}_0) = b'(h_b
h^{[b]}_a g^{[a][b]}_0)$.]

So,
\mn
\begin{enumerate}
\item[$(*)_3$]  $(f_b \circ f_a)((g_0,g_1,g_2)) = (h_b h^{[b]}_a
g^{[a][b]}_0,a',b')$.
\end{enumerate}
\mn
Now,
\mn
\begin{enumerate}
\item[$(*)_4$]  $f_b((g_0,g_1,g_2)) = (h_b g^{[b]}_0,g_1,b')$.
\end{enumerate}
\mn
[Why?  As $g_2 g_0 b = g_2 b g^{[b]}_0 = b' h_b g^{[b]}_0)$.]
\mn
\begin{enumerate}
\item[$(*)_5$]  $f_a((h_b g^{[b]}_0,g_1,b')) = (h_a h^{[a]}_b
g^{[b][a]}_0,a',b')$.
\end{enumerate}
\mn
[Why?  As $g_1(h_b g^{[b]}_0) a = g_1 a(h^{[a]}_b g^{[b][a]}_0) = a'(h_a
h^{[a]}_b g^{[b][a]}_0)$.]

Hence,
\mn
\begin{enumerate}
\item[$(*)_6$]  $(f_a \circ f_b)((g_0,g_1,g_2)) = (h_a h^{[a]}_b
g^{[b][a]}_0,a',b')$.
\end{enumerate}
\mn
Together we can deduce:
\mn
\begin{enumerate}
\item[$(*)_7$]  $(f_b \circ f_a)(g_0,g_1,g_2)) = (f_a \circ
f_b)(g_0,g_1,g_2)$ iff $h_b h^{[b]}_a g^{[a][b]}_0 = h_a h^{[a]}_b
g^{[b][a]}_0$ in $G_0$.
\end{enumerate}
\mn
Now, \underline{not} assuming $\oplus$ we shall prove the claim by cases 
(using $(*)_7$ when $\oplus$ holds).
\mn
\begin{enumerate}
\item[$\oplus_1$]  $a,b$ commute in $G_3$ \when \,:
\sn
\begin{enumerate}
\item[$\bullet_1$]  $a$ commutes with $G_0$ in $G_1$,
\sn
\item[$\bullet_2$]  $b$ commutes with $G_0$ in $G_2$,
\sn
\item[$\bullet_3$]  $G_0$ is commutative.
\end{enumerate}
\end{enumerate}
\mn
[Why? Note that the assumption $\oplus$ holds (by $\bullet_1 +
\bullet_2$), and so let $\mathbf x \in
\mathbf X_{\mathbf K}$ be as above.   For any $(g_0,g_1,g_2) \in \cU_{\mathbf
x}$, we can apply $(*)_7$ thus $h_a,h_b \in G_0$ are well defined,
by $(*)_0$.  Now as $h_b,h_a,g_0 \in G_0$ and $a \in \mathbf C_{G_1}(G_0),b 
\in \mathbf C_{G_1}(G_0)$ and $G_0$ is commutative, 
by the present assumptions, clearly
$h_b h^{[b]}_a g^{[a][b]}_0 = h_b h_a g_0 = h_a
h_b g_0 = h_a h^{[a]}_b g^{[b][a]}_0$.  As $G_3 = G_{\mathbf x},G_{\mathbf
  x}$ is a group of permutations of $\cU_{\mathbf x}$ and $(*)_7$ holds
for any $(g_0,g_1,g_3) \in \cU_{\mathbf x}$, clearly
$f_a,f_b \in G_{\mathbf x}$ commute, so we are done.]
\mn
\begin{enumerate}
\item[$\oplus_2$]  $a,b$ do not commute in $G_3$ \when \,:
\sn
\begin{enumerate}
\item[$\bullet$]  $a$ commutes with $G_0$,
\sn
\item[$\bullet$]  $b$ commutes with $G_0$,
\sn
\item[$\bullet$]  $G_0$ is not commutative.
\end{enumerate}
\end{enumerate}
\mn
[Why?  Choose $h_1,h_2 \in G_0$ which do not commute and let
$(g_0,g_1,g_2) = (e_{G_0},e_{G_1},e_{G_2}) = (e,e,e)$; note that 
$ah^{-1}_\ell \notin G_0,b h^{-1}_\ell \notin G_0$ for $\ell=1,2$.

Above we could have chosen $\mathbf x \in \mathbf X_{\mathbf K}$ 
such that $(G_{\mathbf x,\ell} = G_\ell$ for $\ell < 3$ and) 
$a h^{-1}_1 \in \mathbf I_{\mathbf x,1},b h^{-1}_2 \in \mathbf I_{\mathbf x,2}$.
Again $\oplus$ holds hence $(*)_7$ holds for any relevant
$\mathbf x,g_0,g_1,g_2$.  Recall $(g_0,g_1,g_2) := (e,e,e)$, so $g_1 a = ea=a=(a
h^{-1}_1)h_1$.  So in 
$(*)_0(a)$, we get $a' = a h^{-1}_1$ and $h_a=h_1$.  Similarly in
$(*)_0(b)$ we get $b'= b h^{-1}_2,h_b = h_2$.  So
$f_a,f_b \in G_{\mathbf x}$ do not commute by $(*)_7$, because 
 we get $h_b h^{[b]}_a g^{[a][b]}_0 = h_b h_a g_0 = h_2 h_1 g_0 \ne
 h_1 h_2 g_0 = h_a h_b g_0 = h_a h^{[a]}_b
g^{[b][a]}_0$, the inequality as $G_0 \models h_1 h_2 \ne h_2 h_1$ 
and we are done by $(*)_7$.]
\mn
\begin{enumerate}
\item[$\oplus_3$]  $a,b$ does not commute in $G_3$ \when \,:
\sn
\begin{enumerate}
\item[$\bullet$]  $a$ normalizes $G_0$ in $G_1$,
\sn
\item[$\bullet$]  $b$ normalizes $G_0$ in $G_2$,
\sn
\item[$\bullet$]  $b$ does not commute with $G_0$ in $G_2$.
\end{enumerate}
\end{enumerate}
\mn
[Why?  Again $\oplus$ holds hence we can apply $(*)_7$ for any
relevant $\mathbf x,g_0,g_1,g_2$.  Let $h_1 \in
G_0$ be such that it does not commute with $b$ in $G_2$ and 
let $h_2 = e_{G_0}$.  Choose
above $\mathbf x \in \mathbf X_{\mathbf K}$ such that $a h^{-1}_1 \in \mathbf I_{\mathbf
x,1}$ and $b = b h^{-1}_2 \in 
\mathbf I_{\mathbf x,2}$ and let $(g_0,g_1,g_2) = (e,e,e)$.  Again 
in $(*)_0$ we get $a' = a h^{-1}_1,h_a = h_1$ and $b' = b h^{-1}_2,h_b
= h_2 = e$.  Now
$h_b h^{[b]}_a g^{[a][b]}_0 = e h^{[b]}_a e = h^{[b]}_a \ne h_a = h_a ee =
h_a h^{[a]}_b g^{[b][a]}_0$, the inequality by the choice of $h_a = h_1$.]
\mn
\begin{enumerate}
\item[$\oplus_4$]  $a,b$ do not commute in $G_3$ \when :
\sn
\begin{enumerate}
\item[$\bullet$]  $a$ normalizes $G_0$ in $G_1$,
\sn
\item[$\bullet$]  $b$ normalizes $G_0$ in $G_2$,
\sn
\item[$\bullet$]  $a$ does not commute with $G_0$ in $G_2$.
\end{enumerate}
\end{enumerate}
\mn
[Why?  Like $\oplus_3$.]

Next
\mn
\begin{enumerate}
\item[$\oplus_5$]  $a,b$ does not commute in $G_3$ \when \,:
\sn
\begin{enumerate}
\item[$\bullet$]  $a \in G_1 \backslash G_0$ does not normalize $G_0$.
\end{enumerate}
\end{enumerate}
\mn
Why?  Choose $h \in G_0$ such that $a^{-1} h a \notin G_0$ hence $h a
\notin a G_0$ and, of course, $ha \notin G_0$ as $a \notin G_0,h \in
G_0$ and similarly $b h^{-1} \in G_2 \backslash G_0$.  Let $a' =ha$ so
$a' \ne a$ because $h \ne e$.

Choose above $\mathbf x \in \mathbf X_{\mathbf K_{\text{lf}}}$ such that 
$b h^{-1} \in \mathbf I_{\mathbf x,2}$ and 
$a,a' \in \mathbf I_{\mathbf x,1}$.  Why can we
choose such $\mathbf I_{\mathbf x,1}$?  Because $a' = ha \in
G_1 \backslash G_0,a \in G_1 \backslash G_0$ and $a G_0 
\ne a' G_0$, as otherwise for some $h_1 \in G_0$ we have 
$a' = a h_1$, and so $a^{-1} h a = a^{-1} a' = 
h_1 \in G_0$, contradicting the choice of $h$.

Let $f_a,f_b$ be as above for this choice of $\mathbf x$.

Now consider $(e,e,e) \in \cU_{\mathbf x}$ so
\mn
\begin{enumerate}
\item[$(*)'_1$]  $f_a((e,e,e)) = (e,a,e)$.
\end{enumerate}
\mn
[Why?  As $a \in \mathbf I_{\mathbf x,1}$.]
\mn
\begin{enumerate}
\item[$(*)'_2$]  $f_b((e,a,e)) = (h,a, b h^{-1})$.
\end{enumerate}
\mn
[Why? Because $b h^{-1} \in \mathbf I_{\mathbf x,2},h \in G_0$.]
\mn
\begin{enumerate}
\item[$(*)'_3$]  $(f_b \circ f_a)(e,e,e) = (h,a,b h^{-1})$.
\end{enumerate}
\mn
[Why?  By $(*)'_1 + (*)'_2$.]
\mn
\begin{enumerate}
\item[$(*)'_4$]  $f_b((e,e,e)) = (h,e, b h^{-1})$.
\end{enumerate}
\mn
[Why?  Because $b h^{-1} \in \mathbf I_{\mathbf x,2}$ and $h \in G_0$.]
\mn
\begin{enumerate}
\item[$(*)'_5$]  $f_a((h,e,b h^{-1}) = (e,a',b h^{-1})$.
\end{enumerate}
\mn
[Why?  As $eha = ha = a' = a'e$ and $a' \in \mathbf I_{\mathbf x,1}$.]
\mn
\begin{enumerate}
\item[$(*)'_6$]  $(f_a \circ f_b)((e,e,e) = (e,a',bh^{-1}))$.
\end{enumerate}
\mn
[Why?  By $(*)'_4 + (*)'_5$.]

By $(*)'_3 + (*)'_6$, as $a' \ne a$ the triple 
$(e,e,e)$ exemplifies $\mathbf j_{\mathbf x,1}(a),\mathbf j_{\mathbf x,2}(b)$ 
do not commute in $G_{\mathbf x}$.

Lastly,
\mn
\begin{enumerate}
\item[$\oplus_6$]  $a,b$ do not commute in $G_3$ \when \,:
\sn
\begin{enumerate}
\item[$\bullet$]   $b \in G_2 \backslash G_0$ does not normalize $G_0$.
\end{enumerate}
\end{enumerate}
\mn
[Why?  As in $\oplus_5$.]

As we have covered all the cases we are done.
\end{PROOF}

\begin{claim}
\label{c73}
Assume $\gS \subseteq \Omega[\mathbf K_{\lf}]$ and $G_1
\le_{\gS} G_2,G_1$ is existentially closed 
and $d \in G_2$.  If conjugation by $d$
(in $G_2$) maps $G_1$ onto itself \then \, for some $c \in G_1$ we
have $a \in G_1 \Rightarrow c^{-1} ac = d^{-1} ad$,
i.e. $dc^{-1} a = adc^{-1}$, i.e. $dc^{-1}$, a commute in $G_1$ so
$dc^{-1}$ commute with $G_1$.
\end{claim}

\begin{PROOF}{\ref{c73}}
Easy.  Clearly there is $(\gs,\bar a) \in \deef(G_1)$ such that 
$\tp_{\bs}(d,G_1,G_2) = q_{\gs}(\bar a,G_1)$), hence if $b,b_1,c_1 \in G_1$
 and $\tp_{\bs}(\langle b_1,c_1\rangle,\bar a,G_1) = \tp_{\bs}(\langle
 b,d^{-1} b d \rangle,\bar a,G_1)$ then $d^{-1} b_1 d = c_1$.  Having disjoint
amalgamation we have $x \in G_1 \Rightarrow d^{-1} xd \in c \ell(\bar
a \char 94 \langle x \rangle,G_1)$.  We can continue or note that if
 there is no $c \in G_1$ as desired, then every existentially closed 
$G$ has a non-inner automorphism, contradiction.
\end{PROOF}
\newpage

\section {Symmetryzing}

Our intention is to start with $\gS \subseteq \Omega[\mathbf K]$ which may
contain $\gs_1,\gs_2$ failing symmetry but have the nice
conclusion as for symmetric $\gS$.  Towards this we define the
operation $\otimes$, related to $\oplus$ defined in Definition
\ref{a21}(4),(4A), and
$\gS-\otimes$-constructions (close but not the same as the constructions
in Definition \ref{a25}, \ref{a37}, \ref{a43}) and
$\gS-\oplus$-constructions.

Note that $\gS_{\atdf}$ has ``quasi symmetry", i.e. when the
parameter (= base of amalgamation) is the same, but when we allow
increasing the base this is not clear.  Now $\otimes$ is like $\oplus$
when we insist on it being symmetric.  We use the construction here
in \S4,\S5 where we sometimes give more details.
Recall def$(G)$ for $G \in \mathbf K$ is from Definition \ref{a14f}(1).  

\noindent
Recall
\begin{definition}
\label{a52}
For $t \in \deef(G)$ let $q_t(G) = q_{\gs_t}(\bar a_t,G)$ and
$n_t = n_{\gs_t},k_t = k_{\gs_t}$ and see Definition \ref{a14f}(6).
\end{definition}

\begin{definition}
\label{a53}
1) On $\deef(G)$ we define a (partial) operation $\otimes$ by $t_1
 \otimes t_2 = (\gs_{t_1} \otimes \gs_{t_2},\bar a_{t_1} \char 94
\bar a_{t_2})$, see below.

\noindent
2) ${\gs} = {\gs}_1 \otimes {\gs}_2$ means that 
 $\gs_1,\gs_2$ are disjoint\footnote{As we
use only invariant $\gS$, this is not a real restriction.},
 $\bar x_{\gs} = \bar x_{\gs_1} \char 94 
\bar x_{\gs_2},\bar z_{\gs} = \bar z_{\gs_1}
\char 94 \bar z_{\gs_2}$, so $k(\gs) = k(\gs_1) + k(\gs_n),
n({\gs}) = n({\gs}_1) + n({\gs}_2)$ and:
\mn
\begin{enumerate}
\item[$\boxplus$]  if $H \subseteq H^+ \in \mathbf K,
\bar a_\ell \in {}^{k(\gs_\ell)}H$ realizes $p_{\gs_\ell}(\bar
x_{\gs_\ell})$ in $H$ and $\bar c_\ell \in {}^{n({\gs}_\ell)}(H^+)$ 
for $\ell=1,2$, then $\bar c_1 \char 94 \bar c_2$ realizes 
$q_{\gs}(\bar a_1 \char 94 \bar a_2,H)$ \underline{iff}:
\mn
\begin{enumerate}
\item[$(a)$]  $\bar c_\ell$ realizes $q_{\gs_\ell}(\bar a_\ell,H)$ in $H^+$,
for $\ell=1,2$;
\sn
\item[$(b)$]   if $\sigma(\bar z_1,\bar z_2,\bar y)$ is a group-term,
$\ell g(\bar z_1) = n({\gs}_1),\ell g(\bar z_2) = n({\gs}_2)$
and $\bar b \in {}^{\ell g(\bar y)}(H)$, then $(\alpha)
\Leftrightarrow (\beta)$ where:
\sn
\item[${{}}$]  $(\alpha) \quad H^+ \models ``\sigma(\bar c_1,
\bar c_2,\bar b) = e_H"$,
\sn
\item[${{}}$]  $(\beta) \quad (\sigma(\bar z_1,\bar c_2,\bar b) = e) 
\in q_{\gs_1}(\bar a_1,H^+) \text{ and}$

\hskip25pt $(\sigma(\bar c_1,\bar z_2,\bar b) =e) \in q_{\gs_2}(\bar a_2,H^+)$.
\end{enumerate}
\end{enumerate}
\end{definition}

\begin{claim}
\label{a54}
1) If $\gs_1,\gs_2 \in \Omega[\mathbf K]$ \then \, 
$\gs = \gs_1 \otimes \gs_2$ belongs to
$\Omega[\mathbf K]$.

\noindent
2) If $G \in \mathbf K$ and $t_1,t_2 \in$ {\rm def}$(G)$ \then \, $t=t_1
   \otimes t_2 \in$ {\rm def}$(G)$.
\end{claim}

\begin{PROOF}{\ref{a54}}
Straightforward.
\end{PROOF}

\begin{definition}
\label{a56}
1) Let $\approx^*_G$ be the following two-place relation on 
def$(G):(\gs_1,\bar a_1) \approx^*_G (\gs_2,\bar a_2)$ if both are in
def$(G)$ and $G \subseteq G^+ \in \mathbf K \Rightarrow q_{\gs_1}
(\bar a_1,G^+) = q_{\gs_2}(\bar a_2,G^+)$, (compare with $\approx_G$
from \ref{a14f}(6)).

\noindent
2) For $t_1,t_2 \in \deef(G)$ let $t_1 \le t_2$ means $\dom(\bar
x_{t_1}) \subseteq \dom(\bar x_{t_2}),\dom(\bar z_{t_1}) \subseteq
\dom(z_{\bar t_2})$ and $\bar a_{t_1} = \bar a_{t_2} \rest \dom(\bar
x_{t_1})$, and if $G \subseteq G_1 \subseteq G_2$ and 
$\bar c_2$ realizes $q_{t_2}(G_1)$ in $G_2$ then
$\bar c_2 \rest \dom(\bar z_{t_1})$ realizes $q_{t_1}(G)$ in $G_2$.

\noindent
3) $t_1 \le_{\bar h} t_2$ is defined similarly as in \ref{a21}(7).
\end{definition}

\begin{claim}
\label{a57}
0) $\approx^*_G$ is an equivalence relation on $\deef(G)$.

\noindent
1) If $(\gs,\bar a) \in \deef(G_1)$ and $G_1 \subseteq G_2 \in
   \mathbf K$ \then \, $q_{\gs}(\bar a,G_1) \subseteq 
q_{\gs}(\bar a,G_2)$  and $(\gs,\bar a) \in \deef(G_2)$.

\noindent
2) If $G \in \mathbf K$ and $(\gs_\ell,\bar a_\ell) \in \deef(G)$ 
for $\ell=1,2$, \then \, the satisfaction of
$(\gs_1,\bar a_1) \approx^*_G (\gs_2,\bar a_2)$ depends just on
$\gs_1,\gs_2$ and $\tp_{\bs}(\bar a_1 \char 94 \bar a_2,\emptyset,G)$. 

\noindent
3) Transitivity: in Definition \ref{a56}(2), $\le$ is indeed a
partial order.

\noindent
4) Moreover if $(\gs_1,\bar a_1) \le_{\bar h_1}
(\gs_2,\bar a_2) \le_{\bar h_2} (\gs_3,\bar a_3)$ \then \, $(\gs_1,\bar a_1)
\le_{\bar h_2 \circ \bar h_1} (\gs_3,\bar a_3)$.  
\end{claim}

\begin{PROOF}{\ref{a57}}
Easy.
\end{PROOF}

\begin{claim}
\label{a59}
0) The operation $\otimes$ on disjoint pairs respects congruency
(see Definition \ref{a14f}(3), Claim \ref{a22f}(1)).

\noindent
1) The operation $\otimes$ respects $\approx^*_G$, i.e. if $t_1
\approx^*_G t'_1$ and $t_2 \approx^*_G t'_2$ \then \, $t_1 \otimes t_2
\approx^*_G t'_1 \otimes t'_2$ assuming the operations are well
   defined, of course.

\noindent
2) If $(\gs,\bar a) = (\gs_1,\bar a_1) \otimes (\gs_2,\bar a_2)$,
\then \, $(\bar{\gs}_\ell,\bar a_\ell) \le (\gs,\bar a)$.

\noindent
3) If in $\deef(G)$ we have $t_\ell \le t'_\ell$ for $\ell=1,2$ and
$t'_1 \otimes t'_2$ is well defined (i.e. $t'_1,t'_2$ are disjoint)
   \then \, $t_1 \otimes t_2 \le t'_1 \otimes t'_2$.

\noindent
4) The operation $\otimes$ is 
associative and is symmetric, e.g. symmetry means: if 
$G \subseteq G^+$ and 
$(\gs_\ell,\bar a_\ell) \in \deef(G)$ and $\bar c^\ell_\ell
 \char 94 \bar c^\ell_{3-\ell}$ realizes $q_{t_\ell}(G)$ in $G^+$,
where $t_\ell = (\gt_\ell,\bar b_\ell) = (\gs_\ell,\bar a_\ell)
 \otimes (\gs_{3-\ell},\bar a_{3-\ell})$, (so assuming disjointness 
for transparency), for $\ell=1,2$,  \then \, 
$\tp_{\bs}(\bar c^1_1 \char 94 \bar c^1_2,G,G^+) 
= \tp_{\bs}(\bar c^2_1 \char 94 \bar c^2_2,G,G^+)$.

\noindent
5) If in $\deef(G)$ we have $(\gs_\ell,\bar a_\ell) \le_{h_\ell}
(\gs'_\ell,\bar a'_\ell)$ for $\ell=1,2$ and $\Dom(h_1) \cap
   \Dom(h_2) = \emptyset,\Rang(h_1) \cap \Rang(h_2) = \emptyset$ 
\then \, $(\gs_1,\bar a_1)
\otimes (\gs_2,\bar a_2) \le_{h_1 \cup h_2} (\gs'_1,\bar a_1) \otimes
(\gs'_2,\bar a_2)$.
\end{claim}

\begin{PROOF}{\ref{a59}}
Straightforward.
\end{PROOF}

\begin{remark}
\label{a60}
1) Also the operation $\oplus$ 
satisfies the parallels of \ref{a59}(1),(2),(3) and the first demand
in (4).

\noindent
2) We may phrase \ref{a59}(5) as in \ref{a59}(3) and vice versa.
\end{remark}

\begin{definition}
\label{a61}
Assume $\gS \subseteq \Omega[\mathbf K]$ is closed.

\noindent
1) We say $\gS \subseteq \Omega[\mathbf K]$ is $\otimes$-closed \when \,
(recalling it is invariant) if 
$\gs_\ell \in \gS$ for $\ell=1,2$ are disjoint \then \, $\gs = \gs_1
\otimes \gs_2 \in \gS$.

\noindent
2) The $\otimes$-closure of $\gS$ is the $\subseteq$-minimal
$\otimes$-closed $\gS' \subseteq \Omega[\mathbf K]$ such that $\gS
\subseteq \gS'$.

\noindent
3) Let $G_3 = G_1 \bigotimes\limits_{G_0}^{\gS} G_2$ or $G_3 =
\otimes_{\gS}(G_0,G_1,G_2)$ mean:
\mn
\begin{enumerate}
\item[$(*)$]
\begin{enumerate}
\item[(a)]  $G_0 \le_{\gS} G_2 \subseteq G_3 \in \mathbf K$ and 
$G_0 \le_{\gS} G_1 \subseteq G_3$ and $G_3 = \langle G_1 \cup G_2
\rangle_{G_3}$
\sn
\item[(b)]  if $\tp_{\bs}(\bar c_\ell,G_0,G_\ell) =
 q_{\gs_\ell}(\bar a_\ell,G_0)$ so $\bar c_\ell \in {}^{\omega
>}(G_\ell),\bar a_\ell \in {}^{\omega >}(G_0)$ 
for $\ell=1,2$, \then \,
$\tp_{\bs}(\bar c_1 \char 94 \bar c_2,G_0,G_3) = q_{\gs}(\bar a_1 
\char 94 \bar a_2,G_0)$ when
$(\gs,\bar a_1 \char 94 \bar a_2) = (\gs,\bar a_1) \otimes 
(\gs_2,\bar a_2)$; note that \wilog \,
$\gs_1,\gs_2$ are disjoint, (i.e. as in the proof of \ref{a23}).
\end{enumerate}
\end{enumerate}
\mn
4) $\NF^2_{\gS}(G_0,G_1,G_2,G_3)$ means that $G_0 \le_{\gS} G_\ell
   \le_{\gS} G_3$ for $\ell=1,2$ and the demands in (3) hold except that
 possibly $G_3 \ne \langle G_1 \cup G_2\rangle_{G_3}$.
\end{definition}

\begin{claim}
\label{a64}
Assume $\gS$ is closed and moreover $\otimes$-closed.

\noindent
1) $G_3 = \otimes_{\gS} (G_0,G_1,G_2)$ iff 
$\NF^2_{\gS}(G_0,G_1,G_2,G_3)$ and $G_3 = \langle G_1 \cup G_2\rangle_{G_3}$.

\noindent
2) (disjointness): $\NF^2_{\gS}(G_0,G_1,G_2,G_3)$ implies $G_1 \cap G_2 =
   G_0$.

\noindent
3) (uniqueness): If $G^\iota_3 =
   \otimes_{\gS}(G^\iota_0,G^\iota_1,G^\iota_2)$ for $\iota = 1,2$ and
   $f_\ell$ is an isomorphism from $G^1_\ell$ onto $G^2_\ell$ for
   $\ell = 1,2$ and $G^1_0 = G^2_0,f_1 \rest G^1_0 = f_2 \rest G^2_0$  
and $G_0$ is existentially closed\footnote{Why?  The problem is that $G
   \le_{\gS} H \in \mathbf K$ does not imply the existence of $\bar t = \langle
   t_{\bar c}:\bar c \in {}^{\omega >}H\rangle$ such that $t_{\bar c}
   \in$ {\rm def}$(G),\tp_{\bs}(\bar c,G,H) = q_t(G)$ and if
$\bar c^1,\bar c^2 \in {}^{\omega >}H,h:\ell g(\bar c^1) \rightarrow
\ell g(\bar c^2)$ and $\bar c^2 = 
\langle c^2_{h(i)}:i < \ell g(\bar c^1)\rangle$ then
   $t_{\bar c^1} \le_h t_{\bar c^2}$.  Moreover, even if there is such $\bar
   t$ we can ``amalgamate for it" but this is not enough as $\bar t$
 is not necessarily unique, which may give different results.
   Why \ref{a64}(3) is O.K.?  As in Definition \ref{a61}(3) we ask ``for
   every $\gs_1,\gs_2$".   In other words if $G_0 \subseteq G_1,G_0
   \subseteq G_2$ and $t_1,t_2 \in \deef(G_0),\tp_{\bs}(\bar c_\ell,
G_0,G_2) = q_{t_\ell}(G_0)$ for $\ell=1,2$ but 
$q_{t_1}(G_1) \ne q_{t_2}(G_1)$ we can amalgamate as
   in \ref{a61}(3).}  \then \, there
   is one and only one isomorphism from $G^1_3$ onto $G^2_3$ extending
   $f_1 \cup f_2$ (which is well defined by (2)).

\noindent
4) (symmetry): $\NF^2_{\gS}(G_0,G_1,G_2,G_3)$ iff 
$\NF^2_{\gS}(G_0,G_2,G_1,G_3)$.

\noindent
5) (monotonicity): If $\NF^2_{\gS}(G_0,G_1,G_2,G_3)$ and $G_0
   \subseteq G'_\ell \subseteq G_\ell$ for $\ell=1,2$ \then \, 
$\NF^2_{\gS}(G_0,G'_1,G'_2,G_3)$.

\noindent
6) (existence): If $G_0 \le_{\gS} G_\ell$ for $\ell=1,2$ and $G_0$ is
   existentially closed and $G_1 \cap G_2 = G_0$ \then
   \, for some $G_3 \in \mathbf K$ we have $\NF^2_{\gS}(G_0,G_1,G_2,G_3)$.
\end{claim}

\begin{remark}
For parts (3) and (6) of \ref{a64} recall: for such $G$, 
if $t_1,t_2 \in \deef(G),q_{t_1}(G) = q_{t_2}(G)$ and 
$G \subseteq G^+ \in \mathbf K$ then $q_{t_1}(G^+) = q_{t_2}(G^+)$. 
\end{remark}

\begin{PROOF}{\ref{a64}}
 Straightforward, e.g. for disjointness (= part (2)) use Claim \ref{a15}(4).  
\end{PROOF}

\noindent
Alternative to \S1 from \ref{a25} on is: we repeat it with 
changes being that we use $\otimes$ instead of $\oplus$ and we incorporated the
$\lambda$-fullness, also in \ref{a70}(3) we choose another version.
We have not sorted out whether we can generalize \ref{a34}(5) based on
\ref{a33} and \ref{a41}(2).
\begin{definition}
\label{a67}
1) We say that $\cA$ is a \underline{one step} 
$(\lambda,\gS)-\otimes$-construction \when \,
 $\cA = (G,H,\langle \bar c_\alpha,t_\alpha:\alpha <
 \alpha(\cA) = \alpha_{\cA}\rangle)$ satisfies:
\mn
\begin{enumerate}
\item[$(a)$]  $G \subseteq H \in \mathbf K$
\sn
\item[$(b)$]  $t_\alpha \in \deef_{\gS}(G)$ for $\alpha < \alpha(\cA)$;
\sn
\item[$(c)$]  if $\alpha_0,\dotsc,\alpha_{n-1} < \alpha(\cA)$ with no
repetitions then $\bar c_{\alpha_0} \char 94 \ldots \char 94 \bar
c_{\alpha_{n-1}}$ realizes $q_t(G_0)$ in $H$ where $t = 
t_{\alpha_0} \otimes \ldots \otimes t_{\alpha -1} \in \deef(G)$;
\sn
\item[$(d)$]  $H = \langle\cup\{\bar c_\alpha:\alpha <
\alpha(\cA)\} \cup G \rangle_H$;
\sn
\item[$(e)$]  $\langle t_\alpha:\alpha < \alpha(\cA)\rangle$ 
lists $\deef_{\gS}(G)$ each appearing exactly $\lambda$ times.
\end{enumerate}
\mn
2) In (1) we may use any index set instead of $\alpha(\cA)$,
e.g. $\deef_{\gS}(G)$ itself when $\lambda=1,\deef_{\gS}(G) \times
\lambda$ in general.

\noindent
3) We say $\cA$ is an
$\alpha(\cA)$-\underline{step}-$(\lambda,\gS)-\otimes$-construction or
$(\alpha(\cA),\lambda,\gS)-\otimes$-construction \when \,:
\mn
\begin{enumerate}
\item[$(a)$]  $\cA = \langle G_\alpha,\langle \bar
c_{\beta,s},t_{\beta,s}:s \in S_\beta\rangle:\alpha \le \alpha(\cA),
\beta < \alpha(\cA)\rangle$
\sn
\item[$(b)$]  $(G_\alpha:\alpha \le \alpha(\cA)\rangle$ is increasing
continuous (in $\mathbf K$)
\sn
\item[$(c)$]  $(G_\alpha,G_{\alpha +1},\langle \bar
c_{\alpha,s},t_{\alpha,s}:s \in S_\alpha\rangle)$ is a one step
$(\lambda,\gS)-\otimes$-construction.
\end{enumerate}
\mn
4) In part (3), let $G^{\cA}_\alpha = G_\alpha[\cA]$ be $G_\alpha$,
etc., and in part (1) let $G^{\cA} = G[\cA]$ be $G$, etc.

\noindent
5) In part (3) if $\alpha(\cA) = \omega$ then we may omit it; also
 for every $\alpha < \alpha(\cA)$ the sequence
   $(G^{\cA}_\alpha,G^{\cA}_{\alpha +1},\langle \bar
   c_{\alpha,s},t_{\alpha,s}:s \in S^{\cA}_\alpha\rangle)$ is called
the $\alpha$-th step of $\cA$.
\end{definition}

\begin{definition}
\label{a70}
1) We say $H$ is a $\lambda$-full \underline{one step} 
$\gS-\otimes$-closure of $G$ \when \, there is a one step
$(\lambda,\gS)-\otimes$-construction $\cA$ such that $G[\cA] =
G,H[\cA]=H$.  We may say $H$ is $\lambda$-full one step
$\gS-\otimes$-constructible over $G$; similarly in part (2).

\noindent
2) We say $H$ is $\lambda$-full $\alpha$-step $\gS$-closure
over $G$ or $H$ is $(\alpha,\lambda,\gS)$-closure of $G$ 
\when \, there is a $(\alpha,\lambda,\gS)-\otimes$-construction $\cA$
with $G = G^{\cA}_0,H =  G^{\cA}_{\ell g(\cA)}$.

\noindent
3) We say $G_*$ is $(\delta,\lambda,\gS)-\otimes$-full over $G$ \when
\, for some $\bar G = \langle G_i:i \le \delta\rangle$ increasing continuous
   sequence in $\mathbf K,G_0 = G,G_\delta = G_*$ and $G_{i+1}$ is
   $(1,\lambda,\gS)-\otimes$-full over $G_i$ which means some
   $G' \subseteq G_{i+1}$ is a one step $(\lambda,\gS)-\otimes$-construction
over $G_i$.  If $\delta = \omega$ one may omit it writing
$(\lambda,\gS)$ instead of $(\delta,\lambda,\gS)$.

\noindent
4) We may in part (3) replace $\otimes$ by $\oplus$. 
\end{definition}

\begin{claim}
\label{a74} 
Assume $\gS \subseteq \Omega[\mathbf K]$ is $\otimes$-closed, $\alpha$
an ordinal, $\lambda$ a cardinal.

\noindent
1) If $G \in \mathbf K$  \then \, there is a one step 
$(\lambda,\gS)-\otimes$-construction $\cA$ over $G$ (i.e. $G^{\cA}_0 = G$) of
   cardinality $\le \lambda + |G| + |\gS|$ and $\ge \lambda$.

\noindent
2) If in part (1), $\cA_1,\cA_2$ are 
one step-$(\lambda,\gS)-\otimes$-constructions over $G$ \then \,
$H[\cA_1],H[\cA_2]$ are isomorphic over $G$.

\noindent
3) For any $G \in \mathbf K$ there is an
$(\alpha,\lambda,\gS)-\otimes$-construction $\cA$ over 
$G$ and $G_\alpha[\cA]$ is unique up to isomorphism over $G$.

\noindent
4) If $\gS$ is dense, $H$ is an $(\alpha,\lambda,\gS)-\otimes$-closure of
 $G$ and $\alpha$ is a limit ordinal \then \, $H$ is existentially
 closed and is $(\alpha,\lambda,\gS)$-full over $G$.
\end{claim}

\begin{PROOF}{\ref{a74}}
Straightforward, as in \ref{a41}(3).
\end{PROOF}

\begin{discussion}
\label{a76}
Essentially we know that if ``$G_1 \subseteq G_2$" implies the
$(\alpha,\lambda,\gS)$-closure of $G_1$ is a subgroup of the
$(\alpha,\lambda,\gS)$-closure of $G_2$.

But we have a delicate problem: what if the
$(\alpha,\lambda,\gS)$-closure of $G_1$ is not disjoint to $G_2
\backslash G_1$?

We have similar problems with ``the algebraic closure of a field" or
``the field of quotients of a field", but there if $G_1 \subseteq G_2$
then the closure $G^+_1$ of $G_1$ inside $G_1$ is definable (from
$G_2,G_1$ and $G^+_2$).  Here this is not true, but clearly this is
not a serious problem.  Ways to circumvent this appear in \ref{y8}(2),
\ref{a26}(2) and below.
\end{discussion}

\begin{claim}
\label{a78}
1) We can choose $\hat G \in \mathbf K_{\exlf}$ such that $\hat G$
extends $G \in \mathbf K_{\lf},G_1 \cong G_2 \Rightarrow \hat G_1 \cong
\hat G_2$ and every embedding $f:G_1 \rightarrow G_2 \in \mathbf K_{\lf}$ can
be extended to $\hat f:\hat G_1 \rightarrow \hat G_2$ canonically.

\noindent
1A) Moreover $G_1 \subseteq G_2 \Rightarrow \hat G_1 \subseteq \hat
G_2$ but pedantically see (2).

\noindent
2) There is a set theoretic class function $\mathbf F$, that 
computes from $G \in \mathbf K,\alpha \in \Ord,\lambda \in \Card,\gamma
\in \Ord$ and $\gS \subseteq \Omega[\mathbf K]$ a 
group $H = \mathbf F(G,\alpha,\gS,\gamma)$ such that:
\mn
\begin{enumerate}
\item[$(a)$]  $\mathbf F(G,\alpha,\gS,\gamma) \in \mathbf K$ extends $G$,
  moreover;
\sn
\item[$(b)$]  $\mathbf F(G,\alpha,\gamma,\gS)$ is an
  $(\alpha,\lambda,\gS)$-closure of $G$;
\sn
\item[$(c)$]  [uniqueness]: if $G_1,G_2 \in \mathbf K$ and $g$ is an
  isomorphism from $G_1$ onto $G_2$ and $H_\ell = 
\mathbf F(G,\alpha,\gS,\gamma)$ for $\ell=1,2$ \then \, there is an
isomorphism $g$ from $H_1$ onto $H_2$ extending $g$;
\sn
\item[$(d)$]  we have $H_1 \subseteq H_2$ and $G_1 = H_1 \cap G_2$
  \when \, $G_1 \subseteq G_1 \in \mathbf K,\gamma > \alpha$
  and\footnote{Recalling tr-cl is the (set-theoretic) transitive closure.}
$\gamma > \sup(\Ord \cap \tr-c \ell(G_\ell))$ for $\ell=1,2$ and $H_\ell =
\mathbf F(G_\ell,\alpha,\gS,\gamma)$;
\sn
\item[$(e)$]  if we restrict ourselves to $G \in \mathbf K' = \{G \in
  \mathbf K$: if $x \in G$ then $x$ is a singleton$\}$ \then \, $G_1
  \subseteq G_2 \Rightarrow \mathbf F(G,\alpha,\gS) \subseteq
\mathbf F(G,\alpha,S,0)$.
\end{enumerate}
\end{claim}
\bigskip

\centerline {$* \qquad * \qquad *$}
\bigskip

In \S4,\S5 we intend to use also some relative of those constructions,
including:
\begin{definition}
\label{a82}
Assume $\bar H = \langle H_i:i < \delta\rangle$ is
$\subseteq$-increasing in $\mathbf K$ and $H_\delta = \cup\{H_i:i <
\delta\}$, (we shall use $\delta = \omega$).  We say $\cA$ is a one step
atomic $\gS-\otimes$-construction above $\bar H$, 
\when \, (and we may say $H$ is weakly atomically
$\gS-\otimes$-constructible over $\bar H$, omitting $\bar H$ means 
for some $\bar H$ of length $\omega$ and we may replace 
$\alpha_{\cA} = \alpha(\cA)$ by any index set) $\cA$ 
has the following objects satisfying the following additional conditions:
\mn
\begin{enumerate}
\item[$(A)$]  $(\bar H,H_\delta,H,\langle \bar c_\alpha,t_{\alpha,i},\alpha <
\alpha_{\cA},i < \delta \rangle)$;
\sn
\item[$(B)$]  $H_\delta \subseteq H \in \mathbf K$;
\sn
\item[$(C)$]  $t_{\alpha,i} \in \deef_{\gS}(H_i)$;
\sn
\item[$(D)$]  $H = \langle \cup\{\bar c_\alpha:\alpha < \alpha_{\cA}\}
\cup H_\delta \rangle_H$;
\sn
\item[$(E)$]  $\bar c_\alpha$ realizes $q_{t_{\alpha,i}}(H_i)$ in $H$ for
$\alpha < \alpha_{\cA},i < \delta$; 
\sn
\item[$(F)$]  $\bar c_{\alpha,i} \subseteq H_{i+1}$ realizes 
$q_{t_{\alpha,i}}(H_i)$ for $i < \delta,\alpha < \alpha_{\cA}$
and moreover;
\sn
\item[$(F)^+$]  assuming 
$\alpha(0) < \ldots < \alpha(n-1) < \alpha_{\cA}$ 
and $\ell g(\bar x_\alpha) = \ell g(\bar c_\alpha)$ and 
\newline
$\varphi =
\varphi(\bar x_{\alpha(0)},\dotsc,x_{\alpha(n-1)},\bar y)$ we
have\footnote{Yes!  $\tp_{\at}$ and not $\tp_{\bs}$.}
\newline
$\varphi(\bar x_{\alpha(0)},\dotsc,\bar x_{\alpha(n-1)},\bar b) 
\in \tp_{\text{at}}(\bar c_{\alpha(0)}
\char 94 \ldots \char 94 \bar c_{\alpha(n-1)},G_\delta,H)$ 
\newline
\Iff \, $\bar b \subseteq {}^{\ell g(\bar y)}G_\delta$ and 
for every permutation $\pi$ of $n$,
\newline
$(\forall^\infty i(0) < \delta)(\forall^\infty i(1) < \delta),\ldots,
(\forall^\infty i(n-1) < \delta)$
\newline 
$\varphi[\bar c_{\alpha(0),i(\pi(0)},\bar c_{\alpha(1),i(\pi(1))},
\dotsc,\bar c_{d(n-1),\pi(n-1)},\bar b]$
\newline
(used in the proof
of $(*)_{5.2}$ stage C in the proof of \ref{p73}); note that $\varphi$
is not necessarily atomic.
\end{enumerate}
\end{definition}

\begin{remark}
\label{a84}
1) We may consider replacing clause $(F)^+$ by:
\mn
\begin{enumerate}
\item[$(F)^*$]   $\bar c_{\alpha(0)} \char 94 \ldots \char 94 \bar
c_{\alpha(n-1)}$ realizes $q_{t_{\alpha(0)} \otimes \ldots \otimes
t_{\alpha(n-1)}}$ for $\alpha(0) < \ldots < \alpha(n-1) <
\alpha(\cA)$.  
\end{enumerate}
\mn
2) In this alternative version we do not need the existence of $\bar
c_{\alpha,i} \subseteq H_{i+1}$, so it is easier to prove existence
but the version above is the one we actually use.  In particular the
version in (1) would create problems in $(*)_{5.7}$ in the proof of
\ref{p73}; we may try to take care of this by changing the definition
of $L^*_\beta$ there.

\noindent
3) A sufficient condition for having the assumptions of \ref{a82} appear
in \ref{c64}.
\end{remark}

\begin{observation}
\label{a86}
Let $\gS$ be closed and $\otimes$-closed.  Assume $\langle G_i:i \le
\alpha\rangle$ is $\subseteq$-increasing continuous in $\mathbf K$.

\noindent
1) In \ref{a67}(1) we can prove $G^{\cA} 
\le_{\gS} H^{\cA}$ and in \ref{a67}(2), we can prove
   $\langle G^{\cA}_\alpha:\alpha \le \alpha_{\cA}\rangle$ is
$\le_{\gS}$-increasing continuous.

\noindent
2) In \ref{a82}, if $\bar H$ is $\le_{\gS}$-increasing \then \, 
we have $i < \delta \Rightarrow H_i \subseteq_{\gS} H$.

\noindent
3) Assume $S$ is a set of limit ordinals $< \delta,\langle G_i:i \le
\delta\rangle$ is a $\subseteq$-increasing continuous sequence of
members of $\mathbf K$ and $G_{i+1}$ is a one step 
$\gS-\otimes$-constructible over $G_i$ for $i \in \delta 
\backslash S$ and $G_{i+1}$ is weakly one step
   $\gS-\otimes$-constructible over $\bar G \rest C_i$ for some
   unbounded $C_i \subseteq i \backslash S$ for each $i \in S$, (hence $i$ is a
limit ordinal).  \Then \, $i < j \le \delta \wedge i \notin S 
\Rightarrow G_i \le_{\gS} G_j$.  
\end{observation}

\begin{remark}
The idea of $\gs_1 \otimes \gs_2$ can be applied to one $\gs$ (and
is used in the end of the proof of $\boxplus_1$ in stage B the proof of
Theorem \ref{p73}). 
\end{remark}

\noindent
Toward this in \S4(B) we shall 
deal with finding such amalgamations and $\gs$'s.
\begin{dc}
\label{a89}
Assume $\gs \in \Omega[\mathbf K_{\lf}]$ and 
 $H_1 \subseteq H_2 \in \mathbf K$ are finite, 
$\bar a \in {}^{k(\gs)}(H_1),\bar c \in {}^{n(\gs)}(H_2)$
 and $\bar a,\bar c$ generate $H_1,H_2$ respectively, and $\bar a$
 realizes $p_{\gs}(\bar x_{\gs})$ in $H_1$ and $\bar c$ realizes
 $q_{\gs}(\bar a,H_1)$ in $H_2$.
Assume further $K$ is a group of automorphisms of $H_2$ mapping $H_1$
onto itself.  \Then \, there is a one and only one $\gt$ such that: 
\mn
\begin{enumerate}
\item[$(a)$]  $\gt \in \Omega[\mathbf K_{\lf}]$
\sn
\item[$(b)$]  $k(\gt) = k(\gs)$ and $p_{\gt}(\bar x_{\gt}) = 
\tp_{\qf}(\bar a,\emptyset,H_1)$
\sn
\item[$(c)$]  if $H_1 \subseteq G_1 \subseteq G_2,H_2 \subseteq G_2$ and
$\bar c$ realizes $q_{\gs}(\bar a,G_1)$ in $G_2$ and $\bar c' \in
{}^n(G_2)$ realizes $q_{\gt}(\bar a,G_2)$ \then \,
$\tp_{\at}(\bar c',G_1,G_2) = \cap\{\tp_{\at}(\pi(\bar c),G_1,G_2):
\pi \in K\}$.
\end{enumerate}
\end{dc}

\begin{remark}
Toward this in \S(4B) we deal with finding such amalgamations and $\gs$'s.
\end{remark}

\begin{PROOF}{\ref{a89}}
Straightforward.
\end{PROOF}
\newpage

\section {For fixing a distinguished subgroup}

In the construction of complete members of $\mathbf K_{\exlf}$
(and related aims) we fix large enough $\gS \subseteq \Omega[\mathbf K]$
and build a $\subseteq$-increasing continuous sequence $\langle
G_\alpha:\alpha < \lambda\rangle,|G_\alpha| < \lambda$; normally we
demand for $\alpha < \beta < \lambda$ that ``usually" $G_\alpha
\le_{\gS} G_\beta$ (i.e. except for $\delta \in S$, where $S
\subseteq S^\lambda_{\aleph_0}$).  But at some moment for $\alpha =
\delta + n$, we like to use $p = \tp_{\bs}(c,G_\alpha,G_{\alpha +1})$ which
extends some $r \in \mathbf S_{\bs}(K),
K \subseteq G_\alpha$ finite but such
that $c$ commutes with $G_\delta$.  Also toward this in \S(4A) we deal
with a relative $\NF^3$ of $\NF_f$, in which we demand $\mathbf C_{G_1}(G_3)$ is
large, this continues \S2 concentrating on the case $G_0$ is with
trivial center.  In \S(4B) we use this to define some schemes from
$\Omega[\mathbf K]$, see e.g. \ref{d93}.  

Another problem is that given $G_1$ instead of extending $G_1$ to
$G_2$ such that $q_t(G_1)$ is realized by $\bar c \in {}^{\omega
>}(G_2)$ for some $t \in \deef_{\gS}(G_1)$, we like to have an
infinite $\bar c = (\ldots \char 94 \bar c_i \char 94 \ldots)_{i \in
I}$, with $\tp(\bar c \rest u,G_1,G_2) \in q_{t_u}(G_1)$ for every
finite $u \subseteq I$; used in stage D of the proof of Theorem
\ref{p73}.  This is done in \S4(C).  

\subsection {Preserving Commutation} \
\bigskip

\begin{claim}
\label{d36}
The subgroups $H'_1,H'_2$ of $G_3$ commute \when \,:
\mn
\begin{enumerate}
\item[$(*)$]  
\begin{enumerate}
\item[(a)]  $\mathbf x \in \mathbf X_{\mathbf K}$;
\sn
\item[(b)]  $G_\ell = G_{\mathbf x,\ell},G'_\ell = 
\mathbf j_{\mathbf x,\ell}(G_\ell)$ for $\ell = 0,1,2$;
\sn
\item[(c)]  $G_3 = G_{\mathbf x}$;
\sn
\item[(d)]  $H_1 \subseteq G_1$ and
$H'_1 = \mathbf j_{\mathbf x,1}(H_\ell)$;
\sn
\item[(e)]  $H_1 = \cup\{b(H_1 \cap G_0):b \in \mathbf I_1\}$ 
where $\mathbf I_1 = \mathbf I_{\mathbf x,1} \cap H_1$;
\sn
\item[(f)]   if $g \in \mathbf I_{\mathbf x,1}$ and\footnote{As $G_1$ is
    locally finite, necessarily $\mathbf I_1$ is a subgroup of $H_1$.}
 $b \in \mathbf I_1$ then $gb \in \mathbf I_{\mathbf x,1}$;
\sn
\item[(g)]  the subgroups $G_0,H_1$ of $G_1$ commute;
\sn
\item[(h)]  $H_2 \subseteq G_2$ commutes with
$G_0 \cap H_1$ and $H'_2 = \mathbf j_{\mathbf x,2}(H_2)$;
\sn
\item[(i)]  $H_2 = \cup\{a(G_0 \cap H_2):a \in \mathbf I_2\}$ where
$\mathbf I_2 = \mathbf I_{\mathbf x,2} \cap H_2$.
\end{enumerate}
\end{enumerate}
\end{claim}

\begin{remark}
\label{d37}
1) Really here it suffices to deal with the case $G_0 \cap H_1 = \{e\}$.

\noindent
2) A natural case is $\mathbf Z(G_0) = \{e_{G_0}\},H_1 = 
\mathbf C_{G_1}(G_0),H_2 = G_2$.

\noindent
3) See the proof of \ref{p73}.
\end{remark}

\begin{notation}
\label{d38}
Let $\mathbf X^3_{\lf} = \mathbf X^3_{\mathbf K_{\lf}}$ be the class 
of tuple $(\mathbf x,H_1,H_2)$ which
satisfies $(*)$ of Claim \ref{d36}.
\end{notation}

\begin{PROOF}{\ref{d36}}
Let $a \in H_2,b \in H_1,f_a = \mathbf j_{\mathbf x,2}(a),f_b = \mathbf
j_{\mathbf x,1}(b)$, so by $(*)(d),(h)$ we just 
have to prove that $f_b f_a((g_0,g_1,g_2)) =
f_a f_b((g_0,g_1,g_2))$ for any $(g_0,g_1,g_2) \in \cU_{\mathbf x}$. 

Clearly
\mn
\begin{enumerate}
\item[$\bullet$]   if $a \in G_0$ or $b \in G_0$ this holds.  
\end{enumerate}
\mn
[Why?  First, if $a \in G_0$ then $f_a = \mathbf j_{\mathbf x,2}(a) = \mathbf
j_{\mathbf x,0}(a) = \mathbf j_{\mathbf x,1}(a) \in \mathbf j_{\mathbf x,1}
(G_1) = G'_1 \subseteq G_{\mathbf x}$ 
and as $b \in H_1 \subseteq G_{\mathbf x}$, by
$(*)(g)$ we have $G_1 \models ``a,b$ commute" hence $G_{\mathbf x}
\models ``\mathbf j_{\mathbf x,2}(a),\mathbf j_{\mathbf x,1}(b)$ commute" and so
$G_{\mathbf x} \models ``f_a,f_b$ commute".  Second, if $b \in G_0$ then
$b \in G_0 \cap H_1 \subseteq G_0 \subseteq G_2$ and $a \in H_2
\subseteq G_2$, so by clause $(*)(h)$ clearly $G_2 
\models ``a,b$ commute" and we finish as above.]

Moreover, as $H_1 = \langle (G_0 \cap H_1) \cup \mathbf I_1 \rangle_{G_1}$ by
clause $(*)(e)$, recalling $\bullet$ above, without loss of generality 
\mn
\begin{enumerate}
\item[$\boxplus_1$]  $b \in \mathbf I_1 \subseteq \mathbf I_{\mathbf x,1}$.
\end{enumerate}
\mn
Similarly as $H_2 = \langle (G_0 \cap H_2) \cup \mathbf I_1\rangle$.
By clause $(*)(i)$, recalling $\bullet$ above \wilog \,:
\mn
\begin{enumerate}
\item[$\boxplus_2$]  $a \in \mathbf I_2 \subseteq \mathbf I_{\mathbf x,2}$.
\end{enumerate}
\mn
Let\footnote{Note that $g^x_\ell$ is
  \underline{not} conjugation by $x$.} 
$f_x((g_0,g_1,g_2)) = (g^x_0,g^x_1,g^x_2)$ and $f_y
f_x((g_0,g_1,g_2)) = (g^{x,y}_0,g^{x,y}_1,g^{x,y}_2)$ for $x \in
\{a,b\}$ and $y \in \{a,b\} \backslash \{x\}$.

We shall prove that $g^{a,b}_\ell = g^{b,a}_\ell$ for $\ell=0,1,2$;
this suffices.

Clearly,
\mn
\begin{enumerate}
\item[$\bullet_1$]  $g^a_1 = g_1$ and $g_2 g_0 a = g^a_2 g^a_0$;
\sn
\item[$\bullet_2$]  $g^{a,b}_2 = g^a_2$ and $g^a_1 g^a_0 b = g^{a,b}_1
g^{a,b}_0$;
\sn
\item[$\bullet_3$]  $g^b_2 = g_2$ and $g_1 g_0 b = g^b_1 g^b_0$;
\sn
\item[$\bullet_4$]  $g^{b,a}_1 = g^b_1$ and $g^b_2 g^b_0 a = g^{b,a}_2
g^{b,a}_0$. 
\end{enumerate}
\mn
Now
\mn
\begin{enumerate}
\item[$\boxplus_3$]  $g^{a,b}_1 G_0 = 
g^{a,b}_1 g^{a,b}_0 G_0 = g^a_1 g^a_0 b G_0 = 
(g^a_1 b)(g^a_0 G_0) = (g^a_1 b) G_0$.
\end{enumerate}
\mn
[Why?  As $g^{a,b}_0 \in G_0$, by the second statement of $\bullet_2$, 
noting that $b,g^a_0$ commute by $(*)(g)$, and as $g^a_0 \in G_0$,
respectively.]

But $g^a_1 \in \mathbf I_{\mathbf x,1}$
(as $(g^a_0,g^a_1,g^a_2) \in \cU_{\mathbf x}$), and $b \in \mathbf I_1
\subseteq \mathbf I_{\mathbf x,1}$ by $\boxplus_1$, hence by $(*)(f)$ we have 
$g^a_1 b \in \mathbf I_{\mathbf x,1}$ and also $g^{a,b}_1 \in
\mathbf I_{\mathbf x,1}$ (as $(g^{a,b}_0,g^{a,b}_1,g^{a,b}_2) \in \cU_{\mathbf
x}$).  Now by $\boxplus_3$, $g^{a,b}_1 G_0 = (g^a_1 b)G_0$ and by the
last sentence $g^{a,b}_1,g^a_1 \in \mathbf I_{\mathbf x,1}$ and thus
\mn
\begin{enumerate}
\item[$\bullet_5$]  $g^{a,b}_1 = g^a_1 b$.
\end{enumerate}
\mn
So by $\bullet_5$ and the second equation in $\bullet_2$ we have
$g^a_1 b g^{a,b}_0 = g^{a,b}_1 g^{ab}_0 = g^a_1 g^a_0 b = g^a_1 b
g^a_0$, the last equality by recalling $b,g^a_0$ commute
by $(*)(g)$, hence we have:
\mn
\begin{enumerate}
\item[$\bullet_6$]  $g^{a,b}_0 = g^a_0$.
\end{enumerate}
\mn
Similarly to $\boxplus_3$ we have 
\mn
\begin{enumerate}
\item[$\boxplus_4$]  $g^b_1 G_0 = 
g^b_1 g^b_0 G_0 = g_1 g_0 b G_0 = (g_1 b)(g_0 G_0) = (g_1 b)G_0$.
\end{enumerate}
\mn
[Why?  As $g^b_0 \in G_0$, by $\bullet_3$ second statement,
 as $b,g_0$ commute by $(*)(g)$, and as $g_0 \in G_0$ respectively.]

Also $g_1 \in \mathbf I_{\mathbf x,1}$ as $(g_0,g_1,g_2) \in \cU_{\mathbf x}$ and
$b \in \mathbf I_1$ by $\boxplus_1$ so recalling $(*)(f)$ we deduce
$g_1,g_1 b \in \mathbf I_{\mathbf x,1}$ thus from $\boxplus_4$ we deduce:
\mn
\begin{enumerate}
\item[$\bullet_7$]  $g^b_1 = g_1 b$.
\end{enumerate}
\mn
Hence by $\bullet_7$ and $\bullet_3$ second statement we have $g_1 b
g^b_0 = g^b_1 g^b_0 = g_1 g_0 b = g_1 b g_0$, the last equation
 recalling $b,g_0$ commute (by $(*)(g)$), hence we have:
\mn
\begin{enumerate}
\item[$\bullet_8$]  $g^b_0 = g_0$.
\end{enumerate}
\mn
So by $\bullet_4,\bullet_7,\bullet_1,\bullet_6$, $b$ commuting with
$G_0$ and $\bullet_2$ second statement respectively, we have
\mn
\begin{enumerate}
\item[$\boxplus_5$]  $g^{b,a}_1 = g^b_1 = (g_1 b) = (g^a_1 b) = (g^a_1 b)
(g^a_0(g^{a,b}_0)^{-1}) = (g^a_1 g^a_0
b)(g^{a,b}_0)^{-1} = g^{a,b}_1$,
\end{enumerate}
\mn
and thus
\mn
\begin{enumerate}
\item[$\bullet_9$]  $g^{b,a}_1 = g^{a,b}_1$.
\end{enumerate}
\mn
Also by $\bullet_4,\bullet_3,\bullet_8,\bullet_1,\bullet_6,\bullet_2$
we have
\mn
\begin{enumerate}
\item[$\boxplus_6$]  $g^{b,a}_2 g^{b,a}_0 = g^b_2 g^b_0 a =
g_2 g^b_0 a = g_2 g_0 a = g^a_2 g^a_0 = g^a_2 g^{a,b}_0 =  
g^{a,b}_2 g^{a,b}_0$.
\end{enumerate}
\mn
So
\begin{enumerate}
\item[$\bullet_{10}$]  $g^{b,a}_2 g^{b,a}_0 = g^{a,b}_2 g^{a,b}_0$
\end{enumerate}
\mn
but $g^{b,a}_0,g^{a,b}_0 \in G_0$ and $g^{b,a}_2,g^{a,b}_2 \in \mathbf
I_{\mathbf x,2}$ hence recalling
$(g^{a,b}_0,g^{a,b}_1,g^{a,b}_2),(g^{b,a}_0,g^{b,a}_1,g^{b,a}_2) \in
\cU_{\mathbf x}$ we have:
\mn
\begin{enumerate}
\item[$\bullet_{11}$]  $g^{b,a}_2 = g^{a,b}_2$ and $g^{b,a}_0 = g^{a,b}_0$.
\end{enumerate}
\mn
But $\bullet_{11} + \bullet_9$ imply that we are done.
\end{PROOF}

The following claim is like Definition \ref{c10}, but now we preserve  
a large $\mathbf C_{G_1}(G_0)$ using \ref{d36}.
\begin{definition}
\label{d39}
Let $\NF^3(\bar G,H_1,L,H_2)$ mean:
\mn
\begin{enumerate}
\item[(A)] 
\begin{enumerate}
\item[(a)]  $\bar G = \langle G_\ell:\ell \le 3\rangle$
are from $\mathbf K_{\lf}$;
\sn
\item[(b)]  $G_0  \subseteq G_\ell$ for $\ell = 1,2$;
\sn
\item[(c)]  $G_0$ is finite;
\sn
\item[(d)]  $H_1 \subseteq \mathbf C_{G_1}(G_0),L
\subseteq H_1,L \cap G_0 = \{e_{G_0}\},H_1 = \langle L,G_0 \cap
H_1\rangle_{G_1}$;
\sn
\item[(e)]  $G_1 \cap G_2 = G_0$;
\sn
\item[(f)]  $H_2 \subseteq \mathbf C_{G_2}(H_1 \cap G_0)$;
\end{enumerate}
\sn
\item[(B)]
\begin{enumerate}
\item[(a)]  $G_\ell \subseteq G_3$ for $\ell = 1,2$;
\sn
\item[(b)]   for $\sigma(\bar x,\bar y)$ a
group-term, $\bar a \in {}^{\ell g(\bar x)}(G_1)$ and $\bar b \in
{}^{\ell g(\bar y)}(G_2)$ the following conditions are equivalent:
\begin{itemize}
\item  $G_3 \models ``\sigma(\bar a,\bar b) = e_{G_3}"$,
\sn
\item   if $(\mathbf x,H_1,H_2) \in \mathbf X^3_{\lf}$, see \ref{d38},
$G_\ell = G_{\mathbf x,\ell}$ for $\ell= 0,1,2$ 
and $\bar a' = \mathbf j_{\mathbf x,1}(\bar a)$ and\footnote{We may add
  $\mathbf I_1 = L$.} $\bar b' = \mathbf j_{\mathbf x,2}(\bar b)$
\then \, $G_{\mathbf x} \models ``\sigma(\bar a',\bar b') = e_{G_{\mathbf x}}"$.
\end{itemize}
\end{enumerate}
\end{enumerate}
\end{definition}

\begin{convention}
\label{d40}
In  \ref{d39}, if $H_1 = L$ we may in addition omit $L$.  We may omit
$L,H_2$ when $L=H_1,H_2 = \mathbf C_{G_2}(H_1 \cap G_0)$.  Lastly, if
$\mathbf Z(G_0) = \{e_{G_0}\},L = \mathbf C_{G_1}(G_0)$ and $H_1 = L$ and
$H_2 = G_2$, then we may omit $H_1,L$ and $H_2$; see \ref{d42}(3) below.
\end{convention}

\begin{claim}
\label{d42}
Assume $\bar G = \langle G_\ell:\ell < 3 \rangle,
H_1,L,H_2$ are as in \ref{d39}(A).

\noindent
1) We can find $\mathbf x$ such that $(\mathbf x,H_1,H_2) \in \mathbf X^3_{\lf}$.

\noindent
2) There is $G_3 \in \mathbf K$ such that $\NF^3(\langle
   G_0,G_1,G_2,G_3\rangle,H_1,L,H_2)$ and $G_3$ is unique up to
   isomorphism over $G_1 \cup G_2$.

\noindent
3) If $\bar G$ satisfies (A)(a),(b),(c) of Definition \ref{d39},
$\mathbf Z(G_0) = \{e_{G_0}\},H_1 = L = \mathbf C_{G_1}(G_0)$ and $H_2 =
G_2$, \then \, $(\bar G,H_1,L,H_2)$ satisfies \ref{d39}(A).

\noindent
4) The relation $\NF^3(\bar G),\bar G = \langle G_\ell:\ell \le
3\rangle$ satisfies the parallel of \ref{c18} omitting symmetry, so
having uniqueness, monotonicity and both sides definability, i.e. $G_1
\le_{\Omega[\mathbf K]} G_3,G_2 \le_{\Omega[\mathbf K]} G_3$.
\end{claim}

\begin{PROOF}{\ref{d42}}
1) It suffices to prove we can choose 
$\mathbf I^*_1,\mathbf I^*_2$ satisfying the demands
on $\mathbf I_{\mathbf x,1},\mathbf I_{\mathbf x,2}$ in \ref{d36}.

\noindent
Why can we do it?  For $\mathbf I^*_2$ the demands are just clauses
(b),(c) from \ref{c3}(1) and $(*)(i)$ of \ref{d36} so just choose
$\mathbf I_2 \subseteq H_2$ such that $e_{G_0} \in \mathbf I_2$ and
$\langle g(G_0 \cap H_2):g \in\mathbf I_2 \rangle$ is a partition of
$H_2$ and then let $\mathbf I^*_2$ be such that $\mathbf I_2 \subseteq 
\mathbf I^*_2 \subseteq G_2$ and $\langle g G_0:g \in \mathbf
I^*_2\rangle$ is a partition of $G_2$.  

For $\mathbf I^*_1$ we have to take care of clauses (b),(c) from \ref{c3}(1),
 of $(*)(e)$ (the parallel of $(*)(i)$) and of $(*)(f)$ from
 \ref{d36}.  For this let $H^+_1 := \langle G_0,H_1\rangle_{G_1}$.  
First, choose $\mathbf I'_1 = L$ so clearly $e_{G_0} \in \mathbf I'_1$ 
and thus $\langle g G_0:g \in \mathbf I'_1\rangle$ is a partition of
$H^+_1$.  Why?  Recalling that $L \subseteq H_1 \subseteq G_1,L \cap G_0 
= \{e_{G_0}\}$ and $H_1 = \langle L,G_0 \cap H_1\rangle_{G_1}$ and
$H_1$ commute with $G_0$ in $G_1$;
 by clause (A)(d) we know that this is satisfied.  
Also let $\mathbf J_1 \subseteq G_1$ be such that $e_{G_0}
= e_{G_1} \in \mathbf J_1$ and $\langle g H^+_1:g \in \mathbf J_1\rangle$
is a partition of $G_1$.  Now let 
$\mathbf I^*_1 = \{gb:g \in \mathbf J_1$ and $b \in \mathbf I'_1\}$.

Clearly $\langle g G_0:g \in \mathbf I^*_1\rangle = \langle g(b G_0):b \in
\mathbf I'_1,g \in \mathbf J_1\rangle$ is a partition of $G_1$ (refining
$\langle g H^+_1:g \in \mathbf J_1\rangle)$, so clause \ref{c3}(1)(b)
holds.  Furthermore, $\mathbf I^*_1 \cap H^+_2 = L = \mathbf I'_1$
 so clause \ref{d36}(e) holds.

Next as $e_{G_0} \in \mathbf J_1$ and $e_{G_0} \in \mathbf I'_1$ clearly
$e_{G_0} \in \mathbf I^*_1$, so $\mathbf I^*_1$ satisfies clause
\ref{c3}(1)(c).  Also if $g \in \mathbf I^*_1 \wedge b \in \mathbf
I'_1$ then for some $g_1 \in \mathbf J_1,b_1 \in \mathbf I'_1$ we have
$G_1 \models ``g = g_1 b_1"$ hence $G_1 \models ``gb = (g_1 b_1)b =
g_1(b_1 b)"$ and recall $g_1 \in \mathbf J_1$ and $b_1 b \in \mathbf I'_1$ as
$\mathbf I'_1 = L$ is closed under products.  Thus together $gb \in \mathbf
I^*_1$, hence clause \ref{d36}(1)(f) is satisfied.  So $\mathbf I^*_1,\mathbf
I^*_2$ are as required in \ref{c3}(1) and \ref{d36}.  Hence 
there is $\mathbf x \in \mathbf X_{\mathbf K}$ such that 
$G_{\mathbf x,\ell} = G_\ell$ for
$\ell=0,1,2$ and $\mathbf I_{\mathbf x,\ell} = \mathbf I^*_\ell$ for
$\ell=1,2$.

\noindent
2) Consider clause (B) of \ref{d39}, the ``if $\mathbf x \in \ldots"$ is
   not empty so $G_3$ is a well defined group.  Easily $G_1 \subseteq
   G_3$ and $G_2 \subseteq G_3$ but is $G_3$ locally finite?  This
   follows from the results in \S2, in particular \ref{c18}.  That is, 
as there if $G'_\ell$ is finite, $G_0 \subseteq G'_\ell \subseteq
   G_\ell$ for $\ell=1,2$ \then \, we have finitely many possible
   choices of $(\mathbf I_{\mathbf x,1} \cap x_1 G'_1,\mathbf I_{\mathbf x,2} \cap
   x_2 G'_2)$ for $x_1 \in G_1,x_2 \in G_2$
hence the group $G_3$ that we get is locally finite.  Probably better
this is $G'_3$ such that $\NF_f(G_0,G_1,G_2,G'_3)$, by the definition
there is a homomorphism from $G'_3$ onto $G_3$ over $G_1 \cup G_2$.
Now as $G'_3$ is $\lf$ so is $G_3$.

\noindent
3),4)  Should be clear.
\end{PROOF}
\bigskip

\subsection {Schemes and derived sets} \
\bigskip

\begin{definition}
\label{d89}
1) Let $\mathbf X_0$ be the set of $\mathbf x$ such that:
\mn
\begin{enumerate}
\item[$(a)$]  $\mathbf x$ has the form $(K_1,K_2,\bar a_2,\bar a_1) =
(K_1[\mathbf x],K_2[\mathbf x],\bar a_2[\mathbf x],\bar a_1[\mathbf x])$;
\sn
\item[$(b)$]  $K_1 \subseteq K_2$ are finite groups;
\sn
\item[$(c)$]  $\bar a_1$ is a finite sequence generating $K_1$;
\sn
\item[$(d)$]  $\bar a_2$ is a finite sequence from $K_2$ such that
$\bar a_2 \char 94 \bar a_1$ generates $K_2$ (if $\bar a_2 = \langle
a_2 \rangle$ we may write just $a_2$);
\sn
\item[$(e)$]  $K_1$ has trivial center.
\end{enumerate}
\mn
2) Let $\mathbf X_1$ be the set of $\mathbf x$ such that:
\mn
\begin{enumerate}
\item[$(a)$]  $\mathbf x = (K,\bar a) = (K[\mathbf x],\bar a[\mathbf x])$;
\sn
\item[$(b)$]  $K \in \mathbf K$ is finite; 
\sn
\item[$(c)$]  $\bar a$ is a finite sequence from $K$ generating
  $K,\ell g(\bar a) \ge 1$; let $a_* = a_*[\mathbf x] = a_0$, 
the first element of $\bar a$.
\end{enumerate}
\mn
3) Let $\mathbf X_2$ be the set of $\mathbf x \in \mathbf X_1$ such that:
\mn
\begin{enumerate}
\item[$(*)$]  $K$ has trivial center.
\end{enumerate}
\mn
4) Let $\mathbf X_3$ be the set of $\mathbf x \in \mathbf X_1$ such
that\footnote{So $\mathbf x_3 \supseteq \mathbf x_2$.}:
\mn
\begin{enumerate}
\item[$(*)$]  if $f$ is a non-trivial automorphism of $K$ \then \,
for some conjugate $b$ of $a_* = a_*[\mathbf x] = a_0[\mathbf x]$ we have
$f(b) \notin \langle a_* \rangle_K$; equivalently, for some conjugate $b$
of $a_*,\langle b \rangle_K \ne \langle a \rangle_K$.
\end{enumerate}
\end{definition}

\begin{observation}
\label{d91}
If $m \in \{2,3,\ldots\}$ then for some $\mathbf x \in \mathbf X_3$ the
element $a_*[\mathbf x] \in K[\mathbf x]$ has order $m$.
\end{observation}

\begin{claim}
\label{d92}
If $\mathbf x \in \mathbf X_0$, \then \, there is one and only one
$\gs$, call it $\gs_{\cm} = \gs_{\cm}[\mathbf x]$ such that:
\mn
\begin{enumerate} 
\item[$(a)$]  $\gs \in \Omega[\mathbf K_{\lf}]$;
\sn
\item[$(b)$]  $k_{\gs} = \ell g(\bar a_1[\mathbf x])$ and $n_{\gs} =
\ell g(\bar a_2[\mathbf x])$;
\sn
\item[$(c)$]  $p_{\gs}(\bar x_{\gs}) = 
\tp_{\bs}(\bar a_1[\mathbf x],\emptyset,K[\mathbf x])$;
\sn
\item[$(d)$]  if $G_1 \subseteq G_3 \in \mathbf K_{\lf}$ and
$\tp_{\bs}(\bar a,\emptyset,G_1) = \tp_{\bs}(\bar a_1[\mathbf
x],\emptyset,K[\mathbf x])$ and $\bar c$ realizes $q_{\gs}(\bar a,G_1)$ 
in $G_3$ \then \, $\NF^3(\langle \bar a \rangle_{G_1},G_1,
\langle \bar a \char 94 \bar c\rangle_{G_3},G_3)$.
\end{enumerate}
\end{claim}

\begin{PROOF}{\ref{d92}}
As in \S2 using \S(4A).  Let $K_\ell = K_\ell[\mathbf x]$ for
$\ell=1,2$; and let $G_0 = K_1$ and $G_1 \in \mathbf K$ be
existentially closed, extend $K_1$ and be such that $K_2 \cap G_1 =
K_0$.  Let $L = \mathbf C_{G_0}(G_1)$, so as $G_0 = K_1$ has trivial
center (by \ref{d89}(1)(e)), clearly we have 
$L \cap G_0 = \{e_{G_0}\}$ and let $H_1 = c \ell(G_0
\cup L,G_1),H_0 = \{e_{K_1}\}$ and let $H_2 = G_2 := K_2$.  Now we
apply Claim \ref{d42}(2), so there is $G_3$ such that
$\NF^3(G_0,G_1,G_2,G_3)$ see Definition \ref{d39}.
By it, the type $\tp_{\bs}(\bar
a_2[\mathbf x],G_2,G_{\mathbf x})$ does not split over $G_0 = K_1$.  From
this it is easy to define $\gs$ and to prove it is as required. 
\end{PROOF}

\begin{dc}
\label{d93}
For $\mathbf x \in \mathbf X_1$ let $\gs = \gs_{\ab}[\mathbf x]$ be
such that:
\mn
\begin{enumerate}
\item[$(a)$]  $\gs \in \Omega[\mathbf K_{\lf}]$;
\sn
\item[$(b)$]  $k_{\gs} = 0$;
\sn
\item[$(c)$]  if $\bar c$ realizes $q_2(<>,G_1)$ in $G_2$ so $G_1
\subseteq G_2$ \then \, $\bar c$ realizes 
$\tp_{\bs}(\bar a[\mathbf x],\emptyset,K[\mathbf x])$ and commutes
with $G_1$, and $\langle \bar c \rangle_{G_2} \cap G_1 = \{e\}$.
\end{enumerate}
\end{dc}

\begin{PROOF}{\ref{d94}}
Easy.
\end{PROOF}

\begin{dc}
\label{d94}
For $\mathbf x \in \mathbf X_2$ we define $\gs = \gs_{\gmm}[\mathbf x]$
such that:
\mn
\begin{enumerate}
\item[$(a)$]  $\gs \in \Omega[\mathbf K_{\lf}]$; 
\sn
\item[$(b)$]  $k_{\gs} = 2 \ell g(\bar a[\mathbf x])$ and $n_{\gs}=1$;
\sn
\item[$(c)$]  if $G_1 \subseteq G_2 \in \mathbf K_{\lf}$ and 
$\tp_{\bs}(\bar a_\ell,\emptyset,G_1) = \tp_{\bs}(\bar a[\mathbf x],
\emptyset,K[\mathbf x])$ for $\ell=1,2$ and $\langle \bar a_1
\rangle_{G_1},\langle \bar a_2 \rangle_{G_1}$ commute in $G_1$
and\footnote{In fact this follows.} have
intersection $\{e_G\}$ then $p_{\gs}(\bar x_{\gs}) = 
\tp_{\bs}(\bar a_1 \char 94 \bar a_2,\emptyset,G_1)$;
\sn
\item[$(d)$]  moreover, in clause (c), if $c \in G_2$ realizes
$q_{\gs}(\bar a_1 \char 94 \bar a_2,G_1)$ in $G_2$ \then \,
conjugation by $c$ interchanges $\bar a_1,\bar a_2$ and is the
identity on $\mathbf C_{G_1}(\bar a_1 \char 94 \bar a_2)$.
\end{enumerate}
\end{dc}

\begin{PROOF}{\ref{d93}}
Let $G_2 \in \mathbf K_{\exlf}$ be an extension of $K[\mathbf x]$ in
which some $c$ realizes $q_{\gs_{\cg}}(K_{\mathbf x})$; let $\bar a_1 =
\bar a[\mathbf x],\bar a_2 = c^{-1} \bar a_1 c := \langle c^{-1} a_{1,\ell}
c:\ell < \ell g(\bar a_1)\rangle$ in $G_2$.

Note that, by inspection, $G_0 = \langle \bar a_1 \char 94 \bar
a_2\rangle_{G_1}$ is finite with trivial center and let $G_0 \subseteq
G_1 \in \mathbf K_{\lf}$.  Now use \ref{d36}
with $G_0,G_1,c \ell(\bar a_1 \char 94 \bar a_2 \char 94 \langle c
\rangle,G_2),\mathbf C_{G_1}(G_0),\mathbf C_{G_1}(G_0),c \ell(\bar a_1
\char 94 \bar a_2 \char 94 \langle c \rangle,G_2)$ here
standing for $G_0,G_1,G_2,G_1,H_1,L,H_2$ there. 
\end{PROOF}

\begin{definition}
\label{d96}
1) For $\gs \in \Omega[\mathbf K]$ and $G_1 \subseteq G_2$ let
$\cp_{\gs}(G_1,G_2) = \{c_0:\bar c \in {}^{n(\gs)}(G_2)$ realizes
$q_t(G_1)$ where $t \in \deef(G_1)$ satisfies $\gs_t = \gs\}$.

\noindent
2) For $\mathbf x \in \mathbf X_1$ and $G_1 \subseteq G_2$ let 
$\cp_{\mathbf x}(G_1,G_2) = \cp_{\gs_{\ab}[\mathbf x]}(G_1,G_2)$.

\noindent
3) For $G_1 \subseteq G_2 \in \mathbf K_{\text{lf}}$ and $\ell \in
\{1,2,3\}$ let $\cp_\ell(G_1,G_2) = \cup\{\cp_{\gs_{ab}[\mathbf x]}(G_1,G_2):
\mathbf x \in \mathbf X_\ell\}$; if $\ell=2$ we may omit it.
\end{definition}
\bigskip

\subsection {Larger Definable Types} \
\bigskip
 
\begin{definition}
\label{d81}
1) For $G \in \mathbf K,\gS \subseteq \Omega[\mathbf K]$ and set $I$ let
Def$_{I,< \kappa}(G,\gS)$ be the set of $t$ such that:
\mn
\begin{enumerate}
\item[$(a)$]  $t = \langle t_u:u \subseteq I$ finite$\rangle$;
\sn
\item[$(b)$]  $t_u \in \deef_{\gS}(G)$ with $\bar x_{t_u} = 
\langle x_i:i \in u \rangle$ and $\bar a_{t_u} = \bar a_t$ or
pedantically $\bar a_{t_u} = \bar a_t \rest w_u$ where $w_u \subseteq
\ell g(\bar a_t)$ is finite;
\sn
\item[$(c)$]  $\ell g(\bar a_t) := I$ has cardinality $< \kappa$ and
Rang$(\bar a_t) \subseteq G$;
\sn
\item[$(d)$]  if $G \subseteq H \subseteq L \in \mathbf K_{\lf}$
and $u \subseteq v \subseteq I$ are finite and $\bar b \in {}^v L$
realizes $q_{t_v}(H)$ then $\bar b \rest u$ realizes $q_{t_u}(H)$.
\end{enumerate}
\mn
2) We define $\Omega_{I,< \kappa}[\mathbf K,\gS]$ parallely and if $\gS
= \Omega[\mathbf K]$ \then \, we may omit it.

\noindent
3) If $t \in \text{ Def}_{I,< \kappa}(G,\gS)$ then $q_t(G) \in \mathbf
S^I_{\bs}(G)$ is defined by $\cup\{q_{t_u}(\langle x_i:i \in u\rangle):u
\subseteq I$ finite$\}$.

\noindent
4) Omitting $\kappa$ means $\aleph_0$.  We may replace ``$< \kappa^+$"
   by $\kappa$ and even a set $I_1$.  We may replace $I$ by ``$< \mu$"
   meaning ``some $\chi < \mu$".  Similarly for ``$\le \mu$".

\noindent
5) For $n < \omega$ and $\gs_0,\dotsc,\gs_{n-1} \in
\Omega_{< \mu,< \kappa}[\mathbf K]$ we define $\gs_0 \oplus \ldots
\oplus \gs_{n-1}$ and $\gs_0 \otimes \ldots \otimes \gs_{n-1}$ naturally.
\end{definition}

\begin{claim}
\label{d83}
1) If $G \in \mathbf K,\gS \subseteq \Omega[\mathbf K]$ and $t \in$ 
{\rm Def}$_I(G,\gS)$ \then \, for some pair $(\bar c,H)$ we have $G
   \subseteq H \in \mathbf K_{\lf},\bar c \in {}^I H,H =
   \langle G \cup \bar c \rangle_H$ and $\tp_{\bs}(\bar c,G,H) = q_t(G)$.

\noindent
2) If $\gS$ is closed \then \, above $G \le_{\gS} H$. 
\end{claim}

\begin{definition}
\label{d85}
Assume $\bar H = \langle H_i:i < \delta\rangle$ is
$\subseteq$-increasing in $\mathbf K$ and $H_\delta = \cup\{H_i:i <
 \delta\}$.  We say $\cA$ is a one step $(< \mu,< \kappa,
\delta,\gS)-\otimes$-construction (if $\delta = \omega$ we may omit
it) \when \,: as in \ref{a82} except that
\mn
\begin{enumerate}
\item[$(c)'$]   $t_{\alpha,i} \in \text{ Def}_{I_{\alpha,i,< \kappa}}(H_i,\gS)$
for some set $I_{\alpha,i}$ of cardinality $< \mu$.
\end{enumerate}
\end{definition}

\noindent
The case we shall actually use in \S5 is:
\begin{claim}
\label{d87}
Assume $K \subseteq L \in \mathbf K_{\lf},K$ is finite and $f$
embeds $K$ into $G_1 \in \mathbf K_{\lf}$ 
and $\langle c_i:i < \mu\rangle$ list the
members of $L$ and $\{c_\ell:\ell < n\}$ is the set of elements of
$K$.  \Then \, there is $t \in$ {\rm Def}$_{\le \mu}(G_1,\gS[\mathbf K])$ 
such that: 
if $\bar c^* = \langle c^*_i:i < \mu\rangle \in {}^\mu(G_2)$ realizes
$q_t(G_1)$ in $G_2$, so $G_1 \subseteq G_2$, \then \, $c_i \mapsto
c^*_i$ (for $i < \mu$) is an embedding of $L$ into $G_2$ extending $f$.
\end{claim}

\begin{PROOF}{\ref{d87}}
Straightforward by \S2.
\end{PROOF}

\begin{discussion}
Those definable types are still locally definable over finite sets.
\end{discussion}
\newpage

\section {Constructing complete existentially closed $G$}

\begin{theorem}
\label{p73}
Assume if $G \in \mathbf K_{\lf}$ and $|G| \le \mu = \mu^{\aleph_0}$.

\noindent
1) There is a complete $G' \in \mathbf K_{\lf}$ which extend
$G$ such that $|G'| = \mu^+$ and $G'$ is existentially closed.

\noindent
2) Moreover $G \le_{\Omega[\mathbf K_{\lf}]} G'$ and $G'$ is full.

\noindent
3) There is $G'$ such that $G \le_{\gS} G'$ and $G' \in \mathbf
K^{\exlf}_{\mu^+}$ is complete and $\gS$-full provided
that $\gS$ satisfies:
\mn
\begin{enumerate}
\item[$(*)$]
\begin{enumerate}
\item[$(\alpha)$]  $\gS \subseteq \Omega[\mathbf K_{\lf}]$
\sn
\item[$(\beta)$]  $\gS$ is dense and $\otimes$-closed 
(for $\mathbf K_{\lf}$)
\sn
\item[$(\gamma)$]   some schemes introduced earlier
  belongs to $\gS$, specifically:
\begin{itemize}
\item  $\gs_{\ab(2)}$ from Definition \ref{c67}, used in the
  paragraph before $\boxplus_3$
\sn
\item  $\gs_{\cm}$ from Definition \ref{d92}, 
used in $(*)_{4.3}$
\sn
\item  $\gs_{\cg}$, from Definition \ref{c50}(1), \ref{c52}(2)
  used after $\boxplus_7$ Stage E
\sn
\item   $\gs_{\gl}$ from Definition \ref{c50}(2),\ref{c52}(3) 
\sn
\item   $\gs_{\gmm}$ from Definition \ref{d93}, see $(*)_{5.1}(f)$.
\end{itemize}
\end{enumerate}
\end{enumerate}
\end{theorem}

\begin{PROOF}{\ref{p73}}
\underline{Proof of \ref{p73}}

We let $\gS = \Omega[\mathbf K_{\lf}]$ for parts (1),(2) and fix $\gS$
for part (3) as there.
\medskip

\noindent
\underline{Stage A}:  Without loss of generality the universe of $G$
is an ordinal $\le \mu$ and let $\lambda = \mu^+$.

Let $S \subseteq S^\lambda_{\aleph_0} := \{\delta <
\lambda:\cf(\delta) = \aleph_0\}$ be a stationary subset of
$\lambda$ such that also $S^\lambda_{\aleph_0} \backslash S$ is
stationary in $\lambda$ and $\alpha \in S \Rightarrow$ ($\mu$ divides
$\alpha$).  Let $\langle S_\zeta:\zeta < \lambda\rangle$ be a 
partition of $S$ to stationary sets.  Let $S_* \subseteq \lambda
\backslash S$ be stationary and a set of limit ordinals.  

Let $C_\delta$ be an unbounded subset of $\delta$ of
order type $\omega$ for $\delta \in S$ such that $\bar C_\zeta =
\langle C_\delta:\delta \in S_\zeta\rangle$ guess clubs for each
$\zeta < \lambda$, this means that for every club $E$ of $\lambda$ the
set $\{\delta \in S_\zeta:C_\delta \subseteq E\}$ is a stationary
subset of $\lambda$; such $\langle C_\delta:\delta \in S_\zeta\rangle$
exists by \cite[Ch.III]{Sh:g} = \cite{Sh:365}.

Let $\alpha_\delta(n)$ be the $n$-th member of $C_\delta$.  

Let $\bar \tau$ be such that:
\mn
\begin{enumerate}
\item[$\bullet$]  $\bar\tau = \langle \tau_\zeta:\zeta < \lambda\rangle$
\sn
\item[$\bullet$]  $\tau_\zeta \subseteq \cH(\aleph_0)$ is a countable
vocabulary
\sn
\item[$\bullet$]  if $\tau \subseteq \cH(\aleph_0)$ is a countable
vocabulary \then \, $\{\zeta:\tau_\zeta = \tau\}$ has cardinality
$\lambda$.
\end{enumerate}
\mn
By \cite[3.26(3)=L6.11A,pg.31]{Sh:309} there is 
$\mathbf b_\zeta$, a BB, black box for 
$(S_\zeta,\bar C_\zeta)$ say $\mathbf b_\zeta = \langle
N^\delta_i:i \in \cT_\delta,\delta \in S_\zeta\rangle$,
that is:
\mn
\begin{enumerate}
\item[$\boxplus_{0,\zeta}$]
\begin{enumerate}
\item[(a)]  $N^\delta_i$ is a model of 
cardinality $\aleph_0$ with universe $\subseteq \delta = 
\sup(N^\delta_i)$ and vocabulary $\tau_\zeta \subseteq \cH(\aleph_0)$
\sn
\item[(b)]   if $N$ is a $\tau_\zeta$-model with
universe $\lambda$ \then \, for stationarily many 
$\delta \in E_N \cap S_\zeta$ for some $i \in \cT_\delta$ 
we have $C_\delta \subseteq E_N \backslash S$ where 
$E_N := \{\alpha:N \rest \alpha \prec N\}$
and $N^\delta_i \prec N$; moreover
\sn  
\item[(b)$^+$]   if $\tau = \tau_\zeta,\bar N = \langle
N_\eta:\eta \in \cT\rangle,\cT$ a non-empty subtree of ${}^{\omega >}\lambda$ 
such that $\tau(N_\eta) = \tau_\zeta,\eta \triangleleft \nu
\Rightarrow N_\eta \prec N_\nu$ and $|N_\eta| \in
[\lambda]^{\aleph_0}$ and $E$ a club 
of $\lambda,\eta \in \cT \Rightarrow
(\exists^\lambda \alpha)(\eta \char 94 \langle \alpha
\rangle \in \cT)$ and $\eta \triangleleft \nu \in \cT \Rightarrow
\sup(N_\eta) < \sup(N_\nu)$  \then \, for some
$\delta \in S_\zeta \cap E$ we have $C_\delta \subseteq E,
i \in \cT_\delta$ and $\eta \in \lim_\omega(\cT)$ we have 
$N^\delta_i= \cup\{N_{\eta \rest n}:n < \omega\}$;
\sn
\item[(c)]   if $i \ne j \in \cT_\delta$ then
$N^\delta_i \cap N^\delta_j$ is bounded in $\delta$ (used just after
$(*)_{5.5}$), moreover:
\sn
\item[(c)$^+$]  if $i \ne j \in \cT_\delta$ \then \,
the set $\{\beta < \delta:\beta$ a limit ordinal such that 
$\sup(N^\delta_i \cap \beta) = \beta =
\sup(N^\delta_j \cap \beta)\}$ is bounded in $\delta$;
\sn
\item[(d)]  $N^\delta_i \cap
(\alpha_\delta(n),\alpha_\delta(n+1)) \ne \emptyset$ and
$N^\delta_i \rest \alpha_\delta(n) \prec N^\delta_i$ for $n <
\omega,\delta \in S,i \in \cT_\delta$;
\sn
\item[(e)]   for notational simplicity we assume 
$\cT_\delta \subseteq \mu$.
\end{enumerate}
\end{enumerate}
\medskip

\noindent
\underline{Stage B}:  By induction on $\gamma < \lambda$ we shall
choose the following: 
\mn
\begin{enumerate}
\item[$\boxplus_1$]
\begin{enumerate}
\item[(a)]  $G_\gamma \in \mathbf K_{\lf}$ of cardinality $\mu$
and the universe of $G_\gamma$ is an ordinal $< \lambda$;
\sn
\item[(b)]  $G_0 = G$;
\sn
\item[(c)]  $\langle G_\beta:\beta \le
\gamma\rangle$ is increasing continuous;
\sn
\item[(d)]   if $\beta \in \gamma \backslash S$ then
$G_\beta \le_{\gS} G_\gamma$ 
\sn
\item[(e)]   if $\gamma = \beta +1,\beta \notin S$, then:
\begin{enumerate}
\item[$(\alpha)$]  $G_\gamma$ is generated by $\{\bar
c_{\beta,i}:i \in \cT_\beta\} \cup G_\beta$, where $\cT_\beta$ is a
set of cardinality $\le \mu$ (to be chosen),
\sn
\item[$(\beta)$]  $t_{\beta,i} \in \text{ Def}_{\le \mu}
(G_\beta,\gS)$, non-trivial (see Definition \ref{d81}(5)) for $i \in 
\cT_\beta$,
\sn
\item[$(\gamma)$]  $\tp_{\bs}(\bar c_{\beta,i},
G_\beta,G_\gamma) = q_{t_{\beta,i}}(G_\beta)$ for $i \in \cT_\beta$,
\sn
\item[$(\delta)$]   if $n < \omega$ and
$i(0),\dotsc,i(n-1) \in \cT_\beta$ are pairwise distinct, then
$\tp_{\bs}(\bar c_{\beta,i(0)} \char 94 \ldots \char 94 \bar
c_{\beta,i(n-1)},G_\beta,G_\gamma) = q_t(G_\beta)$,
where $t = t_{\beta,i(0)} \otimes \ldots \otimes
t_{\beta,i(n-1)}$,
\sn
\item[$(\varepsilon)$]   if $t = (\gs,\bar a) \in 
\deef_{\gS}(G_\beta)$ is non-trivial \then \, for some $i \in
\cT_\beta$ we have $t_{\beta,i} = t$;
\end{enumerate}
\sn
\item[(f)]   if $\gamma = \delta +1,\delta \in S$ \then \,:
\sn
\begin{enumerate}
\item[$(\alpha)$]  $G_\gamma$ is generated by 
$\{\bar c_{\delta,i}:i \in \cT_\delta\} \cup G_\delta$,
\sn
\item[$(\beta)$]  $\cA_\gamma = (G_{\delta +1},G_\delta,
\langle \bar c_\delta,t_{\delta,i,n}:i \in \cT_\delta\rangle)$
is a one step $(< \aleph_0,< \aleph_0,\gS)-\otimes$-construction over 
$\langle G_{\alpha_\delta(n)}:n < \omega\rangle$, 
see \ref{a82}; used in $(*)_{5.2}$'s
proof\footnote{Actually can use a one step $(\le \mu,<
\aleph_0,\gS)-\otimes$-construction.},
\end{enumerate}
\sn
\item[(g)]  $t_{\beta,i} = (\gs_{\beta,i},\bar a_{\beta,i})$ for
  $\beta \in \gamma \backslash S$.
\end{enumerate}
\end{enumerate}
\mn
First we shall show:
\mn
\begin{enumerate}
\item[$\boxplus_2$]   we can carry the induction.
\end{enumerate}
\mn
Why?  For $\gamma = 0$ we have nothing to do by clause $(b)$.

For $\gamma$ limit we let $G_\gamma = \cup\{G_\beta:\beta
< \gamma\}$.

For $\gamma = \beta +1,\beta \notin S$ we have some freedom, as we
have $t_{\beta,i} \in \text{ Def}_{\le \mu}(G_\beta,\gS)$ not just
$\deef(G_\beta,\gS)$.  So let $\cT_\beta = \mu,\{t_{\beta,i}:i \in
\cT_\delta\} \subseteq 
\text{ Def}_{\le \mu}(G_\beta,\gS)$ be of cardinality $\mu$ and
including $\deef(G_\beta,\gS)$ and so $\langle t_{\beta,i} = (\gs_{\beta,i},
\bar a_{\beta,i}):i < \mu\rangle$, possibly with
repetitions.  Clearly $\boxplus_1(e)(\varepsilon)$ holds.

Now as in Claim \ref{a74} we can find 
$G_\gamma,\langle \bar c_{\beta,i}:i < \mu\rangle$ such that:
\mn
\begin{enumerate}
\item[$\bullet$]  $G_\beta \le_{\gS} G_\gamma$;
\sn
\item[$\bullet$]  $G_\gamma = \langle \{\bar c_{\beta,i}:i < \mu\} 
\cup G_\beta \rangle_{G_\gamma}$;
\sn
\item[$\bullet$]  if $n < \omega$ and $i_k < \mu$
for $k < n$ and $\langle i_k:\ell < n\rangle$ is with no repetitions
\then \,
\[
\bar c_{\beta,i_0} \char 94 \ldots \char 94 \bar
c_{\beta,i_{n+1}} \text{ realizes } q_t(G_\beta) 
\text{ where } t = t_{\beta,i_0} \otimes \ldots \otimes t_{\beta,i_{n-1}}.
\]
\end{enumerate}
\mn
If $\gamma = \delta +1,\delta \in S$ we can let $\gs_{\delta,i} = 
\gs_{\ab(2)}$,
clearly we satisfy clause (f); but we may act differently.  Clearly,
as in the previous case, there is some freedom left: what 
we do for $\gamma = \delta +1,\delta
\in S$ and this will depend on the $\langle N^\delta_i:i \in
\cT_\delta \rangle$ from $\boxplus_0$.  
During the rest of the proof we shall use (some
of the freedom left) to guarantee that $G_*$ (see below) is as required.

Of course, we let:
\mn
\begin{enumerate}
\item[$\boxplus_3$]  $G_* = G_\lambda = \cup\{G_\alpha:\alpha < \lambda\}$.
\end{enumerate}
\mn
We now point out some useful properties of the construction:
\mn
\begin{enumerate}
\item[$(*)_{3.1}$]   there is a model $N_*$ expanding $G_*$, so with
universe $\lambda$, and a countable vocabulary such that for any $N
\subseteq N_*$ we have:
\sn
\begin{enumerate}
\item[$(a)$]  $G_* \rest N$ is a subgroup of $G_*$;
\sn
\item[$(b)$]  $\beta \in N$ \Iff \, $N \cap G_{\beta +1} \backslash
G_\beta \ne \emptyset$ \Iff \, $\beta +1 \in N$;
\sn
\item[$(c)$]  if $\gamma = \beta +1,\gamma \in N$ then $N \cap
G_\gamma = \langle \cup\{\bar c_{\beta,i}:i \in N \cap \cT_\beta\} 
\cup (N \cap G_\beta)\rangle_{G_\gamma}$;
\sn
\item[$(d)$]  if $i \in N \cap \cT_\beta$ and $\beta \in N$, \then \, $|\ell
g(\bar c_{\beta,i})| \le \omega \Rightarrow \bar c_{\beta,i} 
\subseteq N \cap G_{\beta +1}$ and $|\ell g(\bar a_{t_{\beta,i}})| \le
\omega \Rightarrow \bar a_{t_{\beta,i}} \subseteq N \cap G_\beta$;
\sn
\item[$(e)$]  $\tau(N_*) \subseteq \cH(\aleph_0)$, but $\cH(\aleph_0)
\backslash \tau(G_*)$ is infinite;
\sn
\item[$(f)$]  if $\delta \in N \cap S$ then $C_\delta \subseteq N$.
\end{enumerate}
\end{enumerate}
\mn
Now note
\mn
\begin{enumerate}
\item[$(*)_{3.2}$]   if $\alpha < \lambda$ is a limit ordinal, \then \,
$G_\alpha \in \mathbf K_{\exlf}$.
\end{enumerate}
\mn
[Why?  Recall clause $(e)(\varepsilon)$ of $\boxplus_1$ noting that
$S$ is a set of limit ordinals, hence $\alpha = \sup(\alpha \backslash S)$.]

We now assume:
\mn
\begin{enumerate}
\item[$\boxplus_4$]  $\mathbf h$ is an automorphism of $G_*$.
\end{enumerate}
\mn
We shall eventually prove that (if we suitably use the freedom left in
$\boxplus_1$, then) $\mathbf h$ is an inner automorphism,
i.e. $b \in G_* \Rightarrow \mathbf h(b) = a^{-1} b a$ for some 
$a \in G_*$, this clearly suffices noting that $G_*$ has trivial
center as $\gs_{\cg} \in \gS$.

We shall often use
\mn
\begin{enumerate}
\item[$(*)_{4.1}$]   for limit $\beta \in \lambda \backslash S$ let
$L^*_\beta = \cp(G_\beta,G_{\beta + \omega})$ (see Definition \ref{d96}(3)),
i.e. $c \in L^*_\beta$ \Iff \, for some finite $K \subseteq
\mathbf C_{G_{\beta + \omega}}(G_\beta)$ with trivial center we have $c \in
K$ and $K \cap G_\beta = \{e_{G_\beta}\}$.
\end{enumerate}
\mn
Note that
\mn
\begin{enumerate}
\item[$(*)_{4.2}$]  the last demand in $(*)_{4.1}$, ``$K \cap G_\beta
= \{e_{G_*}\}"$, is redundant.
\end{enumerate}
\mn
[Why?  Recall $\beta$ is a limit ordinal hence by $(*)_{3.2}$, $G_\beta$ has
trivial center.]

Note:
\mn
\begin{enumerate}
\item[$(*)_{4.3}$]  if $a \in L^*_\beta$ and $K$ witnesses it, \then \, $K
\subseteq L^*_\beta,K \cap G_\beta = \{e\}$, and 
moreover there is $L \in \mathbf K_{\exlf}$
included in $L^*_\beta$ and including $K$.
\end{enumerate}
\mn
[Why?  We can choose $\bar K = \langle K_n:n < \omega \rangle$ such that $K_0
= K,K_n$ is a finite group with trivial center, $K_n \subseteq K_{n+1}$
and $\bigcup\limits_{n} K_n \in \mathbf K_{\exlf}$.  We now choose by
induction on $n$ an embedding $f_n$ of $K_n$ into 
$G_{\beta + \omega}$ such that $f_0 =
\id_K,f_n \subseteq f_{n+1}$ and $\Rang(f_n) \subseteq L^*_\beta$; the
induction step is possible by \ref{d92}.  
Now $\bigcup\limits_{n} f_n(K_n)$ is as required.]

We shall use:
\mn
\begin{enumerate}
\item[$(*)_{4.4}$]  let $E_{\mathbf h} = \{\delta:\delta$ is a limit
ordinal and $\mathbf h$ maps $G_\delta$ onto $G_\delta$ and $(N^* \rest
\delta,\mathbf h \rest \delta) \prec (N^*,\mathbf h)\}$.
\end{enumerate}
\mn
Now
\mn
\begin{enumerate}
\item[$(*)_{4.5}$]   $E_{\mathbf h}$ is a club of $\lambda$.
\end{enumerate}
\mn
[Why?  Just look at $(*)_{4.4}$.]
\medskip

\noindent
\underline{Stage C}:  We shall prove
\mn
\begin{enumerate}
\item[$\boxplus_5$]   for some $\alpha(*) < \lambda$, for every 
$\beta \in S_* \cap E_{\mathbf h} \backslash \alpha(*)$ and 
$c \in L^*_\beta$ we have 
$\mathbf h(c) \in c \ell(G_{\alpha(*)} \cup \{c\},G_*)$.
\end{enumerate}
\mn
Why?  If not, for every $\alpha < \lambda$ there are $\beta_\alpha \in
S_* \cap E_{\mathbf h} \backslash \alpha,m(\alpha) = 
m_\alpha \in \{2,3,\ldots\}$ and $c_\alpha \in L^*_{\beta_\alpha}$ of order
$m_\alpha$ such that $\mathbf h(c_\alpha) \notin 
c \ell(G_\alpha \cup \{c_\alpha\},G_*)$.  Now let
$\bar c_\alpha$ witness that $c_\alpha \in L^*_{\beta_\alpha}$ 
with $c_{\alpha,0} = c_\alpha$, i.e. $\bar c_\alpha$
list the members of a finite subgroup of $G_{\beta_\alpha + \omega}$
 commuting with $G_{\beta_\alpha}$ with trivial center and so included in
$L^*_{\beta_\alpha}$.
\mn
\begin{enumerate}
\item[$(*)_{5.0}$]  \wilog \, $\mathbf h(c_\alpha) \in c \ell(G_\alpha
  \cup \bar c_\alpha)$.
\end{enumerate}
\mn
[Why?  Let $K_0$ be the subgroup of $G_{\beta_\alpha + \omega}$ with
universe $\bar c_\alpha$; we can find $K_1,K_2,K_3$ such that $K_3$
is a finite group and for $\ell = 0,1,2$ and $K_1 \cap K_2 = \langle
c_0\rangle_{K_0}$ \wilog \, $K_3 \subseteq G_{\beta_\alpha + \omega}$,
so we can replace $K_0$ by $K_1$ or by $K_2$.]

Let $\mathbf x_\alpha \in X_2$ be such that
$\bar c_\alpha$ realizes $q_{\gs_{\ab}[\mathbf x_\alpha]}
(\langle \rangle,G_{\beta_\alpha})$, see \ref{d89}(2) + \ref{d94}.
But if $\alpha_1 < \alpha_2$ then $(\beta_{\alpha_2},
c_{\alpha_2},m_{\alpha_2})$ can serve 
as $(\beta_{\alpha_1},c_{\alpha_1},m_{\alpha_1})$, hence, 
\wilog \,, $\mathbf x_\alpha = \mathbf x,m_\alpha = m_*$ for every $\alpha$.
\mn
\begin{enumerate}
\item[$(*)_{5.1}$]
\begin{enumerate}
\item[(a)]   Let $\bar b_{\alpha,1} = \bar c_\alpha$; let $k_{\alpha,1}
< \omega$ be such that $\bar b_{\alpha,1} \subseteq G_{\beta_\alpha +
k_{\alpha,1} +1},\bar b_{\alpha,1} \nsubseteq G_{\beta_\alpha + k_{\alpha,1}}$;
\sn
\item[(b)]  let $k_{\alpha,*} \in 
(k_{\alpha,1} +1,\omega)$ be such that: 
tp$_{\bs}(\mathbf h(\bar b_{\alpha,1}),G_{\beta_\alpha +
\omega},G_*) = q_{\gs_\alpha}(\bar a^\bullet_\alpha,
G_{\beta_\alpha + \omega})$ for some $\gs_\alpha \in \gS$ with
$\bar a^\bullet_\alpha \subseteq G_{\beta_\alpha + k_{\alpha,*}}$;
\sn
\item[(c)]  let $\bar b_{\alpha,2} \subseteq 
G_{\beta_\alpha + \omega}$ realize 
$q_{\gs_{\ab}[\mathbf x]}(\langle \rangle,G_{\beta_\alpha + k_{\alpha,*}})$;
\sn
\item[(d)]  let $k_{\alpha,2} < \omega$ 
be such that $\bar b_{\alpha,2} \subseteq
G_{\beta_\alpha + k_{\alpha,2}+1},\bar b_{\alpha,2} 
\nsubseteq G_{\beta_\alpha + k_{\alpha,2}}$, 
so actually \wilog \, $k_{\alpha,2} = k_{\alpha,*} +1$;
\sn
\item[(e)]  note that $\bar b_{\alpha,1} \char 94 \bar
b_{\alpha,2}$ realizes $p_{\gs_{\gem}}(\bar x)$, see \ref{d94};
\sn
\item[(f)]  let $k_{\alpha,3} < \omega$ be 
$> k_{\alpha,1},k_{\alpha,2}$ and let $b_{\alpha,3} \in 
G_{\beta_\alpha + k_{\alpha,3} +1}$ realizes 
\newline
$q_{\gs_{\gem}[\mathbf x]}(\bar b_{\alpha,1} \char 94 
\bar b_{\alpha,2},G_{\beta_\alpha +
k_{\alpha,3}},G_{\beta_\alpha})$, (see Definition \ref{d94});
so it commutes with $\mathbf C_{G_{\beta_\alpha} + 
k_{\alpha,2}+1}(\bar b_{\alpha,1} \char 94
\bar b_{\alpha,2})$, hence with $G_{\beta_\alpha}$ 
and conjugating by it 
interchange $\bar b_{\alpha,1},\bar b_{\alpha,2}$;
\sn
\item[(g)]   \wilog \, $\gs_\alpha = \gs_*$ and
$(\ell g(\bar b_{\alpha,1}),k_{\alpha,2},\ell g(\bar
b_{\alpha,2}))$ does not depend on $\alpha$.
\end{enumerate}
\end{enumerate}
\mn
Our intention (in this stage) is to find $\alpha_n < \lambda$
increasing with $n$ satisfying 
$\beta_{\alpha_n} < \alpha_{n+1}$ and element $d$
such that, on the one hand, conjugating with $d$ maps $c_{\alpha_n} =
b_{\alpha_n,1,0}$ to $b_{\alpha_n,2,0}$ for each $n$, and on the other
hand, $\tp_{\bs}(d,G_{\beta_n + \omega},G_\lambda)$ does not split over
$G_{\beta_n} + b_{\alpha_n,3}$, a contradiction.

Let $N$ be such that:
\mn
\begin{enumerate}
\item[$(*)_{5.2}$]   
\begin{enumerate}
\item[(a)]  $N$ is a model with universe $\lambda$;
\sn
\item[(b)]  $N$ is with countable vocabulary;
\sn
\item[(c)]  $N$ expands $N_*$ from $(*)_{3.1}$;
\sn
\item[(d)] 
\begin{itemize}
\item  $F^N_0 = \mathbf h$, so $F_0$ is a unary function symbol,
\sn
\item  $F^N_{1,\iota,\ell}(\alpha) =
b_{\alpha,\iota,\ell}$ for $\iota = 1,2$ and $\ell < \ell g(\bar
b_{\alpha,\iota})$, (if $\ell=0$ we may omit it),
\sn
\item  $F^N_{1,3}(\alpha) = b_{\alpha,3}$,
\sn
\item  $F^N_2(\alpha) = \beta_\alpha$,
\sn
\item  $F_{2,\iota}(\alpha) = \beta_\alpha +
k_{\alpha,\iota}$ for $\iota = 1,2,3$,
\sn
\item  $F^N_3(\alpha) = \beta_\alpha + \omega$,
\end{itemize}
\sn
\item[(e)]  $F^N_{4,n}$ is an $(n+1)$-place function such
that: if $\alpha_0 < \ldots < \alpha_n,c_{\alpha_\ell} \in 
G_{\alpha_{\ell +1}}$, each $\alpha_\ell$ is a limit ordinal then
$F^N_{4,n}(\alpha_0,\ldots,\alpha_n)$ is 
the product of $a_0 a_1 \ldots a_n$ where $a_k = F_{1,1,0}(\alpha_k)$;
\sn
\item[(f)]  $P^N = \{(\alpha,c):\alpha < \lambda$ and $c \in G_\alpha\}$.
\end{enumerate}
\end{enumerate}
\mn
Without loss of generality $\tau_N \subseteq \cH(\aleph_0)$, choose
$\zeta(1) < \lambda$ such that $\tau_{\zeta(1)} = \tau_N$ and for
each $\delta \in S_{\zeta(1)}$ we use the amount of freedom we are left
with (see before $\boxplus_3$), choosing $G_{\delta +1}$ such that:
\mn
\begin{enumerate}
\item[$(*)_{5.3}$]   if $\delta \in S_{\zeta(1)},i \in
\cT_\delta$, letting $\alpha_{\delta,i,n} := \min(N^\delta_i
\backslash \alpha_\delta(n))$ \then \, $(a) \Rightarrow (b)$ where:
\sn
\begin{enumerate}
\item[$(a)$]
\begin{itemize}
\item  $\beta_{\delta,i,n} :=
F^{N^\delta_i}_2(\alpha_{\delta,i,n})$ is $\ge \alpha_{\delta,i,n}$ but $<
\alpha_\delta(n+1)$,
\sn
\item  $F^{N^\delta_i}_3(\alpha_{\delta,i,n}) =
\beta_{\delta,i,n} + \omega$,
\sn
\item  $b_{\delta,i,n,\iota} = F^N_{1,\iota,\ell}(\alpha_{\delta,i,n})$
 for $\iota = 1,2$ and $\ell=0$,
\sn
\item  $k_{\delta,i,n,\iota} =
F^{N^\delta_i}_{2,\iota}(\alpha_{\delta,i,n}) - \alpha_{\delta,i,n}$ 
for $\iota = 1,2,3$,
\sn
\item  $b_{\delta,i,n,3} =
F^{N^\delta_i}_3(\beta_{\delta,i,n})$,
\sn
\item  $b_{\delta,i,n,\iota} \in
G_{\beta_{\delta,i,n}+k_{\delta,i,n,\iota}+1}$
commute with $G_{\beta_{\delta,i,n}}$ and conjugating by $b_{\delta,i,n,3}$
interchange $b_{\delta,i,n,1,\ell},b_{\delta,i,n,2,\ell}$,
\sn
\item  $\delta$ is the set of elements of
$G_\delta$, similarly $\alpha_{\delta,i,n}$ (as they $\in E_{\mathbf h}$),
\sn
\item   for every $\beta < \delta$ we have
$(G_{\beta +1} \backslash G_\beta) \cap N^\delta_i \ne \emptyset
\Leftrightarrow \beta \in N_{\delta,i}$,
\sn
\item  if $\beta \in N^\delta_i \backslash S$
and $\bar c \in {}^{\omega >}(N^\delta_i)$, so $\bar c \in 
{}^{\omega>}(G_\delta)$, then $\tp_{\bs}(\bar c,G_\beta,G_\delta) \in
q_t(G_\beta)$ for some $t \in \deef(G_\beta)$ satisfying 
$\bar a_t \in {}^{\omega >}(N^\delta_i \cap G_\beta)$,
\end{itemize}
\sn
\item[(b)]  $\bar c_{\delta,i} = \langle c_{\delta,i}\rangle$ and
$\tp_{\bs}(c_{\delta,i},G_\delta,G_{\delta +1})$ is as in claim
\ref{c64} with $G_{\alpha_\delta(n)}(n < \omega),G_\delta,
b_{\beta_{\delta,i,n,3}} \in N^\delta_i (n < \omega)$ here standing for 
$G_n(n < \omega),G_\omega,a^t_n(n < \omega)$ with $I = \{t\}$ there;
\end{enumerate}
\sn
\item[$(*)_{5.4}$]   let $\cT'_\delta = \{i \in \cT_\delta$: clause 
(a) of $(*)_{5.3}$ holds$\}$;
\sn
\item[$(*)_{5.5}$]    let $\beta_{\delta,i,n} =
  \alpha_{\delta,i,n} + \omega$ and $b_{\delta,i,n} = b_{\delta,i,n,\iota}
\in G_{\beta_{\delta,i,n}+1}$ for $\iota = 1,2,3$ realizes $q_{\gs_{\ab(2)}}
(\langle \rangle,G_{\beta_{\delta,i,n}})$ \when \, 
the assumption of clause (a) fails.
\end{enumerate}
\mn
Why can we fulfill $(*)_{5.3}$?  Let $\langle i_\ell:\ell 
< \ell(*)\rangle$ be a finite sequence of members of 
$\cT_\delta$.  For $\ell< \ell(*)$ and $n < \omega$ let 
$d_{\ell,n} = b_{\delta,i_\ell,n,3}$.

Now
\mn
\begin{enumerate}
\item[$(*)_{5.6}$]  $\langle d_{\ell,n}:n < \omega\rangle$ pairwise
commute if $i(\ell) \in \cT'_\delta$ for each $\ell < \ell(*)$.
\end{enumerate}
\mn
[Why?  As $b_{\delta,i(\ell),n,3} \in
\mathbf C(G_{\beta_{\delta,i(\ell),n}},
G_{\beta_{\delta,i(\ell),n}+\omega})$
for $n < \omega$ and
$\beta_{\delta,i(\ell),n} + \omega < \alpha_\delta(n+1) \le
\alpha_{\delta,i(\ell),n+1} \le \beta_{\delta,i(\ell),n+1}$, recalling
$N^i_\delta \rest \alpha_\delta(n+1) \prec N^i_\delta$ and $N^i_\delta
\cap (\alpha_\delta(n),\alpha_\delta(n+1)) \ne \emptyset$.]
\mn
\begin{enumerate}
\item[$(*)_{5.7}$]   $\langle d_{\ell,n}:n < \omega \rangle$ pairwise
  commute when $i(\ell) \notin \cT'_\delta$.
\end{enumerate}
\mn
[Why?  Even easier.]
\mn
\begin{enumerate}
\item[$(*)_{5.8}$]  if $\ell(1) \ne \ell(2)$ \then \, for every $n(1) < n(2)$
the elements $b_{\delta,i(\ell(1)),n(1),3},
b_{\delta,i(\ell(2)),n(2),3}$ commute.
\end{enumerate}
\mn
[Why?  Recall that $b_{\delta,i(\ell(1)),n(1),3} \in
G_{\alpha_\delta(n(2))} \subseteq G_{\beta_{\delta,i(\ell(2)),n(2)}}$; note
that $b_{\delta,i,n,\iota} \in G_{\beta_{\delta,i,n,\iota} + \omega}$
commute with $G_{\beta_{\delta,i,n}}$ rather than with
$G_{\alpha_{\delta,i,n}}$ but not used.] 
\mn
\begin{enumerate}
\item[$(*)_{5.9}$]  if $\ell(1),\ell(2) < \ell(*)$, \then \, for $n$ large
enough, for every $n(1),n(2) \in (n,\omega)$ the elements
$d_{\ell(1),n(1)},d_{\ell(2),n(2)}$ of $G_\delta$ commute.
\end{enumerate}
\mn
[Why?  Similarly, as $N^\delta_{i_{\ell(1)}} \cap
N^\delta_{i_{\ell(2)}}$ is bounded in $\delta$, but not used.]
\mn
\begin{enumerate}
\item[$(*)_{5.10}$]  The conditions in \ref{c64} hold hence we can
  fulfill $(*)_{5.3},(*)_{5.4}$, i.e. we can carry the induction
  in $\boxplus_1$.
\end{enumerate}
\mn
[Why?  Think.]

Next let
\mn
\begin{enumerate}
\item[$(*)_{5.11}$]  $E = \{\delta < \lambda:
\delta$ a limit ordinal is the universe of $G_\delta$ and 
$N \rest \delta \prec N$, hence $\mathbf h$ maps $G_\delta$ onto itself$\}$.
\end{enumerate}

Clearly $E$ is a club of $\lambda$, hence by $\boxplus_{0,\zeta(1)}$ 
from stage A, there is a pair $(\delta,i_*) = (\delta,i(*))$ such that
\mn
\begin{enumerate}
\item[$(*)_{5.12}$]  $\delta \in E \cap S_{\zeta(1)}$ and $i_* \in \cT_\delta$
and $N^\delta_{i_*} \prec N$.
\end{enumerate}
\mn
Let $d = \mathbf h(c_{\delta,i_*}) \in G_*$, so:
\mn
\begin{enumerate}
\item[$(*)_{5.13}$]
\begin{enumerate}
\item[(a)]   the pair $(\delta,i_*)$ satisfies the
demands in $(*)_{5.3}(a)$; 
\sn
\item[(b)]  for some finite set $u_* \subseteq
 \cT_\delta$ and $\bar b_* \in {}^{\omega >}(G_\delta)$, 
the type $\tp_{\bs}(d,G_{\delta +1},G_*)$ does not split over 
$\{c_{\delta,i}:i \in u_*\} \cup \bar b_*$;
\sn
\item[(c)]  \wilog \,  $i_* \in u_*$.
\end{enumerate}
\end{enumerate}
\mn
[Why?  For clause (a), as $\delta \in E$ and $N^\delta_{i(*)} \prec N$,
recalling the choice of $N$ (including $\mathbf h = F^N_0$).  For clause
(b), apply properties of the construction in $\boxplus_1$,
i.e. $G_{\delta +1} \le_{\gS} G_*$.]
\mn
\begin{enumerate}
\item[$(*)_{5.14}$]   conjugating by $d$ in $G_*$ interchange
$b_{\delta,i(*),n,1}$ with $b_{\delta,i(*),n,2}$ for $n < \omega$.
\end{enumerate}
\mn
[Why?  Should be clear as for $m \in \omega \backslash
\{n\}$ and $\iota(1),\iota(2) \in \{1,2,3\}$, the element
$b_{\delta,i(*),m,\iota(1)}$
commutes with $b_{\delta,i(*),m,\iota(2)}$.]

Recalling $\boxplus_0(c)$ there is 
$n(*) < \omega$ large enough such that:
\mn
\begin{enumerate}
\item[$(*)_{5.15}$]  $\bar b_* \subseteq G_{\beta_{\delta,i(*),n(*)}}$
and $j_1 \ne j_2 \in u_* \Rightarrow N_{j_1}
\cap N_{j_2} \subseteq G_{\alpha_\delta(n(*))}$ and $j_1,j_2$ are like
$i,j$ as in $\boxplus_{0,\zeta(1)}(c)^+$.
\end{enumerate}
\mn
Clearly for some $\beta(*) < \lambda$ we have 
$\mathbf h(b_{\delta,i(*),n(*),1}) \in G_{\beta(*)+1} \backslash
G_{\beta(*)}$.  As $\alpha_{\delta,i(*),n(*)} = 
\min(N^\delta_{i(*)} \backslash \alpha_\delta(n(*)) \in 
N^\delta_{i(*)}$, clause (d)
of $\boxplus_0$ and $N^\delta_{i_*} \prec N$, clearly:
\mn
\begin{enumerate}
\item[$(*)_{5.16}$] 
\begin{enumerate}
\item[(a)]   $\mathbf h$ maps $G_{\alpha_{\delta,i(*),n(*)}} \cap
N^\delta_{i_*}$ onto itself and so $\beta(*) \in N^\delta_{i_*} \backslash
\alpha_{\delta,i(*),n(*)}$
\sn
\item[(b)]   $\mathbf h$ maps $G_{\alpha_{\delta,i(*),n(*)+1}}$ 
onto itself hence $\beta(*) \in N^\delta_{i(*)} \cap
\alpha_{\delta,i(*),n(*)+1} \backslash \alpha_{\delta,i(*),n(*)}$.
\end{enumerate}
\end{enumerate}
\mn
Also,
\mn
\begin{enumerate}
\item[$(*)_{5.17}$]  if $\beta(*) < \beta_{\delta,i(*),n(*)}  +
\omega$ then $\beta(*) \le \beta_{\delta,i(*),n(*)} +
k_{\delta,i(*),n(*),2}$.
\end{enumerate}
\mn
[Why?  By $(*)_{5.1}$.]

Now,
\mn
\begin{enumerate}
\item[$(*)_{5.18}$]   there is $\beta \in N^\delta_{i_*} \cap
(\beta(*)+1) \backslash \alpha_\delta(n(*)) \backslash S$ such that
$[\beta,\beta(*) + \omega) \cap N^\delta_{i_*}$ is disjoint from
$N^\delta_j$ if $j \in u_*$ but $j \ne i_*$.  
\end{enumerate}
\mn
[Why?  First assume $\beta(*) \notin S$, let $\beta = \beta(*)$, so
clearly $\beta \in N^\delta_{i_*}$ by $(*)_{5.14},\beta \in
(\beta(*)+1)$, also $\beta \notin \alpha_\delta(n(*))$ as by
$(*)_{5.6}$ and the fact that 
$\beta \notin S$ by its choice.  Also $[\beta,\beta(*)
+ \omega) = [\beta(*),\beta(*) + \omega) \subseteq N^\delta_{i_*}$ as
$N$ is closed under $\alpha \mapsto \alpha +1$ by $(*)_{3.1}(b)$.  If
$j \in u_*$ but $j \ne i_*$ then $N^\delta_j \cap
N^\delta_{i_*} \subseteq \alpha_\delta(n(*)) \le \beta$, hence
$[\beta,\beta(*) + \omega) \cap N^\delta_j = \emptyset$, so we are
done.

Second, assume $\beta(*) \in S$, hence cf$(\delta) = \aleph_0$, and by
$(*)_{3.1}(f),\{\alpha_{\beta(*)}(n):n < \omega\} \subseteq
N^\delta_{i_*}$.  But by $\boxplus_0(c)^+$ we have $j \in u_*
\wedge j \ne i_* \Rightarrow \sup(N^\delta_j \cap \beta(*)) <
\beta(*)$.  As $u_*$ is finite there is $\beta \in
\{\alpha_{\beta(*)}(n):n < \omega\}$ such that $(\beta,\beta(*)) \cap
N^\delta_j = \emptyset$; hence as before also $(\beta,\beta(*) +
\omega) \cap N^\delta_j = \emptyset$, whenever $j \in u_* \wedge
j \ne i_*$.  So $(*)_{5.16}$ holds indeed.]

We finish the proof of $\boxplus_5$ by getting a contradiction as follows.
\medskip

\noindent
\underline{Case 1}:  $\beta(*) \ge \beta_{\delta,i(*),n(*)} + \omega$.

So by the choice of $\beta$ and the proof of $(*)_{5.3}$ the type
$\tp_{\bs}(d,G_{\beta(*)+\omega},G_*)$ does not split
over $G_\beta$, and even over some finite subset of it.

Now by $\boxplus_1(e)$ in $G_{\beta(*)+ \omega}$ there is $d' \ne \mathbf
h(b_{\beta_{\delta,i(*),n(*),1}})$ realizing 

\[
\tp_{\bs}(\mathbf h(b_{\beta_{\delta,i(*),n(*),1}}),
G_\beta,G_{\beta(*)+\omega}) \text{ so } 
\mathbf h(b_{\beta_{\delta,i(*),n(*),3}}) \notin c \ell(G_\beta
\cup\{d\}).
\]

\mn
However, $G_* \models d^{-1} \, 
\mathbf h(c_{\beta_{\delta,i(*),n(*),1}})d = 
\mathbf h(c_{\beta_{\delta,i(*),n(*),2}})$, contradiction.
\medskip

\noindent
\underline{Case 2}:  $\beta(*) < \beta_{\delta,i(*),n(*)} + \omega$.

Hence $\beta(*) \le \beta_{\delta,i(*),n(*)} + k_{\delta,i(*),n(*),2}$
and so $\{\tp_{\bs}(d),G_{\beta_{\delta,i(*),n(*)} + \omega},G_*)\}$
does not split over $G_{\beta_{\delta,i(*),n(*)}} \cup
\{b_{\delta,i(*),n(*),3}\})$ but
$\tp(b_{\delta,i(*),n(*),3},G_{\beta_{\delta,i(*),n(*)} +
k_{\delta,i(*),n(*),3}},G_*)$ does not split over
$G_{\beta_{\delta,i(*),n(*)}} \cup \Rang(\bar b_{\delta,i(*),n(*),1})
\cup \Rang(\bar b_{\delta,i(*),n(*),2})$.  

It follows that $\tp_{\bs}(d,G_{\beta_{\delta,i(*),n(*)} +
k_{\delta,i(*),n(*),2}},G_*)$ does not split over
\newline
$G_{\beta_{\delta,i(*),n(*)}} \cup \bar b_{\delta,i(*),n(*),1}$ and recall
  $\mathbf h(\bar b_{\beta_{\delta,i(*),n(*),1}}) \subseteq
  G_{\beta_{\delta,i(*),n(*)}+k_{\delta,i(*),n(*),2}}$,
contradiction by $(*)_{5.0}$.

So we have finished proving $\boxplus_5$.
\medskip

\noindent
\underline{Stage D}:
\mn
\begin{enumerate}
\item[$\boxplus_6$]  
\begin{enumerate}
\item[(a)]  for some stationary $S^*_1 \subseteq S_* 
(\subseteq \lambda \backslash S)$ for every $\beta \in S^*_1 \backslash
\alpha(*)$ if $b \in L^*_\beta$ 
 \then \, $\mathbf h(b) = \sigma^{G_*}(b,\bar a)$
for some $\bar a \in {}^{\omega >}(G_\beta)$ 
and group-term $\sigma(x,\bar y)$
\sn
\item[(b)]   moreover $\mathbf h(b) = \sigma^{G_*}(b)$ if $b \in L^*_\beta$.
\end{enumerate}
\end{enumerate}
\mn
Why?  
\mn
\begin{enumerate}
\item[$(*)_{6.1}$]   clause (a) of $\boxplus_6$ holds even for every
$\beta \in S^*_2 := S_* \cap E_{\mathbf h} \backslash \alpha(*)$.
\end{enumerate}
\mn
[Why?  By $\boxplus_5$.]
\mn
\begin{enumerate}
\item[$(*)_{6.2}$]   \Wilog \, if $\beta \in S^*_2$ and
$b \in L^*_\beta$, \then \, $\mathbf h(b) = \sigma_b(b) a_b$ for some $a_b
\in G_\beta$.
\end{enumerate}
\mn
[Why?  This by $(*)_{6.1}$ because $\mathbf h$ maps $G_\beta$ onto itself,
 $b$ commutes with $G_\beta$
whereas $\bar a_b \in {}^{\omega >}(G_\beta)$.]
\mn
\begin{enumerate}
\item[$(*)_{6.3}$]
\begin{enumerate}
\item[(a)]  $b \mapsto \sigma_b(b)$ is a
homomorphism from the set $L^*_\beta$ into $L^*_\beta$
(but we did not claim $L^*_\beta$ is a subgroup);
\sn
\item[(b)]  $b \mapsto a_b$ induces a homomorphism from
the set $L^*_\beta$ into the group 
$G_\beta$, that is if
$\sigma(x_0,\dotsc,x_{n-1})$ is a group term and $b_0,\dotsc,b_{n-1}
\in L^*_\beta$ and $G_{\beta + \omega} \models \sigma(b_0,\ldots,b_{n-1}) = e$
then $G_\beta \models \sigma(a_{b_0},\ldots,a_{b_{n-1}}) =e$.
\end{enumerate}
\end{enumerate}
\mn
[Why?  As $\mathbf h$ is an automorphism of $G_*$ and as
$a_{b_1},\sigma_{b_2}(b_2)$ commute for $b_1,b_2 \in L^*_\beta$.]

We try to get rid of the homomorphism from $(*)_{6.3}(b)$ in order to
prove $\boxplus_6(b)$.

Toward contradiction assume (for the rest of this stage):
\mn
\begin{enumerate}
\item[$(*)_{6.4}$]  $\gamma \in S^*_2 \subseteq
\lambda \backslash S^*_1$ is a limit ordinal and
$b_* \in L^*_\gamma$ and $a_{b_*} \ne e$.
\end{enumerate}
\mn
Now as $\gamma \in S_* \subseteq \lambda \backslash S$ 
we can find a sequence  
$\bar f^\gamma  = \langle f^\gamma_\eta:\eta \in {}^\omega \mu
\rangle$ satisfying $f^\gamma_\eta$ is a function from $\{\eta \rest n:n <
\omega\}$ into $G_\gamma$ such that for every $f:{}^{\omega >} \mu
\rightarrow G_\gamma$ for some $\eta \in {}^\omega \mu$ we have $f^\gamma_\eta
\subseteq f$; i.e. a simple black box, see \cite[Fact
1.5=L4.5A]{Sh:309}, it exists
as $\mu = \mu^{\aleph_0}$.  Now
generally for $\gamma \in \lambda \backslash S$ let $\cW_\gamma = \{\eta
\in {}^\omega \mu$: for some $c \in G_\gamma$ of order 2 we have $n < \omega
\Rightarrow c^{-1} f^*_\eta(\eta \rest (2n))c = f^*_\eta(\eta \rest (2n+1))\}$.

Let $K_*$ be the group of permutations of $I = {}^{\omega >} \mu \times
\{0,1\}$ with finite support, i.e. $\{f \in \Sym(I):
(\exists^{< \aleph_0}t \in I)(f(t) \ne t)\}$.  For $\eta \in 
{}^{\omega >}\mu$ let $h_\eta \in K_*$
be such that $h_\eta((\eta,\iota)) \equiv (\eta,1-\iota)$, for $\iota =
0,1$, and is the identity otherwise.  Let $K_\gamma$ be the group of 
permutations of $I = {}^{\omega >}\mu \times \{0,1\}$ 
generated by $K_* \cup \{y_\eta:\eta \in {}^{\omega >}\mu\}$, where:
\mn
\begin{enumerate}
\item[$(*)_{6.5}$]  
\begin{enumerate}
\item[(a)]  if $\eta \in \cW_\gamma$ then 
$y_\eta$ interchanges $(\eta \rest (2n+1),\iota),(\eta \rest (2n+2),\iota)$ 
for $n < \omega,\iota = 0,1$ and otherwise is the identity;
\sn
\item[(b)]   if $\eta \in {}^\omega \mu \backslash
\cW_\gamma$ then $y_\eta$ interchanges 
$(\eta \rest (2n),\iota)$ and $(\eta \rest (2n+1),\iota)$
for $n < \omega,\iota=0,1$, and is the identity otherwise.
\end{enumerate}
\end{enumerate}
\mn
Let
\mn
\begin{itemize}
\item    $d$ be the permutation of $I$
interchanging $(<>,0),(<>,1)$ and being the identity otherwise.
\end{itemize}
\mn

Now we shall use some of the amount of freedom left, clearly:
\mn
\begin{enumerate}
\item[$(*)_{6.6}$]
\begin{enumerate}
\item[(a)]   there is $K \subseteq \mathbf C_{G_{\gamma +
\omega}}(G_\beta)$ finite with trivial center such that $b_* \in K$;
\sn
\item[(b)]   there is $\bar b$ which lists the member of $K$
 such that $d_0 = b_*$;
\sn
\item[(c)]   there is $\bar d$, a finite sequence from $K_\gamma$
realizing $\tp(\bar b,\emptyset,G_*)$;
\sn
\item[(d)]   there is $n(*)$ such that $K \subseteq G_{\gamma + n(*)}$.
\end{enumerate}
\sn
\item[$(*)_{6.7}$]   There is an embedding $g_\gamma$ of $K_\gamma$
into $\mathbf C_{G_{\gamma + n(*)+1}}(G_\gamma)$ mapping $\bar d$ to $\bar b$
hence $d_0$ to $b_*$;
\sn
\item[$(*)_{6.8}$]   $b \mapsto a_b$ (for $b \in g_\gamma(K_\gamma))$ 
is a homomorphism from $g_\gamma(K_\gamma)$ into $G_\gamma$;
\sn
\item[$(*)_{6.9}$]   let $f:{}^{\omega >} \lambda \rightarrow G_\beta$
be defined by $f(\eta) = a_{g_\gamma(h_\eta)}$.  
\end{enumerate}
\mn
By the choice of $\langle f_\eta:\eta \in {}^{\omega >}\mu\rangle$
for some $\eta \in {}^\omega \mu$ we have $n < \omega \Rightarrow
f^\gamma_\eta(\eta \rest n) = f(\eta \rest n)$.  

Now does $\eta \in \cW_\gamma$?  First, assume $\eta \notin \cW_\gamma$,
then (by the choice of $K_\gamma$) $(g_\gamma(y_\eta) \in 
G_\gamma$ and) conjugating by $g_\gamma(y_\eta)$ for each $n$, 
interchanges $g_\gamma(h_{\eta \rest
  (2n)}),g_\gamma(h_{\eta \rest (2n+1)})$ which means that in
$K_\gamma$, conjugating by $h_\eta$ interchanges
$f^\gamma_\eta(\eta \rest (2n)),f^\gamma_\eta(\eta \rest (2n+1))$, but
by the choice of $\cW_\gamma$ this means $\eta \in \cW_\gamma$.  

Second, assume $\eta \in \cW_\gamma$,
by the definition of $\cW_\gamma$ there is 
$c \in G_\gamma$ of order 2 such that
conjugating by $c$ for each $n$ interchanges $g_\gamma(h_{\eta \rest
  (2n)}),g_\gamma(h_{\eta \rest (2n+1)})$.  But conjugating by
$g_\gamma(y_\eta)$ for $n$ interchange $g_\gamma(h_{\eta \rest
  (2n+1)}),g_\gamma(h_{\eta \rest (2n+2)})$.
So in $G_*$, the subgroup generated by
$\{c,g_\gamma(y_\eta),g_\gamma(h_{\eta \rest 1})\}$ includes
$g_\gamma(h_{\eta \rest n})$ for $\eta = 1,2,\ldots$; why?  just prove it by
induction on $n$.  But $\{g_\gamma(h_{\eta \rest n}):n = 1,2,\ldots\}
\subseteq G_*$ is infinite, contradiction.
\medskip

\noindent
\underline{Stage E}:
\mn
\begin{enumerate}
\item[$\boxplus_7$]    there is a finite sequence 
$\bar a_*$ such that for every $b \in G_*$ we have
$\mathbf h(b) \in c \ell(\bar a_* \cup\{(b,G_*)\}$.
\end{enumerate}
\mn
[Why?  For $\beta \in S^*_1$ let $d_\beta \in G_{\beta +1}$ realize
$\gs_{\text{cg}}(<>,G_\beta)$ in $G_{\beta +1}$.  
So for every $a \in G_\beta$ of order $m$
as $G_\beta$ is existentially closed there is a finite $K_a \subseteq
G_\beta$ with trivial center to which $a$ belongs.
Hence the element $d_\beta ad^{-1}_\beta$ commute with $G_\beta$
and belongs to $G_{\beta +1}$ and moreover to $L^*_\beta$.  Hence, by
$\boxplus_6(b)$, for some $k(a) < m$ we have:
\mn
\begin{enumerate}
\item[$(*)_{7.1}$]  $\mathbf h(d^{-1}_\beta a d_\beta) = (d^{-1}_\beta
ad_\beta)^{k(a)}$.
\end{enumerate}
\mn
Hence
\mn
\begin{enumerate}
\item[$(*)_{7.2}$]  $\mathbf h(a) = \mathbf h(d^{-1}_\beta) \mathbf
h(d^{-1}_\beta a d_\beta) \mathbf h(d^{-1}_\beta) = \mathbf
h(d^{-1}_\beta)(d^{-1}_\beta a d_\beta)^{k(a)} \mathbf h(d_\beta)$.
\end{enumerate}
\mn
Also, as $\beta \notin S$, there is a finite $K_\beta 
\subseteq G_\beta$ such that $\tp_{\bs}(\langle \mathbf h(d_\beta),
d_\beta \rangle,G_\beta,G_*;\mathbf K_{\lf})$ does not split over 
$K_\beta$.  By $(*)_{7.2},\tp_{\bs}(\mathbf h(a),G_\beta,G_*;\mathbf K_{\lf})$ 
does not split over $K_\beta \cup \{d\}$, but $\mathbf h(a) \in G_\beta$
hence $\mathbf h(a) \in \langle K \cup \{d\} \rangle_{G_*}$.  
By Fodor's lemma this is enough for $\boxplus_7$. 

Clearly we are done by \ref{c73}.
\end{PROOF}
\bigskip

\centerline {$* \qquad * \qquad *$}
\bigskip

\begin{question}
\label{p81}
1) In \ref{p73} we can easily get $2^\lambda$ pairwise non-isomorphic groups
$G'$.  But can they be pairwise far? (i.e. no $G \in \mathbf K_\lambda$,
can be embedded in two of them)?

\noindent
2) Even more basically can 
we demand $G_*$ has no uncountable Abelian subgroup (when
$G$ does not)?  Or at least no Abelian group of cardinality $\lambda$?

\noindent
3) Can we prove \ref{p73} for every $\lambda > \aleph_0$? or at least
   $\lambda \ge \beth_\omega$? 
\end{question}

\begin{discussion}
\label{p83}
1) Concerning \ref{p81}(1), the problem with our approach is using $p
   \in \mathbf S_{\gS}(G)$, so as $\lambda$ is regular we will get
   subgroups generated by indiscernible sequences, but let us elaborate.
Assume $G_* \in \mathbf K_\lambda,G_* = \cup\{G_\alpha:\alpha <
   \lambda\},G_\alpha$ increases with $\alpha$ and $|G_\alpha| <
   \lambda$.  Further, assume $\gs \in \Omega[\mathbf K]$ and $\bar a
   \in {}^{n(\gs)}G_*$ and $S = \{\alpha < \lambda:\bar a \subseteq
   G_\alpha$ and the type $q_{\gs}(\bar a,G_\alpha)$ is realized in
   $G_*\}$ is unbounded in $\lambda$ and thus it is an end segment.  Let
   $\bar c_\alpha \in {}^{k(\gs)}G_*$ realize $q_{\gs}(\bar
   a,G_\alpha)$ and so for some club $E$ of $\lambda,\alpha \in S \cap
   E \Rightarrow \bar c_\alpha \in G_{\min(E \backslash (\alpha
     +1)}$.  Now $\bar{\mathbf c} = \langle \bar c_\alpha:\alpha \in S
   \cap E\rangle$ satisfies: if $h$ is a partial increasing finite
   function from $S \cap E$ to $S \cap E$, then it induces a partial
   automorphism of $G_*:\bar c_\alpha \mapsto \bar c_{h(\alpha)}$.
   This is a case of indiscernible sequences.  Hence the isomorphism type of $c
   \ell(\cup\{\bar c_\alpha:\alpha \in S \cap E\},G_*)$ depends only
   on $\gs$ (and $\tp_{\bs}(\bar a,\emptyset,G_*)$.  Hence the number
   of pairwise far such $G_*$'s is $\le |\gS| + \aleph_0$.

\noindent
2) Concerning \ref{p81}(2), the problem with our approach is that we
use $\gs = \gs_{\ab(k)}$ and more
generally $\gs \in \Omega[\mathbf K]$ such that if $q_{\gs}(\bar a,G) =
 \tp_{\bs}(\bar c,G,H)$ \then \, some $c \in H \backslash G$ 
commute with every (or simply many) members of $G$.  Hence
in the construction above, $G_*$ has Abelian subgroups of 
cardinality $\lambda$.

\noindent
3) What about considering the class of $(G,F_h)_{h \in H},F_h \in
\aut(G),G \in \mathbf K_{\lf},h \mapsto F_h$ a homomorphism?  We
intend to deal with it in \cite{Sh:1098}.
\end{discussion}

\begin{discussion}
\label{p85}
1) Naturally the construction in the proof of \ref{p73} 
is not unique, the class has many
 complicated models.  In the construction in the proof of \ref{p73} we
 choose one where we realize many definable types.

\noindent
2) We may like in $\boxplus_5$ of Stage C in the proof of 
\ref{p73} to consider $c \in
   G_\lambda$, not necessarily from $G_{\beta + \omega}$; (so later
the role of $\gs_{\cg}$ in translating knowledge on $\mathbf h
   \rest G_{\beta + \omega}$ to knowledge on $G_\beta$ + use of Fodor
   is not necessary).  Presently the way we combine $\langle
   b_{\delta,i(\ell),n,3}:n < \omega,\ell < \ell(*)\rangle$ to one
   $n$-type in $\mathbf S_{\bs}(G_\delta)$ works using \ref{c64}.
\end{discussion}

\noindent
Concerning the existence of complete groups in $\mathbf K^{\lf}_\lambda$
extending any $G \in \mathbf K^{\lf}_\lambda$ there are some restrictions.
\begin{claim}
\label{p88}
Assume $\lambda > \cf(\lambda) = \aleph_0,\chi = \lambda^{\aleph_0}$.

\noindent
1) If $G \in \mathbf K^{\lf}_\lambda$ is full, \then \, its outer automorphism
group has cardinality $\ge \chi$.

\noindent
2) $G$ has $\ge \chi$ outer automorphisms \when \, $G \in
\mathbf K^{\lf}_\lambda$ and for some sequence 
$\bar a = \langle a_\alpha:\alpha <
\lambda\rangle$ listing the elements of $G$, letting $G_\alpha = c
\ell(\{a_\beta:\beta < \alpha\},G)$ we have:
\mn
\begin{enumerate}
\item[$(a)$]  for every $\alpha < \lambda$ for $\lambda$ ordinals
  $\beta < \lambda,a_\beta$ commutes with $G_\alpha$
\sn
\item[$(b)$]  for every $a \in G \backslash \{e_G\}$
some element $b \in G,a$ does not commute with $b$.
\end{enumerate}
\mn
3) Like (2) but $G_\alpha$ has center of cardinality $< \lambda$.

\noindent
4) Instead of (a),(b) we can use:
\mn
\begin{enumerate}
\item[$(a)'$]  for every $\alpha < \lambda$ we have $\lambda =
|\{a/\Cent(G):a \in G$ commute with $G_\alpha\}|$.
\end{enumerate}
\end{claim}

\begin{PROOF}{\ref{p88}}
1) We reduce it to part (2).  Let $\bar{\mathbf a} = \langle
a_\alpha:\alpha < \lambda\rangle$ witness fullness (so $\lambda \ge
2^{\aleph_0}$).  Now using the schemes $\gs = s_{\ab(2)}$, the pair
$(G,\bar a)$ satisfies clause (a) of part (2).  Using, e.g. the scheme $\gs
= \gs_{\cg}$ and the claim on non-commuting, \ref{c70}, 
also clause (b) there holds.

\noindent
2) Let $\lambda = \sum\limits_{n} \lambda_n,\lambda_n <
\lambda_{n+1}$.  For each $n$, by clause (a) we have $|S^1_n| =
\lambda$ where $S^1_n := \{\alpha:a_\alpha$ commute with $c
\ell(\{a_\beta:\beta < \lambda_n\},G\})\}$.  Hence for some $k_n > n$ we
have $S^3_n = \{\alpha < \lambda_{k_n}:\alpha \in S^1_n\}$ has
cardinality $> \lambda_n$.

Replacing $\langle \lambda_n:n < \omega\rangle$ by a subsequence
\wilog \, $\bigwedge\limits_{n} k_{2n} = 2n+1$.  Let $\langle
\alpha_{n,i}:i < \lambda_n\rangle$ be a sequence of pairwise distinct
members of $S^3_{2n} \backslash \lambda_{2n}$.  
Now for each $\eta \in \prod\limits_{\ell < n}
\lambda_{2 \ell}$ let $b_\eta = a_{\eta(0)} a_{\eta(1)} \ldots
a_{\eta(n-1)} \in G$ and so $h_\eta := \square_{b_\eta}$, conjugation by 
$b_\eta$, is an inner automorphism of $H$.  
Also $\nu \triangleleft \eta \in \prod\limits_{\ell
  < n} \lambda_{2 \ell} \Rightarrow \square_{b_\eta},\square_{b_\nu}$
agree on $\{a_\beta:\beta < \lambda_{2 \ell g(\nu)}\}$.

Hence if $\eta \in \prod\limits_{n} \lambda_{2n}$ then $\langle
h_{\eta \rest n}:n < \omega\rangle$ converge, i.e. for every $a \in
G$, the sequence $\langle h_{\eta \rest n}(a):n < \omega\rangle$ is
eventually constant and called the eventual value $h_\eta(a)$.

So $h_\eta$ is an automorphism of $G$ (for each $\eta \in
\prod\limits_{n} \lambda_{2n}$).  Now if $\eta_1,\eta_2 \in
\prod\limits_{n} \lambda_{2n},\eta_1(k) \ne \eta_2(k),\eta_1 \rest k =
\eta_2 \rest k$ and for some $\alpha < \lambda_{2k},a_\alpha$ does not
commute with $a_{\eta_1(k)} a^{-1}_{\eta_2(k)}$ then $h_{\eta_1} \ne
h_{\eta_2}$.  Hence we can easily find $2^{\aleph_0}$ pairwise
distinct $h_\eta$'s.  So if $\lambda < 2^{\aleph_0}$ we are done;
otherwise, let $\mu = \min\{\mu:\mu^{\aleph_0} \ge \lambda$
equivalently $\mu^{\aleph_0} = \lambda^{\aleph_0}\}$, so $2^{\aleph_0}
< \mu < \lambda$ and $\alpha < \mu \Rightarrow |\alpha|^{\aleph_0} <
\mu$.

Choose $\bar\mu = \langle \mu_n:n < \omega\rangle$ such that
$\sum\limits_{n} \mu_n = \mu,\mu_n < \mu_{n+1}$; moreover each $\mu_n$
regular and $\alpha < \mu_n \Rightarrow |\alpha|^{\aleph_0} < \mu_n$.
Now for $n<k$ let $E_{n,k} = \{(i,j):i,j < \mu_n$ and the conjugation
$\square_{a_{\alpha_{n,i}}},\square_{a_{\alpha_{n,j}}}$ agree on
$\{a_\beta:\beta < \lambda_{2k}\}\}$, an equivalence relation.  By clause
(b) in the assumption, $\bigcap\limits_{k>n} E_{n,k}$ is the equality
on $\mu_n$, hence for some $k(n) > n,\mu_n/E_{n,k}$ has $\mu_n$
equivalence class.  The rest should be clear.

\noindent
3),4)  Similarly.
\end{PROOF}
\newpage

\section {Other Classes}

Note that
\begin{theorem}
\label{n0}
The results of \S1 holds for any universal class $\mathbf K$ - see
\cite{Sh:300b}.
\end{theorem}

\noindent
However, we cannot in general prove the existence of dense $\gS
\subseteq \Omega[\mathbf K]$, in fact, possibly $\Omega[\mathbf K] = \emptyset$.
We refer the reader to \S0 before \ref{z1}, and to \ref{z8}, 
\ref{c1}.  We may expand an lf group by choosing representations for 
left cosets $\bK$, for $K$ a finite subgroup of $G,b \in G$.  
Then the density of $\Omega[\mathbf K]$ is easy.
\begin{definition}
\label{n7}
1) Let $\mathbf K_{\text{\rm clf}}$ be the class of structures $M$ such
   that $M$ is an expansion of an lf group $G = G_M$ by $F_n =
 F^M_n$ for $n  \ge 1$ such that:
\mn
\begin{enumerate}
\item[$(a)$]  $F^M_n$ is a partial $(n+1)$-place function from $G$ to $G$;
\sn
\item[$(b)$]  if $(a_0,\dotsc,a_n) \in \Dom(F^M_n)$ then
$(a_0,\dotsc,a_{n-1})$ list without repetitions the elements of a
subgroup of $G_M$ and $a_n \in G_M$, of course;
\sn
\item[$(c)$]  if $F^M_n(a_0,\dotsc,a_n) = b$ then $b \in \{a_n
a_\ell:\ell < n\}$;
\sn
\item[$(d)$]  if $K$ is a finite subgroup of $G_M$ with $n$ elements
and for some $(a_0,\dotsc,a_{n-1})$ listing its elements with no
repetitions and $b$ we have $(a_0,\dotsc,a_{n-1},b) \in 
\Dom(F^M_n)$, \then \, for every $(a'_0,\dotsc,a'_{n-1})$
listing the members of $K$ and $b' \in bK \subseteq G_M$  
we have $(a'_0,\dotsc,a'_{n-1},b') \in \Dom(F^M_n)$
and $b'K = bK \Rightarrow F^M_n(a_0,\dotsc,a_{n-1},b') =
F^M_n(a_0,\dotsc,a_{n-1},b)$; 
\sn
\item[$(e)$]  if $K_1,K_2$ are as in clause (d) then also $K_1 \cap
K_2$ is;
\sn
\item[$(f)$]  if $A \subseteq G_M$ is finite then there is a minimal $K$ as in
clause (d) which contains $A$ and if $A$ is empty then $K =
\{e_{G_M}\}$. 
\end{enumerate}
\end{definition}

\begin{definition}
\label{n9}
Let $\mathbf K_{\plf}$ be the class of structures $M$ such
that:  $M$ expands a lf group $G$ by $P^M_n$ for $n < \omega$ and
$F^M_n$ for $n < \omega$ (actually definable from the rest) such that:
\mn
\begin{enumerate}
\item[$(a)$]   $P^M_n$ is an $(n+3)$-place relation;
\sn
\item[$(b)$]  if $\bar a = (a_0,\dotsc,a_{n+2}) \in P^M_n$ then
$\{a_0,\dotsc,a_{n-1}\}$ list with no repetitions the elements of a
finite subgroup of $G_M$;
\sn
\item[$(c)$]  if $\{a_0,\dotsc,a_{n-1}\} = \{a'_0,\dotsc,a'_{n-1}\}$
  are as above and moreover $b,b' \in M$ and
$\{b a_0,\dotsc,b a_{n-1}\} = \{b' a'_0,\dotsc,b' a'_{n-1}\}$
then $M \models ``P_n(a_0,\dotsc,a_{n-1},b,c,d) =
P_n(a'_0,\dotsc,a'_{n-1},b',c,d)"$ for every $c,d \in M$;
\sn
\item[$(d)$]  if $(a_0,\dotsc,a_{n-1})$ list the members of a finite
subgroup $K$ of $G$ with no repetitions and $b \in G$ then
$\{(c,d):(a_0,\dotsc,a_{n-1},b,c,d) \in P^M_n\}$ is a linear order on
the right coset bK, which we denote by $<^M_{K,b}$;
\sn
\item[$(e)$]  if the sequence 
$(a_0,\dotsc,a_{n-1})$ is as above and $b \in G$ then
$F^M_n(a_0,\dotsc,a_{n-1},b)$ is the first element by the order there
in $\{b a_0,\dotsc,b a_{n-1}\}$.
\end{enumerate}
\end{definition}

\begin{definition}
\label{n11}
1) For $M \in \mathbf K_{\clf}$ let $\fsb(M)$ be the set of
finite subgroups $K$ of $G_M$ such that for some $a_0,\dotsc,a_{n-1}$
listing with no repetitions the elements of $K$ and for some $b \in G_M$ we
have $(a_0,\dotsc,a_{n-1},b) \in \Dom(F^M_n)$, i.e. they are as
in clause (d) of Definition \ref{n7}.

\noindent
2) In this case we may write $F^M_K(b) = F^M_n(a_0,\dotsc,a_{n-1},b)$.

\noindent
3) For $M,N \in \mathbf K_{\clf}$ let $M \le_{\text{elf}} N$ or
$M \subseteq N$ mean that $G_M \subseteq G_N$ and $F^M_n = F^N_n
\rest M$ hence $K \in \sfb(N) \wedge K \subseteq M \Rightarrow K \in
\fsb(M)$.  We define similarly $\le_{\plf},\le_{\olf}$, see Definition
\ref{n9}, \ref{z4}.  We
may write $M \le_{\mathbf K} N$ for the appropriate $\mathbf K$, etc.

\noindent
4) ``$M \in \mathbf K_{\clf}$ is (existentially closed)" is
   defined as in \ref{z1}(2).

\noindent
5) Let $\clf$-group mean a member of $K_{\text{\rm clf}}$ and similarly
an $\olf$-group.

\noindent
6) Similarly for ``$\olf$-groups" and ``$\plf$-groups".
\end{definition}

\begin{convention}
\label{n13}
1) Let $\mathbf K$ denote one of the classes defined above, but let it be
   $\mathbf K_{\clf}$ if not said otherwise.
\end{convention}

\begin{dc}
\label{n15}
1) For $M \in \mathbf K_{\olf}$ let $M^{[\clf]}$ be
   the unique $N \in \mathbf K_{\clf}$ such that: $G_N = G_M$
 and $\fsb(N) = \{K:K \subseteq G_M$ is finite$\}$ and $F^M_K(b)$
   is the $<_M$-first member of bK $\subseteq G$ (well defined as bK
   is finite non-empty).

\noindent
1A) For $M \in \mathbf K_{\olf}$ we define $M^{[\plf]}$ 
and for $M \in \mathbf K_{\plf}$ we define $M^{[\clf]}$ parallely.

\noindent
2) For $M \in \mathbf K_{\clf}$ and $A \subseteq M$, there is
   $N \subseteq M$ from $\mathbf K_{\clf}$ with universe $A$
   iff for every finite $A \subseteq B$ there is $K \in \fsb(M)$ 
such that $A \subseteq K \subseteq B$.

\noindent
2A) So if $M \in \mathbf K_{\clf}$ and $K \in {\fsb}(M)$ 
then $M \rest K \in \mathbf K_{\clf}$ and is finite.

\noindent
3) For $A \subseteq M \in \mathbf K$ let $c\ell(A,M)$ be the minimal $N
   \subseteq M$ such that $A \subseteq N$, equivalently $\cup\{K:K \in
\fsb(M)$ and there is no $L \in \fsb(M)$ 
such that $A \cap K \subseteq L \subset K\}$.

\noindent
4) For $A \subseteq M \in \mathbf K$ let $c \ell_{\gr}(A,M)$
   be the closure of $A$ under the group operations.  

\noindent
5) We call $M \in \mathbf K_{\clf}$ full when $\fsb(M)$ is the set
   of finite $K \subseteq G_M$.
\end{dc}

\begin{claim}
\label{n17}
1) The objects in \ref{n15} are well defined (in the right class).

\noindent
2) If $M \in \mathbf K_{\olf}$ or $M \in \mathbf K_{\plf}$
\then \, $M^{[\clf]} \in \mathbf K_{\clf}$ is full.

\noindent
3) $\gS(\mathbf K_{\olf})$ is dense.

\noindent
4) $\gS(\mathbf K_{\clf})$ is dense.
\end{claim}

\begin{PROOF}{\ref{n17}}
1) Straightforward, e.g. in part (3) for $K_{\clf}$ the closure is well
   defined because $\fsb(M)$ is closed under intersections.

\noindent
2) Easy, too.

\noindent
3),4)  As in \S2.
\end{PROOF}

\begin{remark}
\label{n19}
Call $M \in \mathbf K_{\clf}$ invariant \when \, for every finite
$K \subseteq G_M$ there is a function $F^M_K:G \rightarrow G$ such
that $F^M_K(g) \in gK$ and is equal to $F^M_n(a_0,\dotsc,a_{n-1})$
when $a_0,\dotsc,a_{n-1}$ list the members of $K$ with no
repetitions.  Restricting ourselves to such $M$ seems to cause
problems in amalgamations, whereas for $\mathbf K_{\plf}$ this is not so.
\end{remark}

\begin{definition}
\label{n23}
For $M \in \mathbf K$ and $n < \omega$ let $\mathbf S^n_{\ged}(M)$ 
be the set of good $n$-types $p(\bar x) \in 
\mathbf S^n_{\bs}(M)$ which means: $p = \tp(\bar a,M,N)$ where $M 
\subseteq N \in \mathbf K$ and $\bar a \in {}^n N$
   and $c \ell_{\gr}(\bar a + M,N) = c \ell(\bar a +M,N)$.
\end{definition}

\begin{claim}
\label{n25}
The classes $\mathbf K = \mathbf K_{\cfl},\mathbf K_{\plf},\mathbf K_{\olf}$ 
have dense closed $\gS \subseteq \Omega[K]$.
\end{claim}

\begin{PROOF}{\ref{n25}}
Straightforward.
\end{PROOF}
\bigskip

\centerline {$* \qquad * \qquad *$}
\bigskip

\noindent
\begin{definition}
\label{n28}
1) Let $\mathbf K_{\sel}$ be the class of locally finite
semi-groups, i.e. $G$, it has only one operation, binary which
is associative.

\noindent
2) Let $\mathbf K_{\usl}$ be defined similarly with an individual
   constant $e$ such that $G \models g e_G = g = e_G g$ for every $g \in G
   \in \mathbf K_{\usl}$.
\end{definition}
\newpage


\bibliographystyle{amsalpha}
\bibliography{shlhetal}

\end{document}